\documentclass[english,dvipsnames]{amsart}
\usepackage{preamble}
\title{Khovanov algebras for the periplectic Lie superalgebras}

\author{Jonas Nehme}
\address{J.N.: Max Planck Institute for Mathematics, Bonn, Germany}

\date{\today}

\begin{document}
	\begin{abstract}
		The periplectic Lie superalgebra $\lie{p}(n)$ is one of the most mysterious and least understood simple classical Lie superalgebras with reductive even part.
		We approach the study of its finite dimensional representation theory in terms of Schur--Weyl duality.
		We provide an idempotent version of its centralizer, i.e.~the super Brauer algebra.
		We use this to describe explicitly the endomorphism ring of a projective generator for $\lie{p}(n)$ resembling the Khovanov algebra of \cite{BS11a}.
		We also give a diagrammatic description of the translation functors from \cite{BDE19} in terms of certain bimodules and study their effect on projective, standard, costandard and irreducible modules.
		These results will be used to classify irreducible summands in $V^{\otimes d}$, compute $\Ext^1$ between irreducible modules and show that $\lie{p}(n)\mmod$ does \emph{not} admit a Koszul grading.
	\end{abstract}
	\maketitle
	\section*{Introduction}
		Simple Lie superalgebras have been classified by Kac in \cite{Kac77} and fall into three groups: \emph{basic}, \emph{strange} and of \emph{Cartan type}.
		The first two groups are exactly those with reductive even part, thus being the only ones, where classical methods could be applied.
		The basic and strange Lie superalgebras consist out of four infinite families of Lie superalgebras: $\lie{gl}(m|n)$ (the analogue of $\lie{gl}(m)$), $\lie{osp}(r|2n)$ (the simultaneous analogue of $\lie{o}(r)$ and $\lie{sp}(2n)$), the queer Lie superalgebra $\lie{q}(n)$ and $\lie{p}(n)$.
		Except for $\lie{p}(n)$, all families have been widely studied and are fairly well understood.
		We aim at closing this gap for $\lie{p}(n)$.
		For an overview and a general introduction to Lie superalgebras, we refer the reader to \cite{Ser17}.

		The periplectic Lie superalgebra $\lie{p}(n)$ is the Lie supersubalgebra of $\lie{gl}(n|n)$ preserving a non-degenerate odd bilinear form on an $(n|n)$-dimensional vector superspace $V$.
		Fixing a basis for $V$, we can also describe $\lie{p}(n)$ explicitly in terms of matrices
		\begin{equation}\label{pninmatrices}
			\lie{p}(n)=\left\{\left(\begin{array}{c|c}
                A&B\\
                \hline
                C&-A^t
              \end{array}\right)\mid B=B^t,C=-C^t\right\}.
		\end{equation}
		We restrict ourselves to finite dimensional representations of $\lie{p}(n)$ which are integrable with respect to $GL(n)$, the algebraic group corresponding to $\lie{p}(n)_0\cong\lie{gl}(n)$,.
		And from now on a finite dimensional representations means finite dimensional \emph{integrable} representation.
		The main reason for this restriction is that we can endow the category of such representations with a highest weight structure (see \cite{Che15} and \cite{BDE19}).
		Namely, the supergrading on $\lie{p}(n)$ comes from a $\bbZ$-grading $\lie{p}(n)\cong\lie{g}_{-1}\oplus\lie{g}_0\oplus\lie{g}_1$ (where $\lie{g}_{-1}$ are the matrices with $A=B=0$ and $\lie{g}_1$ those with $A=C=0$), which allows to define \emph{thick} and \emph{thin Kac modules} via
		\begin{equation*}
			\Delta(\lambda)=\mathrm{Ind}_{\lie{g}_0\oplus\lie{g}_{-1}}^{\lie{p}(n)}\mathcal{L}^{\lie{g}_0}(\lambda)\qquad\qquad\qquad
          \nabla(\lambda)=\mathrm{Coind}_{\lie{g}_0\oplus\lie{g}_1}^{\lie{p(n)}}\mathcal{L}^{\lie{g}_0}(\lambda),
		\end{equation*}
		where $\mathcal{L}^{\lie{g}_0}(\lambda)$ denotes the irreducible $\lie{g}_0\cong\lie{gl}(n)$-module with ($\rho$-shifted) highest weight $\lambda=(\lambda_1>\lambda_2>\dots>\lambda_n)$ and $\lambda_i\in\bbZ$.
		These are called thick respectively thin as they have different dimensions (cf.~\eqref{pninmatrices}), and these give the standard and costandard modules for the highest weight structure.
		They have a unique irreducible head respectively socle $\mathcal{L}(\lambda)$ with highest weight $\lambda$ and this describes all the irreducible modules up to parity shift.
		Every $\mathcal{L}(\lambda)$ admits a projective cover $\mathcal{P}(\lambda)$ and every projective is also injective, see \cite{BDE19} for more details.

		\underline{Question:} Can we describe the category of finite dimensional representations explicitly, for instance by describing the endomorphism ring of a projective generator?
		
		To tackle this problem, we will study the representation theory via Schur--Weyl duality.
		For this recall from \cite{Moo03}, see also \cite{BDE19}, the super Brauer algebra $\sBr_d$ which provides exactly the counterpart in Schur--Weyl duality for $\lie{p}(n)$.
		In particular, we have a surjective algebra homomorphism 
		\begin{equation}\label{SchurWeylDuality}\tag{SW}
			\sBr_d\twoheadrightarrow\End_{\lie{p}(n)}(V^{\otimes d}),
		\end{equation}
		which is an isomorphism in case $n\gg0$, see e.g.~\cite{Cou18}*{Theorem 8.3.1}.
		The Deligne category of type $P$, denoted by $\Rep(P)$, i.e.~the Karoubian closure of $\sBr$ was introduced in \cite{KT17} and studied e.g.~in \cites{Cou18,CE21}, see also \cite{ES21b} for the perspective of an abelian envelope.

		Denoting by $\Cm$ the Karoubian closure of the category containing $V^{\otimes d}$ and \emph{all} morphisms, Schur--Weyl duality extends to a full monoidal functor 
		\begin{equation}
			\FunctorsBrtopn\colon\Rep(P)\to\Cm.
		\end{equation}
		The functor $\FunctorsBrtopn$ intertwines the functor $\_\otimes V$ with the functor of adding strands in the super Brauer algebra.
		A main insight of \cite{BDE19} was that $\_\otimes V$ can be refined to $\_\otimes V\cong\bigoplus_{i\in\bbZ}\Theta_i$ using a \emph{\enquote{fake} Casimir}. 
		The summand $\Theta_i$ corresponds to a refinement $i\mind$ on the super Brauer side.

		We want to use this refinement to understand $V^{\otimes d}$ and its decomposition into indecomposable summands.
		An abstract classification can be obtained via understanding indecomposable objects of $\Rep(P)$ (using primitive idempotents in $\sBr_d$).
		This was done in \cite{CE21} giving rise to a combinatorial bijection
		\begin{equation}
			\{\substack{\text{indec.~objects}\\\text{in $\Rep(P)$ up to iso.}}\}\overset{1:1}{\leftrightarrow}\{\text{partitions}\}.
		\end{equation}
		Moreover, the non-zero images of these indecomposable objects under $\FunctorsBrtopn$ give a complete classification of indecomposable summands in $V^{\otimes d}$, see \cite{CE21} for an explicit description of the corresponding partitions.
		Unfortunately, these classification results provide no further description of the structure of the indecomposable summands in $V^{\otimes d}$.
		To overcome this issue, we will define an idempotent version of $\Rep(P)$ called $\sGm$, which simplifies computations with indecomposable objects in $\Rep(P)$.
		This category $\sGm$ fits into a more general picture as follows. 
		\begin{equation*}
			\begin{tikzcd}[column sep = large]
				\univAlg\arrow[r, "\text{periplectic}", "\text{relations}"']&\sGalg\arrow[r, "\text{cyclotomic}", "\text{quotient}"']&\sG\arrow[d, "\cong"]\\
				\mathrm{R}&\mathrm{sR}\arrow[r, "\text{level 1 cycl.}", "\text{quotient}"']\arrow[l, rightsquigarrow, "\text{assoc.}"', "\text{graded}"]&\mathrm{sR}^{\mathrm{cycl}}
			\end{tikzcd}.
		\end{equation*}
		In a follow-up paper, we will define a filtered supercategory $\mathrm{sR}$ categorifying the quantum electrical algebra $U_q$ from \cite{BGKT23}. 
		Its associated graded is the KLR-algebra $R$ from \cite{KL09} categorifying the positive half of $U_q(\lie{gl}(n))$, which also describes natural transformations between translation functors for $\lie{gl}(n)$, see \cite{BK08}.
		The here defined category $\sGm$ arises as a level 1 cyclotomic quotient of $\mathrm{sR}$.
		Both, $\sGm$ and $R$, are defined diagrammatically with strands labelled by integers.
		
		A strand labelled by $i$ in $\sGm$ should be thought of as the functor $\Theta_i$ in the decomposition $\_\otimes V=\bigoplus_{i\in\bbZ}\Theta_i$ and diagrams as natural transformations between these functors.
		We can now try to construct an $R$-matrix for the $\Theta_i$ but it turns out that this is parameter dependent, so we do not obtain a braiding but rather a \emph{meromorphic} braiding in the sense of \cite{Soi99}.
		To encapsulate this meromorphic braiding, we will first define a universal $\bbC$-linear meromorphic braided rigid monoidal (super) category $\univAlg$ generated by one object $X$ with $X\otimes X\cong 0$ for $\eps\in\{\pm1\}$.
		Another auxiliary category $\sGalg$ arises from $\univAlg$ by adding two \emph{periplectic relations}, which come from the relations between translation functors in \cite{BDE19}.
		Then, $\sG$ is obtained as a cyclotomic quotient of $\sGalg$.
		
		The whole point of introducing $\sGm$ is to express the refinement of $\_\otimes V\cong\bigoplus_{i\in\bbZ}\Theta_i$ in Schur--Weyl duality, giving rise to the following commutative diagram:
		\begin{equation*}
			\begin{tikzcd}
				\left(\sGm\right)^{\oplus}\arrow[rr, "\FunctorKLRtosBr", "\sim"']\arrow[loop, out=195, in = 165, distance = 2em, "\theta_i", start anchor={[yshift=-1ex]west}, end anchor={[yshift=1ex]west}]\arrow[dr, "\FunctorKLRtopnm"']&&\Rep(P)\arrow[dl, "\FunctorsBrtopn"]\arrow[loop, in = 345, out = 15, distance = 2em, "i\mind", start anchor={[yshift=1ex]east}, end anchor={[yshift=-1ex]east}]\\
				&\Cm\arrow[loop, "\Theta_i", distance = 2em, out = 285, in = 255, start anchor={[xshift=1ex]south}, end anchor={[xshift=-1ex]south}, pos=0.12]
			\end{tikzcd},
		\end{equation*}
		where $\theta_i$ correspond to adding a strand labelled $i$ in $\sGm$.
		In \cref{KLRtosuperBraueriso}, we prove that $\FunctorKLRtosBr$ is an isomorphism of categories.
		On the way, we also show
		\begin{thmintro}\label{introbasissg}
			The category $\sG$ has a basis indexed by pairs of up-down-tableaux of the same shape and this endows $\sG$ with the structure of an upper finite based quasi-hereditary algebra in the sense of \cite{BS21}.
		\end{thmintro}

		If we denote by $\Cdg$ the category of direct sums of direct summands of $V^{\otimes d}$ with parity shifts and \emph{all} morphisms and by $\Ceven$ its restriction to even morphisms, we have the following situation:
		\begin{equation*}
			\begin{tikzcd}[column sep = small]
				&\begin{tabular}{c}
					$\Cdg$\\
					{\small(enriched over $\bbZ/2\bbZ$)}
				\end{tabular}\arrow[loop, out=105, in = 75, distance = 2em, "\Pi", start anchor={[xshift=-1.5ex]north}, end anchor={[xshift=1.5ex]north}]\arrow[dl, "\deg 0\text{ morphisms}"', squiggly]\arrow[rd, "\text{forget $\Pi$-action}", squiggly]\\
				\begin{tabular}{c}
					$\Ceven$\\
					{\small(underlying category)}
				\end{tabular}\arrow[loop, out=195, in = 165, distance = 2em, "\Pi", start anchor={[yshift=-1.5ex,xshift=1.5ex]west}, end anchor={[yshift=1.5ex,xshift=1.5ex]west}]
					&{\displaystyle\simeq}&\begin{tabular}{c}
						$\Cm$\\
						{\small(enriched over $\bbZ/2\bbZ$),}
					\end{tabular}
			\end{tikzcd}
		\end{equation*}
		where the lower two are essentially the same.
		Now, $\Cm$ arises via choosing representatives for the $\Pi$-action in $\Cdg$.
		A different choice, called $\Cp$, will be presented in \cref{action}, and we obtain a full functor
		\begin{equation*}
			\FunctorKLRtopnp\colon\sGp\to\Cp
		\end{equation*}
		that is essentially the same as $\FunctorKLRtopnm$.
		The main purpose for introducing $\Cp$ is that the graded homomorphism spaces of $\Cp$ (and $\sGp$) are all concentrated in even degree, which is more feasible for our purposes.

		Now, $V$ is a projective generator for $\lie{p}(n)\mmod$ (see e.g.~\cite{BDE19}), this means that every indecomposable projective module for $\lie{p}(n)$ appears as a direct summand in $V^{\otimes d}$ (or equivalently in $\Cp$).

		On the other hand, $\sGp$ was built exactly in the way that all these indecomposable projective modules arise as the image of some explicit object in $\sGp$ (and not as an abstract direct summand as for $\Rep(P)$).
		For our endeavor to describe explicitly the endomorphism ring of a projective generator, \cref{introbasissg} can be used to extract a basis of this endomorphism ring.
		Thus, it can be viewed as an analogue of the well-known arc algebras from \cites{BS11,ES16a,ES21} describing the finite dimensional representations of $GL(m|n)$ and $OSp(r|2n)$.
		But the multiplication of the basis elements in $\sGp$ is very complicated and there is no way to describe explicitly the multiplication using these basis elements. 

		Hence, we will introduce yet another combinatorial approach in \cref{Khovanovdefisection}.
		We will define the \emph{Khovanov algebra $\K$ of type $P$}.
		These algebras will have a distinguished basis and an explicit multiplication procedure for these basis elements.
		It is a locally unital locally finite dimensional algebra and the distinguished basis elements look like
		
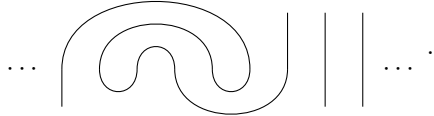
\begin{center}
			\begin{tikzpicture}[scale=0.5]
				\LINE{0}{1}
				\CAP{5}
				\CUP{-1}
				\CAP{-3}
				\CUP{1}
				\CAP{1}
				\CUP{3}
				\LINE{0}{1.5}
				\SETCOORD{1}{0}
				\LINE{0}{-2.5}
				\SETCOORD{1}{0}
				\LINE{0}{2.5}
				\node at (-1, 1) {$\cdots$};
				\node at (9, 1) {$\cdots$};
			\end{tikzpicture}.
		\end{center}
		For $\sGp$, we considered the endofunctor $\theta_i$ by adding an $i$-labelled strand.
		We want to mimic this for $\K$ and we introduce an endofunctor $\hat{\theta}_i$ of $\K\mmod$ given by tensoring with a certain $\K$-$\K$-bimodule $\hat{G}_i$.
		This gives all the ingredients for
		\begin{thmintro}\label{isosgkmod}
			There exists an isomorphism $\Phi$ from $\sGp$ to the full subcategory of $\K\mmod$ containing $\hat{\theta}_{i_k}\dots\hat{\theta}_{i_1}\hat{P}(\overline{\iota})$ intertwining $\hat{\theta}_i$ and the functor of adding a strand labelled $i$ in $\sGp$.
		\end{thmintro}
		The main difference between $\sGp$ and $\K$ is that $\sGp$ is defined via generators and relations whereas $\K$ comes with a distinguished basis and an explicit multiplication rule for this.
		This difference between an explicit basis and generators with relations appears very often in representations theory, most prominently in diagram algebras (like Temperley--Lieb, Brauer algebras and versions thereof).

		For our ultimate goal of an explicit description of the endomorphism ring of a projective generator for $\lie{p}(n)$, we define a quotient $\K_n$ of an idempotent truncation of $\K$ in \cref{kndefisection} (where the idempotents are labeled by dominant integral weights for $\lie{p}(n)$) and show that its finite dimensional representation category is upper finite highest weight.
		In \cref{geombimodknsection} we analyze the endofunctors $\theta_i$ of $\K_n\mmod$ induced by $\hat{\theta}_i$.
		We will show the following.
		\begin{thmintro}
			We have an adjunction $(\theta_i, \theta_{i-1})$.
		\end{thmintro}
		We will also study thoroughly the effect of $\theta_i$ on projective, standard, costandard and irreducible modules in \cref{actionofgeombimodonprojnuclear,geombimodonstandard,geombimodoncostandard,geombimodonirred}.

		In \cref{geombimodknsection}, we will show that $\K_n$ is isomorphic to the endomorphism ring of a projective generator of $\lie{p}(n)\mmod$.
		We construct a projective generator by applying translation functors to the trivial module.
		This gives us two different ways to associate a cup diagram to an indecomposable module, we can either use \cref{defidominttocupdiag} or realize this module as the image of an object in $\sGp$ and use \cref{isosgkmod}.
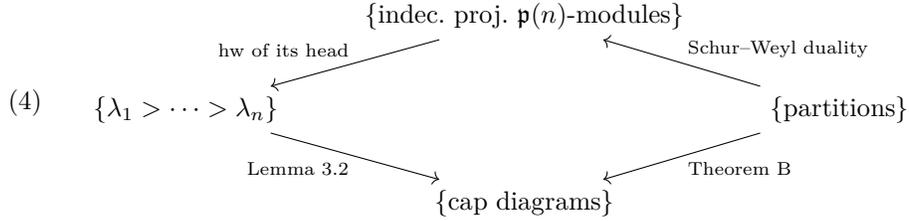
		\begin{equation}\label{combinatorialdiagram}
			\begin{tikzcd}
				&\{\text{indec.~proj.~$\lie{p}(n)$-modules}\}\arrow[dl, "\text{hw of its head}"']\\
				\{\lambda_1>\dots>\lambda_n\}\arrow[dr, "\text{\cref{defidominttocupdiag}}"']&&\{\text{partitions}\}\arrow[ul, "\text{Schur--Weyl duality}"']\arrow[dl, "\text{\cref{isosgkmod}}"]\\
				&\{\text{cap diagrams}\}
			\end{tikzcd}
		\end{equation}
		We will show in \cref{notionsofcupdiagramsagree} that this diagram in fact commutes.
		This gives us all the ingredients for our main theorem.
		\begin{thmintro}[Main theorem]
			There is an equivalence of categories 
				\begin{equation}
					\Psi\colon\K_n\mmod\to\lie{p}(n)\mmod
				\end{equation}
				identifying the highest weight structures and intertwining $\theta_i$ and $\Theta_i$.
		\end{thmintro}
		In particular, we obtain explicit results on the action of (iterated) translation functors on projective, standard, costandard and irreducible modules.
		
		\subsection*{Applications and consequences}
		\begin{itemize}[leftmargin=*]
			\item We show that the dual of an irreducible $\lie{p}(n)$-module can be computed by just rotating the cap diagram associated to its highest weight by $180^\circ$, providing a much easier formula than the combinatorial procedure from \cite{BDE19}.
			\item In \cite{BGS96}, it was shown, that category $\mathcal{O}$ of a semisimple Lie algebra admits a Koszul grading.
			This also holds for $\lie{gl}(m|n)$ by \cite{BS12} and is still conjectural for $\lie{osp}(r|2n)$ (see e.g.~\cite{ES21} or \cite{HNS23} and \cite{Neh21}).
			However, we will show that $\lie{p}(n)\mmod$ does not admit a Koszul grading, we even prove:
			\begin{thmintro}
				There does not exist a non-negative grading on $\K_n$ with semisimple degree $0$ part, that is generated in degree $1$ for $n\geq2$.
				In particular, $\lie{p}(n)\mmod$ does not admit a Koszul grading.
			\end{thmintro}
			\item There exist exactly $n$ irreducible summands of $V^{\otimes d}$, one for each block (except for the one containing $\mathcal{L}(n-2,n-4,\dots,-n)$).
			\item We will give a simple combinatorial criterion to compute extensions between irreducible $\lie{p}(n)$-modules.
			\begin{thmintro}
				The dimension of $\Ext^1$ is given by
				\begin{equation*}
					\dim\Ext^1_{\lie{p}(n)}(\mathcal{L}(\lambda), \mathcal{L}(\mu))=\begin{cases}
						1&\text{if $\underline{\mu}\overline{\lambda}$ satisfies Def.~\ref{deficircdiagprimitive},}\\
						0&\text{otherwise.}
					\end{cases}
				\end{equation*}
			\end{thmintro}
			\noindent We conclude by giving an explicit description of $\lie{p}(1)\mmod$ and $\lie{p}(2)\mmod$ as a quiver with relations, which also explicitly shows that $\lie{p}(2)$ does not admit a Koszul grading.
		\end{itemize}
	\subsection*{Acknowledgements}
	I would like to thank Catharina Stroppel for many fruitful discussions and helpful comments on earlier drafts.
	This project was supported by the Max Planck Institute for Mathematics (IMPRS Moduli Spaces) and
	the Hausdorff Center for Mathematics, which is funded by the Deutsche Forschungsgemeinschaft (DFG, German Research Foundation) under Germany's Excellence
	Strategy (EXC-2047/1 -- 390685813).
	\section{Definition of the idempotent version \texorpdfstring{$\sG$ of $\sBr$}{A of sBr}}\label{defiKLRsection}
	We begin by introducing a category $\univAlg$, which will be the universal monoidal $\bbC$-linear category generated by one object $X$ with $X\otimes X\cong 0$ and its iterated left and right duals with a meromorphic braiding.
	\begin{defi}\label{defiKLR}
		Define the \emph{universal category} $\univAlg$ to be the monoidal $\bbC$-linear supercategory with objects generated by $a\in\bbZ$ and morphisms generated as a monoidal $\bbC$-linear supercategory by 
		\begin{tikzpicture}[line width = \lw, myscale=0.7]
			\node at (0,0) (A){$a+1$}; \node at (1,0) (B) {$a$};\draw (A)..controls +(0.2,-0.8) and +(-0.2,-0.8)..(B);
	\end{tikzpicture}, \begin{tikzpicture}[line width = \lw, myscale=0.7]
	\node at (0,0) (A){$a$}; \node at (1,0) (B) {$a+1$};\draw (A)..controls +(0.2,0.8) and +(-0.2,0.8)..(B);
	\end{tikzpicture} and \begin{tikzpicture}[line width = \lw, myscale=0.7]
	\draw (0,0) -- (1,1) (1,0)--(0,1);\node[fill=white, anchor=north] at (0,0) {$a$}; \node[fill=white, anchor=north] at (1,0) {$b$};\node[fill=white, anchor=south] at (0,1) {$b$}; \node[fill=white, anchor=south] at (1,1) {$a$};
	\end{tikzpicture} for $a\notin\{b,b-1\}$ where the first two have degree $\eps$ and the third one is even, subject to the following relations:

	\begin{tabularx}{\textwidth}{l>{\centering\arraybackslash}Xl>{\centering\arraybackslash}X}
		\refstepcounter{equation}(\theequation)\label{twicezero}&$\begin{tikzpicture}[line width = \lw, myscale=0.7]
			\node at (0,0) (A) {$a$};\node at (1,0) (B) {$a$};
			\node at (0,1.5) (C) {$a$};\node at (1,1.5) (D) {$a$};
			\draw (A.north)--(C.south) (B.north) --(D.south);
		\end{tikzpicture}=0$&
	\refstepcounter{equation}(\theequation)\label{inverse}&$\begin{tikzpicture}[line width = \lw, myscale=0.7]
		\node at (0,0) (A) {$a$};\node at (1,0) (B) {$b$};
		\node at (0,1.5) (C) {$b$};\node at (1,1.5) (D) {$a$};
		\node at (0,3) (E) {$a$};\node at (1,3) (F) {$b$};
		\draw (A.north)--(D.south) (B.north) --(C.south) (D.north)--(E.south) (C.north)--(F.south);
	\end{tikzpicture}=
	\begin{tikzpicture}[line width = \lw, myscale=0.7]
		\node at (0,0) (A) {$a$};\node at (1,0) (B) {$b$};
		\node at (0,1.5) (C) {$a$};\node at (1,1.5) (D) {$b$};
		\draw (A.north)--(C.south) (B.north) --(D.south);
	\end{tikzpicture}$\\
	\refstepcounter{equation}(\theequation)\label{untwist}&
	$\begin{tikzpicture}[line width = \lw, myscale=0.7]
		\node at (0,0) (A) {$a+1$};\node at (1,0) (B) {$a$};
		\node at (0,1.5) (C) {$a$};\node at (1,1.5) (D) {$a+1$};
		\draw (A.north)--(D.south) (B.north) --(C.south);
		\draw (C)..controls +(0.2,0.8) and +(-0.2,0.8)..(D);
	\end{tikzpicture}=0$&
	\refstepcounter{equation}(\theequation)\label{snake}&
	$\begin{tikzpicture}[line width = \lw, myscale=0.7]
		\node at (0,0) (A){$a$};
		\node at (0,1) (B){$a$};
		\node at (1,1) (C){$a+1$};
		\node at (2,1) (D){$a$};
		\node at (2,2) (E){$a$};
		\draw (A.north)--(B.south) (B)..controls +(0.2,0.8) and +(-0.2,0.8)..(C)(C)..controls +(0.2,-0.8) and +(-0.2,-0.8)..(D) (D)--(E);
	\end{tikzpicture}=\begin{tikzpicture}[line width = \lw, myscale=0.7]
	\node at (0,0) (A) {$a$};	\node at (0,2) (B) {$a$};\draw (A) -- (B);
	\end{tikzpicture}=\eps
	\begin{tikzpicture}[line width = \lw, myscale=0.7]
	\node at (2,0) (A){$a$};
	\node at (2,1) (B){$a$};
	\node at (1,1) (C){$a-1$};
	\node at (0,1) (D){$a$};
	\node at (0,2) (E){$a$};
	\draw (A.north)--(B.south) (C)..controls +(0.2,0.8) and +(-0.2,0.8)..(B)(D)..controls +(0.2,-0.8) and +(-0.2,-0.8)..(C) (D)--(E);
	\end{tikzpicture}$\\
	\refstepcounter{equation}(\theequation)\label{tangleone}&
	$\begin{tikzpicture}[line width = \lw, myscale=0.7]	
		\node at (0,2.5) (A) {$a$};
		\node at (0,1.5) (B) {$a$};
		\node at (1,1.5) (C) {$b$};
		\node at (2,1.5) (D) {$b+1$};
		\node at (0,0) (E) {$b$};
		\node at (1,0) (F) {$a$};
		\node at (2,0) (G) {$b+1$};
		\draw (A)--(B) (B.south)--(F.north) (C.south)--(E.north) (G)--(D) (C)..controls +(0.2,0.8) and +(-0.2,0.8)..(D);
	\end{tikzpicture}=
	\begin{tikzpicture}[line width = \lw, myscale=0.7]	
		\node at (2,2.5) (A) {$a$};
		\node at (0,1.5) (B) {$b$};
	\node at (1,1.5) (C) {$b+1$};
		\node at (2,1.5) (D) {$a$};
		\node at (0,0) (E) {$b$};
		\node at (1,0) (F) {$a$};
		\node at (2,0) (G) {$b+1$};
		\draw (A)--(D) (D.south)--(F.north) (C.south)--(G.north) (E)--(B) (B)..controls +(0.2,0.8) and +(-0.2,0.8)..(C);
	\end{tikzpicture}$&
	\refstepcounter{equation}(\theequation)\label{braid}&
	$\begin{tikzpicture}[line width = \lw, myscale=0.7]
		\node at (0,0) (A) {$c$};
		\node at (1,0) (B) {$b$};
		\node at (2,0) (C) {$a$};
		\node at (0,1.5) (D) {$b$};
		\node at (1,1.5) (E) {$c$};
		\node at (2,1.5) (F) {$a$};
		\node at (0,3) (G) {$b$};
		\node at (1,3) (H) {$a$};
		\node at (2,3) (I) {$c$};
		\node at (0,4.5) (J) {$a$};
		\node at (1,4.5) (K) {$b$};
		\node at (2,4.5) (L) {$c$};
		\draw (A.north)--(E.south) (E.north)--(I.south) (I)--(L)(B.north)--(D.south) (D)--(G) (G.north) -- (K.south)(C)--(F) (F.north)--(H.south)(H.north)--(J.south);
	\end{tikzpicture}=\begin{tikzpicture}[line width = \lw, myscale=0.7]
	\node at (0,0) (A) {$c$};
	\node at (1,0) (B) {$b$};
	\node at (2,0) (C) {$a$};
	\node at (0,1.5) (D) {$c$};
	\node at (1,1.5) (E) {$a$};
	\node at (2,1.5) (F) {$b$};
	\node at (0,3) (G) {$a$};
	\node at (1,3) (H) {$c$};
	\node at (2,3) (I) {$b$};
	\node at (0,4.5) (J) {$a$};
	\node at (1,4.5) (K) {$b$};
	\node at (2,4.5) (L) {$c$};
	\draw (C.north)--(E.south) (E.north)--(G.south) (G)--(J)(B.north)--(F.south) (F)--(I) (I.north) -- (K.south)(A)--(D) (D.north)--(H.south)(H.north)--(L.south);
	\end{tikzpicture}$
\end{tabularx}
	\end{defi}
	\begin{rem}
		We can think of \begin{tikzpicture}[line width = \lw, myscale=0.7]
			\draw (0,0) -- (1,1) (1,0)--(0,1);\node[fill=white, anchor=north] at (0,0) {$a$}; \node[fill=white, anchor=north] at (1,0) {$b$};\node[fill=white, anchor=south] at (0,1) {$b$}; \node[fill=white, anchor=south] at (1,1) {$a$};
			\end{tikzpicture} as a meromorphic braiding with poles at $(a,a+1)$ and $(a,a)$, the only difference to the original definition of \cite{Soi99} is that the braiding for $(a,a+1)$ is not an isomorphism.
			
		Furthermore, $\univAlg$ is universal in the following sense.
		Given any object $X$ in a monoidal $\bbC$-linear category $\mathcal{C}$ with $X\otimes X\cong 0$ admitting infinitely many left and right duals with this meromorphic braiding, we obtain a monoidal functor $\univAlg\to\mathcal{C}$, $(0)\mapsto X$.
	\end{rem}
    \begin{defi}
	Define the monoidal $\bbC$-linear supercategory $\sGalg$ as the quotient of $\univAlg$ by the tensor ideal generated by the \emph{periplectic relations}

\begin{tabularx}{\textwidth}{l>{\centering\arraybackslash}Xl>{\centering\arraybackslash}X}
	\refstepcounter{equation}(\theequation)\label{idempotentone}&
	$\begin{tikzpicture}[line width = \lw, myscale=0.7]
		\node at (0,0) (A) {$a$};
		\node at (1,0) (B) {$a+1$};
		\node at (2,0) (C) {$a$};
		\node at (0,1.5) (D) {$a$};
		\node at (1,1.5) (E) {$a-1$};
		\node at (2,1.5) (F) {$a$};
		\node at (0,3) (G) {$a$};
		\node at (1,3) (H) {$a+1$};
		\node at (2,3) (I) {$a$};
		\draw (A)..controls +(0.2,0.8) and +(-0.2,0.8)..(B)(D)..controls +(0.2,-0.8) and +(-0.2,-0.8)..(E)(E)..controls +(0.2,0.8) and +(-0.2,0.8)..(F)(H)..controls +(0.2,-0.8) and +(-0.2,-0.8)..(I) (C)--(F) (D)--(G);
	\end{tikzpicture}=\eps\begin{tikzpicture}[line width = \lw, myscale=0.7]
	\node at (0,0) (A) {$a$};
	\node at (1,0) (B) {$a+1$};
	\node at (2,0) (C) {$a$};
	\node at (0,1.5) (D) {$a$};
	\node at (1,1.5) (E) {$a+1$};
	\node at (2,1.5) (F) {$a$};
	\draw (A)--(D)(B)--(E)(C)--(F);
\end{tikzpicture}$&
\refstepcounter{equation}(\theequation)\label{idempotenttwo}&
	$\begin{tikzpicture}[line width = \lw, myscale=0.7]
		\node at (0,0) (A) {$a$};
		\node at (1,0) (B) {$a-1$};
		\node at (2,0) (C) {$a$};
		\node at (0,1.5) (D) {$a$};
		\node at (1,1.5) (E) {$a+1$};
		\node at (2,1.5) (F) {$a$};
		\node at (0,3) (G) {$a$};
		\node at (1,3) (H) {$a-1$};
		\node at (2,3) (I) {$a$};
		\draw (B)..controls +(0.2,0.8) and +(-0.2,0.8)..(C)(E)..controls +(0.2,-0.8) and +(-0.2,-0.8)..(F)(D)..controls +(0.2,0.8) and +(-0.2,0.8)..(E)(G)..controls +(0.2,-0.8) and +(-0.2,-0.8)..(H) (A)--(D) (F)--(I);
	\end{tikzpicture}=\eps\begin{tikzpicture}[line width = \lw, myscale=0.7]
		\node at (0,0) (A) {$a$};
		\node at (1,0) (B) {$a-1$};
		\node at (2,0) (C) {$a$};
		\node at (0,1.5) (D) {$a$};
		\node at (1,1.5) (E) {$a-1$};
		\node at (2,1.5) (F) {$a$};
		\draw (A)--(D)(B)--(E)(C)--(F);
	\end{tikzpicture}$.
\end{tabularx}
\end{defi}
The next lemma is a direct consequence of \cref{snake} and holds for $\univAlg$ and $\sGalg$.
\begin{lem}\label{adjunctionaap1}
	We have a (super) adjunction $\left(\_\otimes (a+1),\_\otimes a\right)$ and the following implication holds.
	\begin{equation*}
		\begin{tikzpicture}[line width = \lw, baseline = (one),myscale=0.7]
			\draw (0,0) rectangle (2,1) (1,0.5) node {A} (1.8,1)..controls +(0.2,0.8) and +(-0.2,0.8)..(2.8,1)--(2.8,0);
			\node at (3.2,0.5) {=};
			\begin{scope}[yshift=-.8ex]
				\node (one) at (3.2,0.5) {};
			\end{scope}
			\draw (6.4,0) rectangle (4.4,1) (5.4,0.5) node {B} (4.6,1)..controls +(-0.2,0.8) and +(0.2,0.8)..(3.6,1)--(3.6,0);
		\end{tikzpicture}\implies
		\begin{tikzpicture}[line width = \lw, rotate=180, baseline=(one),myscale=0.7]
			\draw (3.6,0) rectangle (5.6,1) (4.6,0.5) node {A} (5.4,1)..controls +(0.2,0.8) and +(-0.2,0.8)..(6.4,1)--(6.4,0);
			\node at (3.2,0.5) {=};
			\begin{scope}[yshift=.8ex]
				\node (one) at (3.2,0.5) {};
			\end{scope}
			\draw (2.8,0) rectangle (0.8,1) (1.8,0.5) node {B} (1,1)..controls +(-0.2,0.8) and +(0.2,0.8)..(0,1)--(0,0);
		\end{tikzpicture}
	\end{equation*}
\end{lem}
\begin{cor} There exist mirror versions of \cref{tangleone} and \cref{untwist}.
	\begin{equation}\label{tangletwo}
		\begin{tikzpicture}[line width = \lw, yscale=-1,myscale=0.7]	
			\node at (2,2.5) (A) {$a$};
			\node at (0,1.5) (B) {$b$};
			\node at (1,1.5) (C) {$b-1$};
			\node at (2,1.5) (D) {$a$};
			\node at (0,0) (E) {$b$};
			\node at (1,0) (F) {$a$};
			\node at (2,0) (G) {$b-1$};
			\draw (A)--(D) (D.north)--(F.south) (C.north)--(G.south) (E)--(B) (B)..controls +(0.2,0.8) and +(-0.2,0.8)..(C);
		\end{tikzpicture}
		=
		\begin{tikzpicture}[line width = \lw, yscale=-1,myscale=0.7]	
			\node at (0,2.5) (A) {$a$};
			\node at (0,1.5) (B) {$a$};
			\node at (1,1.5) (C) {$b$};
			\node at (2,1.5) (D) {$b-1$};
			\node at (0,0) (E) {$b$};
			\node at (1,0) (F) {$a$};
			\node at (2,0) (G) {$b-1$};
			\draw (A)--(B) (B.north)--(F.south) (C.north)--(E.south) (G)--(D) (C)..controls +(0.2,0.8) and +(-0.2,0.8)..(D);
		\end{tikzpicture}\qquad\qquad
        \begin{tikzpicture}[line width = \lw,myscale=0.7]
            \node at (0,0) (A) {$a+1$};\node at (1,0) (B) {$a$};
            \node at (0,1.5) (C) {$a$};\node at (1,1.5) (D) {$a+1$};
            \draw (A.north)--(D.south) (B.north) --(C.south);
            \draw (A)..controls +(0.2,-0.8) and +(-0.2,-0.8)..(B);
        \end{tikzpicture}=0
	\end{equation}
\end{cor}
\begin{proof}
	This first one is immediate from \cref{adjunctionaap1}.
    In combination with \cref{tangleone} it implies $\begin{tikzpicture}[line width = \lw,myscale=0.58]
		\node at (0,0) (A) {$a+1$};\node at (1,0) (B) {$a$};\node at (2,0) (E) {$a-1$};
		\node at (0,1.5) (C) {$a$};\node at (1,1.5) (D) {$a+1$};\node at (2,1.5) (F) {$a-1$};
		\draw (A.north)--(D.south) (B.north) --(C.south);
		\draw (C)..controls +(0.2,0.8) and +(-0.2,0.8)..(D);
		\draw (B)..controls +(0.2,-0.8) and +(-0.2,-0.8)..(E)--(F);
	\end{tikzpicture}=\begin{tikzpicture}[line width = \lw,myscale=0.58]
		\node at (0,0) (A) {$a$};\node at (1,0) (B) {$a-1$};\node at (2,0) (E) {$a+1$};
		\node at (0,1.5) (C) {$a-1$};\node at (1,1.5) (D) {$a$};\node at (2,1.5) (F) {$a+1$};
		\draw (A.north)--(D.south) (B.north) --(C.south);
		\draw (D)..controls +(0.2,0.8) and +(-0.2,0.8)..(F)--(E);
		\draw (A)..controls +(0.2,-0.8) and +(-0.2,-0.8)..(B);
	\end{tikzpicture}$ and the claim follows from \cref{untwist,snake}.
\end{proof}
The next lemma is a technical result which will be used to describe gradings on $\univAlg$ and $\sGalg$.
\begin{defi}
	Let $\sgn\colon\bbZ\setminus\{0\}\to\bbZ\setminus\{0\}$, $x\mapsto\frac{x}{\abs{x}}$.
	A map $f\colon\bbZ\times\bbZ\setminus\{(a,a),(a,a+1)\mid a\in\bbZ\}\to\bbZ$ is a \emph{grading function} if there exists $k\in\bbZ$ such that $f(a,b)=\sgn(b-a)(-1)^{b+a}k$ whenever $b\neq a,a+1$.
\end{defi}
\begin{lem}\label{technicalhelpergrading}
	A map $f\colon\bbZ\times\bbZ\setminus\{(a,a),(a,a+1)\mid a\in\bbZ\}\to\bbZ$ is a grading function if and only if
	\begin{enumerate}
		\item\label{one} $f(a,b)=-f(b,a)$ for all $\abs{a-b}>1$ and
		\item \label{two} $f(a,b)=f(b,a+1)$.
	\end{enumerate}
\end{lem}
\begin{proof}
	We clearly have $\sgn(b-a)(-1)^{a+b}k = -\sgn(a-b)(-1)^{b+a}k$ and $\sgn(b-a)(-1)^{a+b}k = \sgn(b-a-1)(-1)^{b+a+1}k$, and thus any grading function satisfies \cref{one} and \cref{two}.

	On the other hand if $f$ satisfies \cref{one} and \cref{two},$f(a,a+2)=f(a+2,a+1)=f(a+1,a+3)$ by \cref{two}
	and thus $f(a,a+2)=k=\sgn{a+2-a}(-1)^{a+a+2}k$ for some fixed $k$.
	Furthermore, for any $s\geq 2$ using $f(a,a+s+1)=f(a+s,a)=-f(a,a+s)$ we obtain $f(a,a+s)=(-1)^s k=\sgn{a+s-a}(-1)^{a+a+s}k$. 
	So far we have shown
	\begin{equation}\label{helper}
		f(a,b)=\sgn(b-a)(-1)^{a+b}k
	\end{equation}
	in case $b>a+1$.
    Using \cref{one} and that $\sgn(b-a)=-\sgn(a-b)$, we also conclude \cref{helper} in case $a-1>b$.

	Finally, for $b+1=a$ we have $f(a,a-1)=f(a-1,a+1)=k=\sgn(a-1-a)(-1)^{a-1+a}k$.
\end{proof}
The following is straightforward to check.
\begin{lem}
	Given a grading function $f$ and for every $a\in\bbZ$ an integer $l_a$, we can endow the supercategory $\sGalg$ with a grading by putting \begin{tikzpicture}[line width = \lw,myscale=0.7]
		\node at (0,0) (A){$a+1$}; \node at (1,0) (B) {$a$};\draw (A)..controls +(0.2,-0.8) and +(-0.2,-0.8)..(B);
	\end{tikzpicture} in degree $l_a$, \begin{tikzpicture}[line width = \lw,myscale=0.7]
		\node at (0,0) (A){$a$}; \node at (1,0) (B) {$a+1$};\draw (A)..controls +(0.2,0.8) and +(-0.2,0.8)..(B);
	\end{tikzpicture} in degree $-l_a$ and \begin{tikzpicture}[line width = \lw,myscale=0.7]
		\draw (0,0) -- (1,1) (1,0)--(0,1);\node[fill=white, anchor=north] at (0,0) {$a$}; \node[fill=white, anchor=north] at (1,0) {$b$};\node[fill=white, anchor=south] at (0,1) {$b$}; \node[fill=white, anchor=south] at (1,1) {$a$};
	\end{tikzpicture} in degree $f(a,b)$.
\end{lem}
\begin{defi}
	We define the supercategory $\sG$ as the quotient of $\sGalg$ by the additional relation.
	\begin{equation}\label{firstzero}
		\begin{tikzpicture}[line width = \lw, myscale=0.7]
			\node at (0,0) (A) {$a$};\node at (0,1.5) (C) {$a$};
			\draw (1,0) -- (1,0.25) (2.5,0) -- (2.5,0.25) (1,1.25) -- (1,1.5) (2.5,1.25)--(2.5,1.5) (0.75,0.25)--(2.75,0.25)--(2.75,1.25)--(0.75,1.25)--(0.75,0.25);
			\draw (A.north)--(C.south);
			\node at (1.75,0) {$\dots$};
			\node at (1.75,1.5) {$\dots$};
			\node at (1.75, 0.75) {something};
		\end{tikzpicture}=0\quad \text{if }a\neq0
	\end{equation}
	As \cref{firstzero} is homogeneous, $\sG$ is then a graded right module category over $\sGalg$, see e.g.~\cite{EGNO15} for an introduction to module categories.

	For $i\in\bbZ$, we write $\theta_i\colon\sG\to\sG$ for the image of $i$ under the action functor $\sGalg\to\End(\sG)$ coming from the $\sGalg$-module structure on $\sG$.
\end{defi}

\begin{rem}\label{dualityonKLR}
	We can endow all three categories $\univAlg$, $\sGalg$ and $\sG$ with a duality given by $a\mapsto -a$ on objects and horizontally flipping the diagrams describing morphisms.
\end{rem}
\subsection{Relation to the super Brauer category}\label{relationSuperBrauer}

This subsection provides an isomorphism between $\Rep(P)$ and $\sGm$.

The algebra $\sBr_d$ has a family of pairwise commuting Jucys--Murphy elements, \cite{Cou18a}*{Lemma 6.1.2}, defined as
\begin{equation}\label{defiJM}
    x_1=0\qquad\qquad x_{k+1}=s_kx_ks_k+s_k+e_k.
\end{equation}
Thus, there exists a system of orthogonal idempotents $e(i_1, \dots, i_d)$ in $\sBr_d$ projecting onto the generalized $i_k$-eigenspaces of $x_k$ for $1\leq k\leq d$, i.e.~for $N\gg0$ we have $(x_k-i_k)^Ne(i_1, \dots, i_d)=0$.

The following  is a remarkable difference to the ordinary Brauer algebra, where Jordan blocks of size $2$ can occur.
\begin{prop}\label{JMdiag}
    The $x_k$ act diagonalizably on $\sBr_d$.
\end{prop}
\begin{proof}
    First, using \cref{SchurWeylDuality}, $x_k-i$ acts nilpotently on any indecomposable $\lie{p}(n)$-module (for some $i$).
    By \cite{BDE19}*{Prop. 8.1.1} the endomorphism space of any indecomposable module is one dimensional and hence $x_k-i$ acts by $0$.
    As \cref{SchurWeylDuality} is an isomorphism for $n\gg0$, the $x_k$ act diagonalizably on $\sBr_d$.
\end{proof}
\begin{defi}\label{defiiindandthetaim}
    Let $\ind$ be the endofunctor on $\Rep(P)$ given by tensoring with the object $1$ (i.e.\ adding one strand to the right in every diagram).
    Using the Jucys--Murphy elements $x_k$, this decomposes into $\ind=\bigoplus_{i\in\bbZ}i\mind$ by projecting onto the $i$-eigenspace for this strand (see also \cite{ES21}*{Lemma 2.15} for a similar statement for the Brauer category).

	Similarly, we can refine $\Theta=\_\otimes V = \bigoplus_{i\in\bbZ}\Theta_i^{-}$ by projecting onto eigenspaces and $\FunctorsBrtopn$ intertwines $i\mind$ and $\Theta_i^{-}$ (see \cite{BDE19} for details). 
\end{defi}
Fix now a choice of scalars $\lambda_{a,b}$ for $a$,$b\in\bbZ$, $b\neq a,a-1$ such that 
\begin{enumerate}
    \item\label{scalarone} $\lambda_{a,b}\lambda_{b,a}=\frac{1}{1-(a-b)^2}$ for all $a,b\in\bbZ$ with $\abs{a-b}>1$ and
    \item\label{scalartwo} $\lambda_{b,a}(b-a)=\lambda_{a,b+1}(a-b-1)$ for all $a, b\in\bbZ$ such that $a\neq b, b+1$.
\end{enumerate}
The following has a straightforward proof given in  \cref{checkingrelationsklrtosbr}.
\begin{prop}\label{functorXiexplicit}
    The category $\Rep(P)^{op}$ is a module over $\sGalgm$ via $\ActionFunctorKLRtosBr\colon\sGalgm\to\End(\Rep(P)^{op})$ given by $i\mapsto i\mind$ on objects and
    \begin{itemize}
		\item $\begin{tikzpicture}[line width = \lw, myscale=0.7]
            \node at (0,0) (A){$a+1$}; \node at (1,0) (B) {$a$};\draw (A)..controls +(0.2,-0.8) and +(-0.2,-0.8)..(B);
       \end{tikzpicture}\mapsto \left((a+1)\mind\circ a\mind\overset{\incl}{\longrightarrow}\ind\circ\ind\overset{-\cap}{\longrightarrow}\id\right)$,
       \item $\begin{tikzpicture}[line width = \lw, myscale=0.7]
       \node at (0,0) (A){$a$}; \node at (1,0) (B) {$a+1$};\draw (A)..controls +(0.2,0.8) and +(-0.2,0.8)..(B);
   \end{tikzpicture}\mapsto \left(\id\overset{\cup}{\longrightarrow}\ind\circ\ind\overset{\pr}{\longrightarrow}a\mind\circ (a+1)\mind\right)$,
   \item $\begin{tikzpicture}[line width = \lw, myscale=0.7]
   \draw (0,0) -- (1,1) (1,0)--(0,1);\node[fill=white, anchor=north] at (0,0) {$a$}; \node[fill=white, anchor=north] at (1,0) {$b$};\node[fill=white, anchor=south] at (0,1) {$b$}; \node[fill=white, anchor=south] at (1,1) {$a$};
   \end{tikzpicture}\mapsto\left(b\mind\circ a\mind\overset{\incl}{\longrightarrow}\ind\circ\ind\overset{f}{\longrightarrow}\ind\circ\ind\right)$
    \end{itemize}
    on morphisms, where $f=\lambda_{b,a}((b-a)(\id\otimes s)+\id)$.
\end{prop}
\begin{rem}
	A priori, \begin{tikzpicture}[line width = \lw, myscale=0.7]
		\draw (0,0) -- (1,1) (1,0)--(0,1);\node[fill=white, anchor=north] at (0,0) {$a$}; \node[fill=white, anchor=north] at (1,0) {$b$};\node[fill=white, anchor=south] at (0,1) {$b$}; \node[fill=white, anchor=south] at (1,1) {$a$};
		\end{tikzpicture} is only assigned to a natural transformation $b\mind\circ a\mind\to\ind\circ\ind$, but in \cite{BDE19}*{Theorem 4.5.1} it was shown that the image of this morphism lies indeed in $a\mind\circ b\mind$.
		Therefore, this assignment is in fact well-defined.
\end{rem}
\begin{thm}\label{KLRtosuperBraueriso}
	Evaluating $\ActionFunctorKLRtosBr$ at the object $0\in\Rep(P)$ factors through $\sGm$ and provides a fully faithful homomorphism $\FunctorKLRtosBr\colon\sGm\to\Rep(P)$ of module categories over $\sGalgm$ with essential image indecomposable objects in $\Rep(P)$.
\end{thm}
\begin{proof}
    The evaluation at $0$ factors through $\sGm$ as $x_1=0$ by definition.
    Hence, we get a homomorphism $\FunctorKLRtosBr\colon\sGm\to\Rep(P)$ of module categories over $\sGalgm$.

	The essential image is given by $\FunctorKLRtosBr(\bm{i})$ for all sequences $\bm{i}$.
	These are exactly the objects in $\Rep(P)$ corresponding to the idempotents $e(\bm{i})$.
	This idempotent is obtained from repeatedly applying $i\mind$.
	On the other side of Schur--Weyl duality this corresponds to repeatedly applying $\Theta_i^{-}$ to the trivial module $\bbC$.
	By \cite{BDE19}*{Theorem 7.1.1} this is either indecomposable or $0$.
	As \cref{SchurWeylDuality} is an isomorphism for $n\gg0$, we see that $e(\bm{i})$ is either primitive or $0$.
	Hence, $1=\sum_{\bm{i}\in\bbZ^d}e(\bm{i})$ is a decomposition of $1$ into primitive orthogonal idempotents, and thus every indecomposable object in $\Rep(P)$ arises via some $e(\bm{i})$.

	\underline{Claim:} $\FunctorKLRtosBr$ is full.

	\Cref{functorXiexplicit} implies that all cups and caps are in the image of $\FunctorKLRtosBr$ as well as $e(i_1, \dots i_{k-1}, i_k, i_{k+1}, i_{k+2}, \dots i_d)s_k$ for $i_{k+1}\neq i_k-1,i_k$.

	We have $\Theta_i^{-}\Theta_i^{-}=0$ by \cite{BDE19}*{Theorem 4.5.1}, so $e(i_1, \dots i_{k}, i_{k+1}, \dots i_d)=0$ if $i_k=i_{k+1}$ as \cref{SchurWeylDuality} is an isomorphism for $n\gg0$.

	With the abbreviation $e_{j,l}\coloneqq e(i_1, \dots i_{k-1}, j, l, i_{k+2}, \dots i_d)$ for $j$ and $l\in\bbZ$, it suffices to show that $e_{i,i+1}s_ke_{j,l}$ is in the image for all $i$, $j$, $l\in\bbZ$.
	
	For $l\neq i$, we obtain from \cref{defiJM} $e_{i,i+1}s_ke_{j,l}=e_{i,i+1}s_k\frac{x_{k+1}-i}{l-i}e_{j,l}=\frac{1}{l-i}e_{i,i+1}\big((x_k-i)s_k+1+\eps_k\big)e_{j,l}=\frac{1}{l-i}e_{i,i+1}\big(1+\eps_k\big)e_{j,l}\in\im\FunctorKLRtosBr$.
	Similar for $j\neq i+1$.

	The case $(j,l)=(i+1,i)$ is more involved.
	In the category of finite dimensional representations of $\lie{p}(n)$ we have by the adjunction $(\Theta_i,\Theta_{i-1})$ from \cite{BDE19}
	\begin{equation*}
		\Hom_{\lie{p(n)}}(\Theta_{i+1}\Theta_{i}M, \Theta_{i}\Theta_{i+1}M)=\Hom_{\lie{p(n)}}(\Theta_{i+1}\Theta_{i+1}\Theta_{i}M, \Theta_{i+1}M)=\{0\},
	\end{equation*}
	since $\Theta_k\Theta_k=0$ by \cite{BDE19}*{Theorem 4.5.1}.
	Similar to before, \cref{SchurWeylDuality} is an isomorphism for $n\gg0$, and thus $e_{i,i+1}s_ke_{i+1,i}=0$.
	Thus, the functor $\FunctorKLRtosBr$ is full.
	Faithfulness of $\FunctorKLRtosBr$ follows from \cref{spanningsetisbasis}.
\end{proof}

\subsection{\texorpdfstring{Action on $V^{\otimes d}$ for the periplectic Lie superalgebra $\lie{p}(n)$}{Action on V\textasciicircum d for the periplectic Lie superalgebra p(n)}}\label{action}

As $\FunctorsBrtopn$ intertwines $i\mind$ and $\Theta_i^{-}$, we get the structure of a module category over $\sGalgm$ on $\Cm$, and we obtain a homomorphism $\FunctorKLRtopnm\coloneqq\FunctorsBrtopn\circ\FunctorKLRtosBr\colon\sGm\to\Cm$ of module categories over $\sGalgm$.

The main disadvantage of $\Cm$ is that it contains odd morphisms.
But as we consider all morphisms, $\Cm$ depends on a choice of representatives under the parity switch.
We chose this representative such that it appears as a direct summand of $V^{\otimes d}$.
We proceed by fixing a new choice of representatives to obtain the category $\Cp$ and a functor $\FunctorKLRtopnp\colon\sGp\to\Cp$ where miraculously only even morphisms are left.

For this define the endofunctor $\Theta_k^{+}\coloneqq\Pi^k\Theta_k^{-}$ and the category $\Cp$ as the full additive subcategory of $\Rep(\lie{p}(n))$ generated by $\Theta_{i_d}^{+}\cdots\Theta_{i_1}^{+}\bbC$.
\begin{thm}\label{KLRtopnp}
    The category $\Cp$ is a module category over $\sGalgp$ via (the dotted diagrams correspond to the action of $\sGalgm$ on $\Cm$)
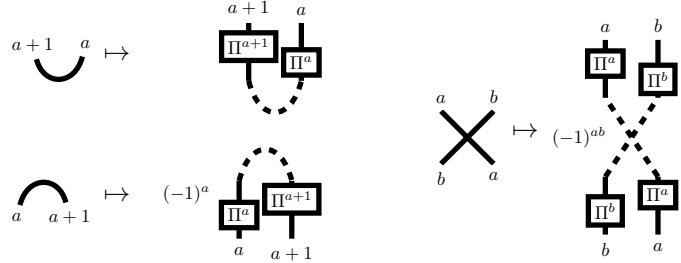
    \begin{equation*}\begin{aligned}
       &\begin{tikzpicture}[line width = \lw, myscale=0.7]
            \node at (0,0) (A){$a+1$}; \node at (1,0) (B) {$a$};\draw (A)..controls +(0.2,-0.8) and +(-0.2,-0.8)..(B);
        \end{tikzpicture}\mapsto&
        \begin{tikzpicture}[line width = \lw, myscale=0.7]	
            \draw (0,0.1) node[anchor=south] {$a+1$} -- (0,-0.35) node[rectangle, fill=white, draw] {$\Pi^{a+1}$} -- (0,-1)(1,-1) -- (1,-0.65) node[rectangle, fill=white, draw] {$\Pi^{a}$} --(1,0.1) node[anchor=south] {$a$};
            \draw[dashed] (0,-1)..controls +(0.2,-0.8) and +(-0.2,-0.8)..(1,-1);
        \end{tikzpicture}\\
        &\begin{tikzpicture}[line width = \lw, myscale=0.7]
            \node at (0,0) (A){$a$}; \node at (1,0) (B) {$a+1$};\draw (A)..controls +(0.2,0.8) and +(-0.2,0.8)..(B);
        \end{tikzpicture}\mapsto&
        \begin{tikzpicture}[line width = \lw, myscale=0.7]
            \draw (0,-0.1) node[anchor=north] {$a$}--(0,0.35) node[rectangle, fill=white, draw] {$\Pi^{a}$} -- (0,1)(1,1) -- (1,0.65) node[rectangle, fill=white, draw] {$\Pi^{a+1}$} --(1,-0.1) node[anchor=north] {$a+1$};
            \draw[dashed] (0,1)..controls +(0.2,0.8) and +(-0.2,0.8)..(1,1);
            \node at (-1,0.75) {$(-1)^{a}$};
        \end{tikzpicture}
    \end{aligned}\qquad\qquad
        \begin{tikzpicture}[line width = \lw, myscale=0.7]
            \draw (0,0) -- (1,1) (1,0)--(0,1);\node[fill=white, anchor=north] at (0,0) {$b$}; \node[fill=white, anchor=north] at (1,0) {$a$};\node[fill=white, anchor=south] at (0,1) {$a$}; \node[fill=white, anchor=south] at (1,1) {$b$};
        \end{tikzpicture}\mapsto
        \begin{tikzpicture}[line width = \lw, myscale=0.7]
            \draw (0,-0.1) node[anchor=north] {$b$} -- (0,0.35) node[rectangle, fill=white, draw] {$\Pi^b$}--(0,1)(1,2.5)--(1,2.85) node[rectangle, fill=white, draw] {$\Pi^b$} -- (1,3.6) node[anchor=south]{$b$};
            \draw (1,-0.1) node[anchor=north] {$a$} -- (1,0.65) node[rectangle, fill=white, draw] {$\Pi^a$}--(1,1)(0,2.5)--(0,3.15) node[rectangle, fill=white, draw] {$\Pi^a$} -- (0,3.6) node[anchor=south]{$a$};
            \node at (-0.5,1.75) {$(-1)^{ab}$};
            \draw[dashed] (0,1)--(1,2.5)(1,1)--(0,2.5);
        \end{tikzpicture}
    \end{equation*}
    if $b\neq a$,$a-1$ and $0$ otherwise.

    Furthermore, evaluation at $\bbC$ factors through $\sGp$, and we obtain the homomorphism $\FunctorKLRtopnp\colon\sGp\to\Cp$ of module categories over $\sGalgm$.
\end{thm}
\begin{proof}
    The following equalities regarding the parity shift hold (all follow from the Koszul sign rule for height moves in supercategories):
    
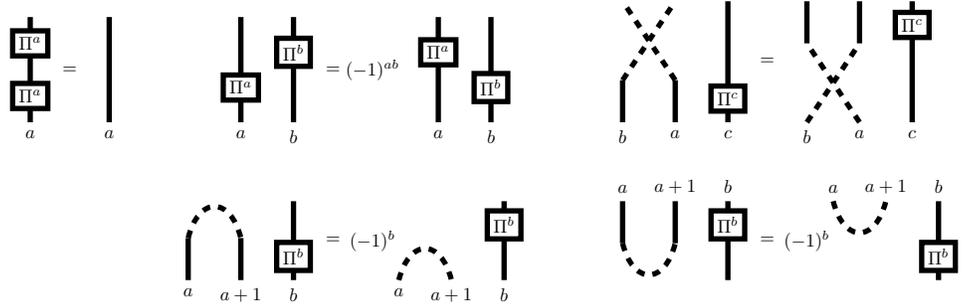
\begin{equation*}
        \begin{tikzpicture}[line width = \lw, myscale=0.7]
            \draw (0,0) node[anchor=north] {$a$} -- (0,0.5) node[rectangle, draw, fill=white] {$\Pi^a$} -- (0,1)-- (0,1.5) node[rectangle, draw, fill=white] {$\Pi^a$} -- (0,2);
            \node at (0.75,1) {$=$};
            \draw (1.5,0) node[anchor=north] {$a$}  -- (1.5,2);
            \begin{scope}[xshift=4cm]
                \draw (0,0) node[anchor=north] {$a$} -- (0,0.66) node[rectangle, draw, fill=white] {$\Pi^a$} -- (0,2);
                \draw (1,0) node[anchor=north] {$b$} -- (1,1.33) node[rectangle, draw, fill=white] {$\Pi^b$} -- (1,2);
                \node at (1.75, 1) {$=$};
                \node at (2.5, 1) {$(-1)^{ab}$};
                \draw (3.75,0) node[anchor=north] {$a$} -- (3.75,1.33) node[rectangle, draw, fill=white] {$\Pi^a$} -- (3.75,2);
                \draw (4.75,0) node[anchor=north] {$b$} -- (4.75,0.66) node[rectangle, draw, fill=white] {$\Pi^b$} -- (4.75,2);
            \end{scope}
            \begin{scope}[xshift=11.25cm]
                \draw (0,0)node[anchor=north] {$b$}--(0,0.8);\draw[dashed](0,0.8)--(1,2.3);
                \draw (1,0)node[anchor=north] {$a$}--(1,0.8);\draw[dashed](1,0.8)--(0,2.3);
                \draw (2,0) node[anchor=north] {$c$} -- (2,0.45) node[rectangle, fill=white, draw] {$\Pi^{c}$} -- (2,2.3);
                \node at (2.75,1.15) {=};
                \draw[dashed] (3.5,0)node[anchor=north] {$b$}--(4.5,1.5);\draw (4.5,1.5)--(4.5,2.3);
                \draw[dashed] (4.5,0)node[anchor=north] {$a$}--(3.5,1.5);\draw (3.5,1.5)--(3.5,2.3);
                \draw (5.5,0) node[anchor=north] {$c$} -- (5.5,1.85) node[rectangle, fill=white, draw] {$\Pi^{c}$} -- (5.5,2.3);
            \end{scope}
            \begin{scope}[yshift=-3cm, xshift=3cm]
                \draw (0,0)node[anchor=north] {$a$}--(0,0.8);\draw[dashed](0,0.8)..controls +(0.2,0.8) and +(-0.2,0.8)..(1,0.8);\draw (1,0.8)--(1,0)node[anchor=north] {$a+1$};
                \draw (2,0) node[anchor=north] {$b$} -- (2,0.45) node[rectangle, fill=white, draw] {$\Pi^{b}$} -- (2,1.5);
                \node at (2.75,0.75) {=};
                \draw[dashed] (4,0)node[anchor=north] {$a$}..controls +(0.2,0.8) and +(-0.2,0.8)..(5,0)node[anchor=north] {$a+1$};
                \draw (6,0) node[anchor=north] {$b$} -- (6,1.05) node[rectangle, fill=white, draw] {$\Pi^{b}$} -- (6,1.5);
                \node at (3.5, 0.75) {$(-1)^b$};
                \begin{scope}[xshift=8.25cm, yshift=1.5cm]
                    \draw (0,0)node[anchor=south] {$a$}--(0,-0.8);\draw[dashed](0,-0.8)..controls +(0.2,-0.8) and +(-0.2,-0.8)..(1,-0.8);\draw(1,-0.8)--(1,0)node[anchor=south] {$a+1$};
                    \draw (2,0) node[anchor=south] {$b$} -- (2,-0.45) node[rectangle, fill=white, draw] {$\Pi^{b}$} -- (2,-1.5);
                    \node at (2.75,-0.75) {=};
                    \draw[dashed] (4,0)node[anchor=south] {$a$}..controls +(0.2,-0.8) and +(-0.2,-0.8)..(5,0)node[anchor=south] {$a+1$};
                    \draw (6,0) node[anchor=south] {$b$} -- (6,-1.05) node[rectangle, fill=white, draw] {$\Pi^{b}$} -- (6,-1.5);
                    \node at (3.5, -0.75) {$(-1)^b$};
                \end{scope}
            \end{scope}
        \end{tikzpicture}
    \end{equation*} 
    Using these, it is straightforward to verify the relations of $\sGalgp$ using the action of $\sGalgm$ on $\Cm$.

    Furthermore, we have $\Theta_i^+\bbC\neq 0$ if and only if $i=0$, and thus we obtain the homomorphism $\FunctorKLRtopnp$ of module categories as claimed.
\end{proof}
We conclude this section by stating some results about the functor of $\FunctorKLRtopn$.
All of these were already proven by Coulembier and Ehrig for the super Brauer category, and we translate their results to $\sG$.
For an introduction to up-down-tableaux and residue sequences (which are needed for the next statements) see \cref{sectionbasissg} below.
\begin{defi}
	We abbreviate the partition $(k,k-1,\dots, 1)$ by $\delta_k$.
\end{defi}
\begin{prop}\label{partitioncriterionforzero}
	Let $\bm{i}=(i_1, \dots, i_k)$ be the residue sequence of an up-tableau of shape $\lambda$.
	Then $F^{\eps}_{n}(\bm{i})=0$ if and only if $\delta_{n+1}\subseteq\lambda$.
\end{prop}
\begin{proof}
	The object $\bm{i}$ gets mapped under $\FunctorKLRtosBr$ to the indecomposable object associated to $\lambda$ as this maps to a generalized eigenspace for the action of the Jucys--Murphy elements (see \cref{relationSuperBrauer} and \cite{Cou18}).
	The statement then follows from \cite{CE21}*{Theorem 6.2.1}.
\end{proof}
\begin{thm}\label{partitioncriterionforproj}
	Let $\bm{i}=(i_1, \dots, i_k)$ be the residue sequence of an up-tableau of shape $\lambda$.
	The module $F^{\eps}_{n}(\bm{i})$ is projective if and only if $\delta_n\subseteq\lambda$.
\end{thm}
\begin{proof}
	This is \cite{CE21}*{Theorem 6.3.1}.
\end{proof}
\section{A basis for \texorpdfstring{$\sG$}{A}}\label{sectionbasissg}
This section aims at providing a basis for $\sG$. 
We begin by recalling the combinatorics of up-down-tableaux.
An up-down-tableau $\ta{t}$ of length $k$ is a sequence $((\ta{t}_0, f_0), (\ta{t}_1, f_1), \dots, (\ta{t}_k, f_k))$ of pairs $(\ta{t}_i, f_i)$, where $f_i\in\bbZ_{\geq0}$ and $\ta{t}_i$ is a partition of $i-2f_i$, such that $\ta{t}_i$ is obtained from $\ta{t}_{i-1}$ via adding (if $f_i=f_{i-1}$) or removing (if $f_i=f_{i-1}+1$) one box of the corresponding Young diagram.
We can draw up-down-tableau via drawing the Young diagrams of the partitions and arrows between consecutive partitions.
Observe that any up-down-tableau $\ta{t}$ necessarily has $\ta{t}_0=\emptyset$.

We will identify up-down-tableaux with sequences $\alpha(\ta{t})=(\alpha_1, \dots \alpha_k)$ of signed boxes $\alpha_i$ from $\ta{t}_{i-1}$ respectively $\ta{t}_i$.
Namely, $\alpha_i=(r,c)$ means adding a box in row $r$ and column $c$ to $\ta{t}_{i-1}$ and $\alpha_i = -(r,c)$ means removing the box from $\ta{t}_{i-1}$.

We call $\ta{t}_k$ the \emph{shape $\mathrm{Shape}(\ta{t})$ of $\ta{t}$}.
By $\restr{\ta{t}}{l}$ for $l<k$ we denote the up-down-tableau $((\ta{t}_0, f_0), (\ta{t}_1, f_1), \dots, (\ta{t}_l, f_l))$.

For a node $\alpha$ we define the \emph{source residue $\res^l(\alpha)$} respectively \emph{target residue $\res^r(\alpha)$}
\begin{flalign*}
	\res^l(\alpha)&\coloneqq \begin{cases}
        c-r&\text{if $\alpha=(r,c)$,}\\
        c-r+1&\text{if $\alpha=-(r,c)$.}
    \end{cases}
	&\res^r(\alpha)&\coloneqq \begin{cases}
	c-r&\text{if $\alpha=(r,c)$,}\\
	c-r-1&\text{if $\alpha=-(r,c)$,}
    \end{cases}
\end{flalign*}
\begin{rem}
	Usually when writing a map $f\colon A\to B$, the source is on the left and the target on the right.
	This is why we use the superscripts $l$ and $r$ for the source and target residue.
\end{rem}
Fix $\ta{t}$ with $\alpha(\ta{t})=(\alpha_1, \dots, \alpha_k)$. 
We define the target residue sequence of $\ta{t}$ to be $\bm{i}_{\ta{t}}^r\coloneqq (\res^r(\alpha_1), \dots, \res^r(\alpha_k))$ and the source residue sequence to be $\bm{i}_{\ta{t}}^l\coloneqq (\res^l(\alpha_1), \dots, \res^l(\alpha_k))$ if $\ta{t}=(\alpha_1, \dots, \alpha_k)$.
For a partition $\lambda$ denote by $\Add(\lambda)$ the boxes $\alpha$ that can be added to $\lambda$ and by $\Rem(\lambda)$ those that can be removed.
Additionally, let $\Add_i(\lambda)\coloneqq\{\alpha\in\Add(\lambda)\mid\res^l(\alpha)=i\}$ and $\Rem_i(\lambda)\coloneqq\{\alpha\in\Rem(\lambda)\mid\res^l(\alpha)=i\}$.
In the definition of $\Add_i(\lambda)$ and $\Rem_i(\lambda)$ we have $\res^l(\alpha)=\res^r(\alpha)$.
For $\alpha\in\Add(\lambda)$, we write $\lambda\oplus\alpha$ for the partition obtained by adding $\alpha$ to $\lambda$.
Similarly, $\lambda\ominus\alpha$ for $\alpha\in\Rem(\lambda)$ denotes the partition $\lambda$ with $\alpha$ removed.

We define $\ta{t}s_l$ to be the sequence of nodes $(\alpha_1, \dots,\alpha_{l-1}, \alpha_{l+1}, \alpha_l, \dots, \alpha_k)$.
Observe that this does not need to be an up-down-tableau.

For each partition $\lambda$ there exists a special up-down-tableaux $\ta{t}^\lambda$ which is given by first adding all the boxes of the first row, then of the second row and so on.

Write $\Par$ for the set of all partitions.
We denote by $\Tud_n(\lambda)$ the set of all up-down-tableaux of shape $\lambda$ and length $n$.
We also write $\Tud$ (and $\Tud_n$, $\Tud(\lambda)$) for the set of up-down-tableaux (of fixed length and fixed shape respectively).
In particular, $\abs{\lambda}=n-2f$ for some $f\geq 0$.
We define an ordering on $\Par$ by $\lambda>\mu$ if either $\abs{\lambda}<\abs{\mu}$ or $\abs{\lambda}=\abs{\mu}$ and $\lambda\geq\mu$ in the dominance ordering.
Furthermore, for $\ta{t}\in\Tud_n(\lambda)$ and $\ta{s}\in\Tud_l(\mu)$ we say $\ta{t}\geq\ta{s}$ if $n<l$ or $n=l$ and $\ta{t}_k\leq\ta{s}_k$ for all $1\leq k\leq n$.

Define the category $\Tudcat$ to be the category with objects $\Tud\sqcup\Tud$ (objects denoted by $\ta{t}^l$ resp.\ $\ta{t}^r$) where nonempty homomorphism spaces $\Hom_{\Tudcat}(\ta{t}, \ta{s})$ only occur (other than identities) if $\ta{t}=\ta{t}^l$ (i.e.~from first copy of $\Tud$), $\ta{s}=\ta{s}^r$ (i.e.~from second copy of $\Tud$) and $\shape(\ta{t})=\shape(\ta{s})$, in which case it contains a unique morphism.
We obtain a functor 
\begin{align*}
    \Psi\colon\Tudcat&\to\sG \\\ta{t}^g&\mapsto\bm{i}_{\ta{t}}^g \qquad\text{for $g\in\{r,l\}$}\\
    (\ta{t}^l\to\ta{s}^r)&\mapsto\Psi_{\ta{t}\ta{s}},
\end{align*} where $\Psi_{\ta{t}\ta{s}}$ is defined as follows.
Consider the minimal $k$ such that $\ta{t}_k$ is obtained from $\ta{t}_{k-1}$ by removing a box.
Denote by $l < k$ the index where this box was added to $\ta{t}$.
Then we draw a cap from $i_l$ to $i_k$ in $\bm{i}_{\ta{t}}^l$.
Do this iteratively for all boxes that were removed and similarly with cups for $\ta{s}$ using $\bm{i}_{\ta{t}}^l$.
Afterwards we connect the remaining residues in $\bm{i}_{\ta{t}}^l$ to the remaining ones of $\bm{i}_{\ta{s}}^r$ using as few crossings as possible.

As cups and caps have odd degree if $\eps=-1$ it is important to keep track of the height of these.
For $\Psi_{\ta{t}\ta{s}}$ we assume that all caps are below all cups and if two caps connect the positions $(k,l)$ and $(k',l')$ with $l<l'$, then $(k,l)$ is lower than $(k',l')$.
And for two cups connecting $(k,l)$ and $(k',l')$ with $l<l'$, then $(k,l)$ is higher than $(k',l')$ (see also the example below).
Observe that $\Psi_{\ta{t}\ta{s}}$ is well-defined due to \cref{tangleone,braid,tangletwo}.
\begin{ex}\label{expsist}
	Let $\ta{t} = \emptyset\to\ydiagram{1}\to\ydiagram{2}\to\ydiagram{2,1}\to\ydiagram{1,1}$ and $\ta{s}=\emptyset\to\ydiagram{1}\to\ydiagram{2}\to\ydiagram{1}\to\emptyset\to\ydiagram{1}\to\ydiagram{1,1}$.
    Then $\bm{i}_{\ta{t}}^l=(0,1,-1,2)$, $\bm{i}_{\ta{t}}^r=(0,1,-1,0)$, $\bm{i}_{\ta{s}}^l=(0,1,2,1,0,-1)$ and $\bm{i}_{\ta{s}}^r=(0,1,0,-1,0,-1)$.
    Hence,
	\begin{equation*}
		\Psi_{\ta{t}\ta{s}}=\begin{tikzpicture}[line width = \lw, myscale=0.6]
			\draw (0,2)..controls +(0.2,-1.2) and +(-0.2,-1.2)..(3,2) (1,2)..controls +(0.2,-0.8) and +(-0.2,-0.8)..(2,2) (1,0)--(4,2) (2,0)..controls +(0.2,0.8) and +(-0.2,0.8)..(4,0) (3,0)--(5,2);
			\draw (0,2) node[anchor=south] {$0$}(1,2) node[anchor=south] {$1$}(2,2) node[anchor=south] {$0$}(3,2) node[anchor=south] {$-1$}(4,2) node[anchor=south] {$0$}(5,2) node[anchor=south] {$-1$};
			\draw (1,0) node[anchor=north] {$0$}(2,0) node[anchor=north] {$1$}(3,0) node[anchor=north] {$-1$}(4,0) node[anchor=north] {$2$};
		\end{tikzpicture}\quad\text{and}\quad\Psi_{\ta{s}\ta{t}}=
		\begin{tikzpicture}[line width = \lw, myscale=0.6]
			\draw (0,0)..controls +(0.2,1.2) and +(-0.2,1.2)..(3,0) (1,0)..controls +(0.2,0.8) and +(-0.2,0.8)..(2,0) (1,2)--(4,0) (2,2)..controls +(0.2,-0.8) and +(-0.2,-0.8)..(4,2) (3,2)--(5,0);
			\draw (0,0) node[anchor=north] {$0$}(1,0) node[anchor=north] {$1$}(2,0) node[anchor=north] {$2$}(3,0) node[anchor=north] {$1$}(4,0) node[anchor=north] {$0$}(5,0) node[anchor=north] {$-1$};
			\draw (1,2) node[anchor=south] {$0$}(2,2) node[anchor=south] {$1$}(3,2) node[anchor=south] {$-1$}(4,2) node[anchor=south] {$0$};
		\end{tikzpicture}.
	\end{equation*}
    Observe that $\Psi_{\ta{t}\ta{s}}$ and $\Psi_{\ta{s}\ta{t}}$ are not composable (in either direction) since $\bm{i}^r_{\ta{s}}\neq\bm{i}^l_{\ta{s}}$ and $\bm{i}^r_{\ta{t}}\neq\bm{i}^l_{\ta{t}}$.
\end{ex}
\begin{rem}\label{restrictiontotlambda}
	For an up-down-tableaux $\ta{t}$ of shape $\lambda$ we have the morphisms $\Psi_{\ta{t}\ta{t}^\lambda}$ and $\Psi_{\ta{t}^\lambda\ta{t}}$.
	Using \cref{inverse}, it is easy to see that $\Psi_{\ta{t}\ta{s}}=\Psi_{\ta{t}^\lambda\ta{s}}\cdot\Psi_{\ta{t}\ta{t}^\lambda}$ (see also \cref{expsisttlambda} below).
	
	But note that this works only in this very special case.
	In general, it is not easy to compute the composition of two basis vectors.
\end{rem}
\begin{ex}\label{expsisttlambda}
	Let $\ta{t}$ and $\ta{s}$ as in \cref{expsist}.
	Then $\ta{t}^\lambda = \emptyset\to\ydiagram{1}\to\ydiagram{1,1}$ and $\bm{i}_{\ta{t}^\lambda}^l=\bm{i}_{\ta{t}^\lambda}^r=(0,-1)$.
	We have 
	\begin{equation*}
		\Psi_{\ta{t}^\lambda\ta{s}}=\begin{tikzpicture}[line width = \lw, myscale=0.7]
			\draw (0,2)..controls +(0.2,-1.2) and +(-0.2,-1.2)..(3,2) (1,2)..controls +(0.2,-0.8) and +(-0.2,-0.8)..(2,2) (4,1)--(4,2) (5,1)--(5,2);
			\draw (0,2) node[anchor=south] {$0$}(1,2) node[anchor=south] {$1$}(2,2) node[anchor=south] {$0$}(3,2) node[anchor=south] {$-1$}(4,2) node[anchor=south] {$0$}(5,2) node[anchor=south] {$-1$};
			\draw (4,1) node[anchor=north] {$0$}(5,1) node[anchor=north] {$-1$};
		\end{tikzpicture}\quad\text{and}\quad\Psi_{\ta{t}\ta{t}^\lambda}=\begin{tikzpicture}[line width = \lw, myscale=0.7]
			\draw (0,0)--(0,1) (1,0)..controls +(0.2,0.8) and +(-0.2,0.8)..(3,0) (2,0)--(1,1);
			\draw (0,1) node[anchor=south] {$0$}(1,1) node[anchor=south] {$-1$};
			\draw (1,0) node[anchor=north] {$0$}(2,0) node[anchor=north] {$1$}(3,0) node[anchor=north] {$-1$}(4,0) node[anchor=north] {$2$};
		\end{tikzpicture}
	\end{equation*}
	and thus clearly $\Psi_{\ta{t}\ta{s}}=\Psi_{\ta{t}^\lambda\ta{s}}\cdot\Psi_{\ta{t}\ta{t}^\lambda}$.
\end{ex}
\begin{thm}\label{spanningsetisbasis}
	The set $\{\Psi_{\ta{t}\ta{s}}\mid \ta{t},\ta{s}\in\Tud(\lambda)\text{ for some }\lambda\}$ forms a basis of $\sG$.
\end{thm}
The proof is technical, relying on filtration arguments and will be given in \cref{proofofspanningsetisbasis}. 
We formulate now two important consequences.
\begin{thm}\label{sGisquher}
	Let $I=\bigcup_{m\in\bbZ_{\geq0}}\bbZ^m$ and consider $\Par\subseteq I$ via identifying $\lambda$ with $\bm{i}^l_{\ta{t}^\lambda}=\bm{i}^r_{\ta{t}^\lambda}$.
	Define $Y(\bm{i},\lambda) = \{\Psi_{\ta{t}^\lambda\ta{s}}\mid\bm{i}^r_{\ta{s}}=\bm{i}\}$ and $X(\lambda, \bm{i}) = \{\Psi_{\ta{s}\ta{t}^\lambda}\mid\bm{i}^l_{\ta{s}}=\bm{i}\}$.
	This data endows $A\coloneqq\bigoplus_{m,n\in\mathbb{N}_0}\bigoplus_{\underline{i}\in\bbZ^m,\underline{j}\in\bbZ^n}\Hom_{\sGalg}(\underline{i},\underline{j})$ with the structure of an upper finite based quasi-hereditary (super-)algebra in the sense of \cite{BS21}.
\end{thm}
\begin{proof}
	Writing $Y(\lambda)\coloneqq \bigcup_{\bm{i}\in I}Y(\bm{i},\lambda)$ and $X(\lambda)\coloneqq\bigcup_{\bm{i}\in I}X(\lambda, \bm{i})$, \cref{restrictiontotlambda} and \cref{spanningsetisbasis} show that $\bigcup_{\lambda\in\Par}Y(\lambda)\times X(\lambda)$ is a basis of $A$.
	The set $Y(\mu, \lambda)$ can only be nonempty if $\lambda=\mu$ or $\abs{\mu}>\abs{\lambda}$, and thus $\mu\leq\lambda$, similarly for $X(\lambda, \mu)$.
	It is also clear that $X(\lambda, \lambda) = Y(\lambda,\lambda) = \{e_\lambda\}$ for each $\lambda\in\Par$. 
\end{proof}
\begin{prop}\label{sequenceiseitheruptableauorreducible}
    Restricting $\Psi$ to the full subcategory generated by up-tableaux (i.e.\ no removal of boxes) is essentially surjective.
\end{prop}
\begin{proof}
In virtue of \cref{spanningsetisbasis} we have to show that any source or target residue sequence of an up-down tableau is isomorphic in $\sG$ to a residue sequence of an up-tableau.
Without loss of generality we prove this only for the target residue sequence.
Denote the target residue sequence of $\ta{t}$ by $(i_1, \dots, i_n)$.
We claim that $(i_1, \dots i_n)$ is either a (target) residue sequence of an up-tableau or it contains a subsequence $(i,i\pm1,i)$.
Given such a subsequence, we can use \cref{idempotentone} and \cref{idempotenttwo} to construct an isomorphism to the altered sequence where $(i,i\pm1,i)$ is replaced by $i$.
Thus, it suffices to prove the claim.

Let $k$ be minimal such that $\ta{t}_k\oplus\{\alpha\}=\ta{t}_{k-1}$ for some $\alpha$.
Let $\res^r(\alpha)=i$ and $j$ be the index at which $\alpha$ was added to $\ta{t}$.
Then $i_k=i-1$ and $i_j=i$.
By assumptions on $k$ and $j$ we have $i_l\neq i,i-1$ for all $j<l<k$.
Thus, we may assume that $j=k-1$.

We claim that either $\ta{t}_{k-1}$ has an addable box of residue $i-1$ or $\ta{t}$ contains a subsequence $(i-1,i,i-1)$.
Suppose that $i_l\neq i-1$ for all $l<k-1$.
This means that $\alpha$ is added in the first column, and thus the box below $\alpha$ is addable and has residue $i-1$.
Otherwise, let $r<k-1$ be maximal with $i_r=i-1$.
By assumptions on $k$ and $r$ we have $i_l\neq i,i-1$ for all $r<l<k-1$.
If there is some index $r<l<k-1$ such that $i_l=i-2$ then the box to the left and to the bottom left of $\alpha$ is in part of $\ta{t}_{k-1}$, and thus it admits and addable box with residue $i-1$.
Otherwise, we find a subsequence $(i-1,i,i-1)$, and thus proving the claim.

Now assume that $\ta{t}_{k-1}$ has an addable box with residue $i-1$.
If there exists no index $l>k$ with $i_l=i$ then adding a box of residue $i-1$ instead of removing $\alpha$ gives actually an up-down-tableau with one less removal, and we can repeat this argument.
Otherwise, take the minimal $l>k$ with $i_l=i$.
As we are only interested in subsequences we may assume that $i_{k+1}$ is either $i+1$, $i-2$ or $i$.
In case it is $i$ we found the subsequence $(i,i-1,i)$.
The value $i-2$ cannot appear as we assumed that $\ta{t}_{k-1}$ admits an addable box of residue $i-1$.
So we may assume that $i_{k+1}=i+1$.
Now again we are only interested in subsequences, so we may assume that $i_{k+2}\in\{i+2,i\}$.
In case $i+2$ we may assume that $i_{k+3}\in\{i+3,i+1,i\}$.
If $i_{k+3}=i+1$ we found a subsequence $(i+1,i+2,i+1)$.
If $i_{k+3}=i$ we also may assume that $i_{k+2}=i$.
If $i_{k+3}=i+3$ we repeat this argument until we either find a subsequence $(i', i'+1,i')$ or that we may assume $i_{k+2}=i$.
So we are left with the case $(i,i-1,i+1,i)$.
As $i-1$ is removing a box of residue $i$, $i+1$ needs to remove a box of residue $i+2$ as it cannot be added.
In terms of Young diagrams it looks like $\ydiagram{2,1}\overset{i-1}{\to}\ydiagram{2}\overset{i+1}{\to}\ydiagram{1}$.
But we could also do $\ydiagram{2,1}\overset{i-1}{\to}\ydiagram{2,1,1}\overset{i+1}{\to}\ydiagram{2,2,1}\overset{i}{\to}\ydiagram{2,2,2}$ instead.
Now if no $j\in\{i+1,i,i-1\}$ appears this is also an up-down-tableau.
In each of the other cases we find a subsequence $(i', i'\pm1,i')$.
\end{proof}

\section{Khovanov algebra of type \texorpdfstring{$P$}{P}}\label{Khovanovdefisection}
In this chapter we aim at giving a different description of $\sGp$, similar to the Khovanov algebras of type $A$ or $B$.
\subsection{Cup and cap diagrams}
\begin{defi}
	A cup diagram $\underline{\lambda}$ is a partitioning $P$ of $\bbZ+\frac{1}{2}$ into subsets of size at most $2$ such that there are only finitely many size $2$ subsets.
	Furthermore, if $\{i,j\}\in P$ and $i<k<j$, then $k$ is part of a subset of size two $\{k,k'\}\in P$ with $i<k'<j$.

	We draw such a cup diagram (without crossings) by attaching a vertical ray to a one element subset and a cup with endpoints given by two element subsets.
	
	The set of all cup diagrams with $n$ cups is denoted by $\Lambda_n$ and by $\Lambda=\bigcup\Lambda_n$ we denote the set of all cup diagrams.
\end{defi}
The following is immediate from the definitions.
\begin{lem}\label{defidominttocupdiag}
	We have a bijection between $\Lambda_n$ and strictly decreasing sequences of integers $\lambda_1>\dots>\lambda_n$ by sending such a sequence to the unique cup diagram with endpoints $\{\lambda_1+\frac{3}{2}\}$.
	In particular, we can identify the set $\Lambda$ with strictly decreasing sequences of integers of arbitrary length.
\end{lem}
\begin{defi}\label{defipartialordering}
	Let $\underline{\lambda}$ and $\underline{\mu}$ be cup diagrams. 
	Denote by $\lambda_1>\lambda_2>\dots>\lambda_k$ the right endpoints of cups in $\underline{\lambda}$ and similar define $\mu_1>\dots>\mu_r$.
	We define a partial order on $\Lambda$ by declaring $\underline{\lambda}\leq\underline{\mu}$ if either $k>r$ or $k=r$ and $\lambda_i\geq\mu_i$ for $1\leq i\leq k$.
\end{defi}
\begin{ex}\label{excupdiag}
	The following are cup diagrams.
	\newcounter{count}
	\begin{gather*}
		\underline{\mu}=\begin{tikzpicture}[myscale=0.8]
			\foreach \x in {9,7,5,3,1} {
				\node[anchor=south] at (\value{count},0) {$-\frac{\x}{2}$};
				\stepcounter{count}
			}
			\foreach \x in {1,3,5,7,9} {
				\node[anchor=south] at (\value{count},0) {$\frac{\x}{2}$};
				\stepcounter{count}
			}
			\LINE{0}{-1}
			\SETCOORD{1}{1}
			\CUP{3}
			\SETCOORD{-2}{0}
			\CUP{1}
			\SETCOORD{2}{0}
			\LINE{0}{-1}
			\SETCOORD{1}{1}
			\LINE{0}{-1}
			\SETCOORD{1}{1}
			\CUP{1}
			\SETCOORD{1}{0}
			\LINE{0}{-1}
			\node at (-1,-0.5) {$\dots$};
			\node at (10,-0.5) {$\dots$};
		\end{tikzpicture}\\
		\underline{\lambda}=\begin{tikzpicture}[myscale=0.8]
			\setcounter{count}{0}
			\foreach \x in {9,7,5,3,1} {
				\node[anchor=south] at (\value{count},0) {$-\frac{\x}{2}$};
				\stepcounter{count}
			}
			\foreach \x in {1,3,5,7,9} {
				\node[anchor=south] at (\value{count},0) {$\frac{\x}{2}$};
				\stepcounter{count}
			}
			\LINE{0}{-1}
			\SETCOORD{1}{1}
			\LINE{0}{-1}
			\SETCOORD{1}{1}
			\CUP{1}
			\SETCOORD{1}{0}
			\CUP{1}
			\SETCOORD{1}{0}
			\LINE{0}{-1}
			\SETCOORD{1}{1}
			\CUP{1}
			\SETCOORD{1}{0}
			\LINE{0}{-1}
			\node at (-1,-0.5) {$\dots$};
			\node at (10,-0.5) {$\dots$};
		\end{tikzpicture}
	\end{gather*}
	The right endpoints of cups of $\underline{\lambda}$ are given by $\frac{7}{2}>\frac{1}{2}>-\frac{3}{2}$ and those of $\underline{\mu}$ by $\frac{7}{2}>-\frac{1}{2}>-\frac{3}{2}$, and thus $\underline{\lambda}<\underline{\mu}$.
\end{ex}
\begin{defi}
	A cap diagram $\overline{\lambda}=\underline{\lambda}^*$ is the horizontal mirror image of a cup diagram $\underline{\lambda}$ shifted one to the right, i.e.\ if $\underline{\lambda}$ has a cup connecting $a$ and $b$ then $\overline{\lambda}$ has a cap connecting $a+1$ and $b+1$.

	A circle diagram $\underline{\lambda}\overline{\mu}$ is a cup diagram $\underline{\lambda}$ drawn under a cap diagram $\overline{\mu}$.
\end{defi}
\begin{defi}\label{defiorientable}
	A circle diagram is called \emph{orientable} if it contains no circle and any non-propagating line has one endpoint to each side of $0$.

	To check this in practice, we first have to redraw cup and cap diagrams using the following rule.
	For a cup diagram with $k$ cups, we draw the bottom endpoints of rays in the unique way such that $-k-\frac{1}{2},-k+\frac{1}{2},\dots,k-\frac{3}{2}$ have no endpoints of rays.
	For a cap diagram with $k$ caps, we draw it in the unique way such that there are no top endpoints of rays at $-k+\frac{1}{2},-k+\frac{3}{2},\dots,k-\frac{1}{2}$.
	If we redraw the cup and the cap diagram in a circle diagram in this way, \emph{one endpoint to each side of $0$} means that, any non-propagating line ending at the top, is centered around $0$ and a non-propagating line ending at the bottom around $-1$.
\end{defi}
\begin{ex}
	For the cup diagram from \cref{excupdiag} we would think of the rays as the following.
	
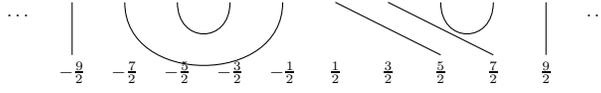
\begin{equation*}
		\begin{tikzpicture}[myscale=0.7]
			\setcounter{count}{0}
			\foreach \x in {9,7,5,3,1} {
				\node[anchor=north] at (\value{count},-1) {$-\frac{\x}{2}$};
				\stepcounter{count}
			}
			\foreach \x in {1,3,5,7,9} {
				\node[anchor=north] at (\value{count},-1) {$\frac{\x}{2}$};
				\stepcounter{count}
			}
			\LINE{0}{-1}
			\SETCOORD{1}{1}
			\CUP{3}
			\SETCOORD{-2}{0}
			\CUP{1}
			\SETCOORD{2}{0}
			\LINE{2}{-1}
			\SETCOORD{-1}{1}
			\LINE{2}{-1}
			\SETCOORD{-1}{1}
			\CUP{1}
			\SETCOORD{1}{0}
			\LINE{0}{-1}
			\node at (-1,-0.25) {$\dots$};
			\node at (10,-0.25) {$\dots$};
		\end{tikzpicture}
	\end{equation*}
\end{ex}
\begin{rem}
	We could have also chosen to define cup and cap diagrams with the skew rays, which would have made the orientability criterion easier but the definition of cup and cap diagrams more involved.
	We chose to do this differently as all these technicalities fade away, when we pass to representations of $\lie{p}(n)$.
\end{rem}
\subsection{Definition of \texorpdfstring{$\K$}{K}}
\begin{defi}
	We define the \emph{Khovanov algebra $\K$ of type $P$} to be the vector space with basis all orientable circle diagrams $\underline{\lambda}\overline{\mu}$.
	We define a multiplication by setting $\underline{\lambda}\overline{\mu}\cdot\underline{\mu'}\overline{\nu}$ to $0$ if $\mu\neq\mu'$, and otherwise we draw $\underline{\lambda}\overline{\mu}$ below $\underline{\mu'}\overline{\nu}$ and connect endpoints of rays from $\overline{\mu}$ to endpoints of rays of $\underline{\mu}$ in the following way:
	If this endpoint is not directly to the left of a cap we connect it with a straight line to the top. 
	If it is to the left of a cap we connect it to the first free endpoint right of the corresponding cup in $\underline{\mu}$.
	This gives us a middle section consisting of subpictures of the form 
	
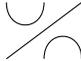
\begin{equation*}
		\begin{tikzpicture}[scale=0.5]
			\CUP{1}
			\SETCOORD{1}{0}
			\LINE{-2}{-1.5}
			\SETCOORD{1}{0}
			\CAP{1}
		\end{tikzpicture}
	\end{equation*}
	We then apply surgery procedures to remove this middle section.
	This is done by replacing these subpictures iteratively by 
	
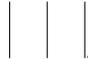
\begin{equation*}
		\tikz[baseline=(current bounding box.south), scale=0.5]{\draw (0,0)--(0,-1.5)(1,0)--(1,-1.5)(2,0)--(2,-1.5);}.
	\end{equation*}
	If any of these intermediate results is not orientable, we define the result to be $0$. 
	Otherwise, we set it to be $\underline{\lambda}\overline{\nu}$ (which is exactly the result of the surgery procedure).
	This turns $\K$ into an associative algebra by \cref{khovanovalgassoc} below.

	Furthermore, denote by $e_\lambda$ the basis vector $\underline{\lambda}\overline{\lambda}\in\K$ for all $\lambda\in\Lambda$.
\end{defi}
\begin{rem}
	We chose to call these replacements surgery procedures (even though they are, strictly speaking, not) as they play the same role as the surgery procedures from \cite{BS11} and \cite{ES16a} in the definition of the Khovanov algebra of types $A$ and $B$.
\end{rem}
\subsection{Detailed analysis of surgery procedures}\label{detailedanalysis}
Before we are going to prove associativity, we inspect the surgery procedure in more detail.
Every surgery procedure falls into one of the following three categories.
\subsubsection{Straightening}
This appears if we are in the situation of \cref{figureStraightening}.
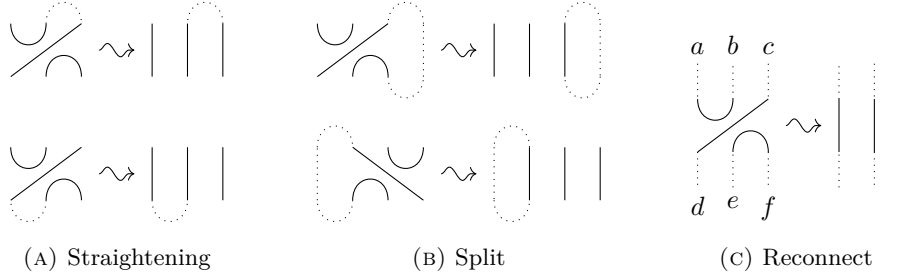
\begin{figure}
	\begin{subfigure}[B]{0.25\textwidth}\centering
		\begin{tikzpicture}[scale=0.47]
			\CUP{1}
			\CAP[dotted]{1}
			\LINE{-2}{-1.5}
			\SETCOORD{1}{0}
			\CAP{1}
			\SETCOORD{0}{-3.5}
			\CAP{-1}
			\CUP[dotted]{-1}
			\LINE{2}{1.5}
			\SETCOORD{-1}{0}
			\CUP{-1}
			\draw[->, decorate, decoration={snake}] (2.5, -0.75) -- (3.5, -0.75);
			\draw[->, decorate, decoration={snake}] (2.5, -4.25) -- (3.5, -4.25);
			\SETCOORD{4}{0}
			\LINE{0}{-1.5}
			\CUP[dotted]{1}
			\LINE{0}{1.5}
			\SETCOORD{1}{0}
			\LINE{0}{-1.5}
			\SETCOORD{0}{3.5}
			\LINE{0}{1.5}
			\CAP[dotted]{-1}
			\LINE{0}{-1.5}
			\SETCOORD{-1}{0}
			\LINE{0}{1.5}
		\end{tikzpicture}
		\subcaption{Straightening}
		\label{figureStraightening}
	\end{subfigure}
	\begin{subfigure}[B]{0.45\textwidth}\centering
		\begin{tikzpicture}[scale=0.47]
			\CUP{1}
			\SETCOORD{0}{-1.5}
			\CAP{1}
			\CUP[dotted]{1}
			\LINE[dotted]{0}{1.5}
			\CAP[dotted]{-1}
			\LINE{-2}{-1.5}
			\draw[->, decorate, decoration={snake}] (3.5, -0.75) -- (4.5, -0.75);
			\draw[->, decorate, decoration={snake}] (3.5, -4.25) -- (4.5, -4.25);
			\SETCOORD{5}{0}
			\LINE{0}{1.5}
			\SETCOORD{1}{0}
			\LINE{0}{-1.5}
			\SETCOORD{1}{0}
			\LINE{0}{1.5}
			\CAP[dotted]{1}
			\LINE[dotted]{0}{-1.5}
			\CUP[dotted]{-1}
			\SETCOORD{1}{-3.5}
			\LINE{0}{1.5}
			\SETCOORD{-1}{0}
			\LINE{0}{-1.5}
			\SETCOORD{-1}{0}
			\LINE{0}{1.5}
			\CAP[dotted]{-1}
			\LINE[dotted]{0}{-1.5}
			\CUP[dotted]{1}
			\SETCOORD{-3}{1.5}
			\CUP{-1}
			\SETCOORD{0}{-1.5}
			\CAP{-1}
			\CUP[dotted]{-1}
			\LINE[dotted]{0}{1.5}
			\CAP[dotted]{1}
			\LINE{2}{-1.5}
		\end{tikzpicture}
		\subcaption{Split}
		\label{figureSplit}
	\end{subfigure}
	\begin{subfigure}[B]{0.24\textwidth}\centering
		\begin{tikzpicture}[scale=0.47]
			\LINE[dotted]{0}{-1}
			\CUP{1}
			\LINE[dotted]{0}{1}
			\SETCOORD{1}{0}
			\LINE[dotted]{0}{-1}
			\LINE{-2}{-1.5}
			\LINE[dotted]{0}{-1}
			\SETCOORD{1}{0}
			\LINE[dotted]{0}{1}
			\CAP{1}
			\LINE[dotted]{0}{-1}
			\draw[->, decorate, decoration={snake}] (2.5, -1.75) -- (3.5, -1.75);
			\SETCOORD{2}{0}
			\LINE[dotted]{0}{1}
			\LINE{0}{1.5}
			\LINE[dotted]{0}{1}
			\SETCOORD{1}{0}
			\LINE[dotted]{0}{-1}
			\LINE{0}{-1.5}
			\LINE[dotted]{0}{-1}
			\SETCOORD{1}{0}
			\LINE[dotted]{0}{1}
			\LINE{0}{1.5}
			\LINE[dotted]{0}{1}
			\node[anchor=south] at (0,0) {$a$};
			\node[anchor=south] at (1,0) {$b$};
			\node[anchor=south] at (2,0) {$c$};
			\node[anchor=north] at (0,-3.5) {$d$};
			\node[anchor=north] at (1,-3.5) {$e$};
			\node[anchor=north] at (2,-3.5) {$f$};
		\end{tikzpicture}
		\caption{Reconnect}
		\label{figureReconnect}
	\end{subfigure}
\caption{Different cases of the surgery procedure}
\end{figure}
In this case the surgery procedure does not change orientability and just \enquote{straightens} the kink.
\subsubsection{Split}
In this situation we split off a circle from a line, and thus the surgery procedure gives $0$.
\Cref{figureSplit} shows this in detail.
\subsubsection{Reconnect}
The last case is that none of the endpoints of the surgery are connected with one another as in \cref{figureReconnect}.
In this situation we only reconnect three lines.

If all three lines are propagating, then the two right endpoints (as well as the two left endpoints) end either both on the top or both on the bottom.
If these endpoints lie on different sides of $0$, the surgery procedure is nonzero and produces two non-propagating lines.

In all other cases, we claim that the surgery procedure will produce $0$.
If all three lines are propagating, but some right (or left) endpoints do not lie on the same side of $0$, then the result is not orientable, and thus $0$.
Now suppose that there is at least one non-propagating line. 
We may assume without loss of generality that it corresponds to the endpoints $a$ and $b$ and without loss of generality these end on the top.
As our diagram is orientable, we know that $a$ and $b$ lie on the same side of $0$.
Thus, (if the result is also orientable) $d$ has to end on the bottom.
Furthermore, $c$ and $f$ cannot both end on the top and $e$, $f$ and $c$ cannot all end on the bottom by the same argument as for $a$, $b$ and $d$.
Hence, $a$, $b$ and $c$ end on the top and $d$, $e$ and $f$ on the bottom.
Thus, we have two non-propagating lines, one associated to the cap and one to the cup.

Then $e$ and $f$ have to be at positions $-k-\frac{3}{2}$ and $k-\frac{1}{2}$ where $k$ denotes the number of cups as the diagram is orientable (see also \cref{defiorientable}).
On the other hand observe that $a$ and $b$ have to be at positions $-k-\frac{1}{2}$ and $k+\frac{1}{2}$.
But this can only be achieved if there are more cups than caps to the left of this surgery procedure, which means that there has to be a non-propagating line ending at the top that is not nested with the one coming from the cup, thus this is not orientable.

\begin{thm}\label{khovanovalgassoc}
	The multiplication on $\K$ is well-defined and associative.
\end{thm}
\begin{proof}
    That it is well-defined follows from \cref{surgproccommute} by horizontal stacking of surgery procedures and associativity by vertical stacking.
\end{proof} 
\subsection{Properties of \texorpdfstring{$\K$}{K}}
\begin{defi}
	For a decreasing sequence $\lambda=(\lambda_1>\dots>\lambda_n)$ of integers let $e_\lambda=\underline{\lambda}\overline{\lambda}$.
\end{defi}
\begin{lem}\label{multwithidempotents}
	The following hold true.
	\begin{equation*}
		e_\lambda\cdot\underline{\nu}\overline{\mu}=\delta_{\lambda\nu}\underline{\lambda}\overline{\mu}\qquad\qquad
		\underline{\lambda}\overline{\nu}\cdot e_\mu=\delta_{\nu\mu}\underline{\lambda}\overline{\mu}
	\end{equation*}
\end{lem}
\begin{proof}
	We only prove the first equation.
	If $\nu\neq\lambda$ the statement is clear by definition of the multiplication.
	Otherwise, as $e_\lambda=\underline{\lambda}\overline{\lambda}$, every surgery procedure is given locally by
	
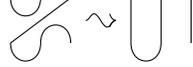
\begin{equation*}
		\begin{tikzpicture}[scale=0.4]
			\CUP{1}
			\SETCOORD{1}{0}
			\LINE{-2}{-1.5}
			\CUP{1}
			\CAP{1}
			\draw[->, decorate, decoration={snake}] (2.5,-0.75)--(3.5,-0.75);
			\SETCOORD{2}{1.5}
			\LINE{0}{-1.5}
			\CUP{1}
			\LINE{0}{1.5}
			\SETCOORD{1}{0}
			\LINE{0}{-1.5}
		\end{tikzpicture}
	\end{equation*}
	which clearly does not change orientability. 
	This means that any surgery procedure gives a non-zero result, and thus we obtain $\underline{\lambda}\overline{\mu}$ in the end.
\end{proof}
\begin{cor}\label{klocuntial}
	The algebra $\K$ is a locally unital locally finite dimensional algebra with $\K=\bigoplus_{\lambda,\mu\in\Lambda}e_\lambda\K e_\mu$.
\end{cor}
\begin{proof}
	By \cref{multwithidempotents} we know that the set of all $e_\lambda$ form a set of mutually orthogonal idempotents and the direct sum decomposition is then immediate from \cref{defidominttocupdiag}.
	It is locally finite dimensional as $e_\lambda\K e_\mu$ is non-zero if and only if $\underline{\lambda}\overline{\mu}$ is orientable in which case it gives a basis.
\end{proof}
\begin{defi}
	For each $\lambda\in\Lambda$ we have a one dimensional irreducible module $\hat{L}(\lambda)$ with action given by
	 \begin{equation*}
		\underline{\mu}\overline{\nu}\cdot v=\begin{cases}
			v&\text{if $\underline{\mu}\overline{\nu}=e_\lambda$,}\\
			0&\text{otherwise.}
		\end{cases}
	 \end{equation*}
	 Each of these has a locally finite dimensional projective cover denoted by $\hat{P}(\lambda)=\K e_\lambda$.
\end{defi}
\subsection{Geometric bimodules}\label{geombimodK}
\begin{defi}
	A \emph{crossingless matching} is a diagram $t$ given by drawing a cap diagram $c$ underneath a cup diagram $d$ and connecting the rays from $c$ to $d$ via an order preserving bijection such that outside some finite strip we only have straight vertical rays.
	For $i\in\bbZ$ we can consider special crossingless matchings $t^i$ given as
	
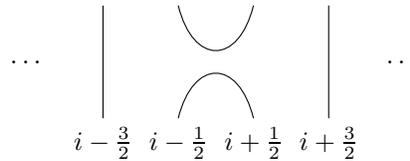
\begin{equation*}
		\begin{tikzpicture}
			\node at (0,0.75) {$\dots$};
			\node at (5,0.75) {$\dots$};
			\node[anchor=north] at (1,0) {$i-\frac{3}{2}$};
			\node[anchor=north] at (2,0) {$i-\frac{1}{2}$};
			\node[anchor=north] at (3,0) {$i+\frac{1}{2}$};
			\node[anchor=north] at (4,0) {$i+\frac{3}{2}$};
			\draw (1,0)--(1,1.5) (2,0)..controls+(0.2,0.8) and +(-0.2,0.8)..(3,0) (2,1.5)..controls+(0.2,-0.8) and +(-0.2,-0.8)..(3,1.5) (4,0)--(4,1.5);
		\end{tikzpicture}
	\end{equation*} 
	A \emph{generalized crossingless matching} $\bm{t}$ is a sequence $t_k\cdots t_1$ of crossingless matchings.

	Given a cup diagram $a$ and a cap diagram $b$ we can form a \emph{generalized circle diagram} $a\bm{t}b$.
	This is called \emph{orientable} if it contains neither a circle nor a non-propagating such that both endpoints lie on the same side of $0$.
\end{defi}
\begin{defi} 
	Given a generalized crossingless matching $\bm{t}=t_k\cdots t_1$ we define $\hat{G}_{\bm{t}}$ to be the vector space with basis all orientable generalized circle diagrams $\underline{\lambda}\bm{t}\overline{\mu}$.

	Given another generalized crossingless matching $\bm{u}=u_l\cdots u_1$ we define $\bm{u}\bm{t}$ to be the generalized crossingless matching $u_l\cdots u_1 t_k\cdots t_1$.
	We define a map $m\colon \hat{G}_{\bm{u}}\otimes \hat{G}_{\bm{t}}\to \hat{G}_{\bm{u}\bm{t}}$ by defining the product $(a\bm{u}b)(c\bm{t}d)$ of two basis vectors as follows. If $c^*\neq b$, we declare $(a\bm{u}b)(c\bm{t}d)=0$. Otherwise, we draw $(a\bm{u}b)$ underneath $(c\bm{t}d)$ to create a new diagram.
	We now can use the surgery procedures two remove the middle section $c^*c$ to obtain a vector in $\hat{G}_{\bm{u}\bm{t}}$.
\end{defi}
Given a third generalized crossingless matching $\bm{s}$, the following diagram commutes (by \cref{surgproccommute}).
\begin{equation}\label{geombimodassoc}
	\begin{tikzcd}
		\hat{G}_{\bm{u}}\otimes \hat{G}_{\bm{t}}\otimes \hat{G}_{\bm{s}}\arrow[r, "1\otimes m"]\arrow[d, "m\otimes 1"]&\hat{G}_{\bm{u}}\otimes \hat{G}_{\bm{t}\bm{s}}\arrow[d, "m"]\\
		\hat{G}_{\bm{u}\bm{t}}\otimes \hat{G}_{\bm{s}}\arrow[r, "m"]&\hat{G}_{\bm{u}\bm{t}\bm{s}}
	\end{tikzcd}
\end{equation}
\begin{rem}
	Using \eqref{geombimodassoc} with $\bm{s}$ and $\bm{u}$ being the empty generalized crossingless matchings, we endow $\hat{G}_{\bm{t}}$ with the structure of a $\K$-$\K$-bimodule.
\end{rem}
\begin{defi}
	The \emph{reduction $\red(\bm{t})$} of a generalized crossingless matching $\bm{t}$ is the unique crossingless matching (i.e.\ only one layer) that is topologically equivalent.
	Similarly, given a crossingless matching $t$ and a cup diagram $a$ (resp.\ cap diagram $b$) we define the \emph{lower reduction $\lred(at)$} of $at$ (resp.\ \emph{upper reduction $\ured(tb)$} of $tb$) to be the unique cup (resp.\ cap) diagram that is topologically equivalent to $at$ (resp.\ $tb$).
\end{defi}
The following is clear from the definitions.
\begin{cor}\label{bimoduleisoreductiontosinglelayer}
	There exists a $\K$-$\K$-bimodule isomorphism $\hat{G}_{\bm{t}}\to \hat{G}_{\red(\bm{t})}$ via sending $\underline{\lambda}\bm{t}\overline{\mu}$ to $\underline{\lambda}\red(\bm{t})\overline{\mu}$.
\end{cor}
\begin{rem}
	\Cref{bimoduleisoreductiontosinglelayer} allows us to reduce to crossingless matchings.
	In the following, we will consider everything only for the non-generalized version.
\end{rem}
\begin{lem}\label{multdescendstoiso}
	The map $m\colon \hat{G}_{u}\otimes \hat{G}_{t}\to \hat{G}_{ut}$ descends to an isomorphism $\overline{m}\colon \hat{G}_{u}\otimes_{\K} \hat{G}_{t}\to \hat{G}_{ut}$.
\end{lem}
\begin{proof}
	The map $m$ induces a map $\overline{m}\colon \hat{G}_{u}\otimes_{\K} \hat{G}_{t}\to \hat{G}_{ut}$ by the associativity \eqref{geombimodassoc}.
	Note that the left-hand side is generated by elements of the form $au b\otimes ct d$ where $a$ and $c$ are cup diagrams and $b$ and $d$ cap diagrams.
	By definition of the multiplication and reduction we have $au(\lred(au))^*\cdot\lred(au)b=au b$.
	But this means that the domain of $\overline{m}$ is in fact generated by 
	\begin{equation}\label{gensettensorofbimod}
		au(\lred{au})^*\otimes \lred(au)t d.
	\end{equation}
	Again by the definition of the multiplication, $\overline{m}(au(\lred{au})^*\otimes \lred(au)t d)=aut d$.
	Thus, the generating set \eqref{gensettensorofbimod} maps to a basis of $\hat{G}_{ut}$.
	Hence, $\overline{m}$ is an isomorphism.
\end{proof}
\begin{defi}
	Given a crossingless matching $t$, we define $\hat{\theta}_{t}$ to be the endofunctor of $\K\mmod$ given by tensoring with the $\K$-$\K$-bimodule $\hat{G}_{t}$ over $\K$.
	If $t=t^i$ we abbreviate $\hat{\theta}_i\coloneqq \hat{\theta}_{t^i}$.
\end{defi}
One should think of $\hat{\theta}_i$ as equivalents of the labelled strands in $\sGp$.
\begin{prop}\label{geombimodonproj}
	We have 
	\begin{equation*}
		\hat{\theta}_{t}\hat{P}(\overline{\gamma})=\begin{cases}
			0&\text{if $t\overline{\gamma}$ is not orientable,}\\
			\hat{P}(\nu)&\text{otherwise, where $\overline{\nu}=\ured(t\overline{\gamma})$.}
		\end{cases}
	\end{equation*}
\end{prop}
\begin{proof}
	We clearly have $\hat{\theta}_{t}\hat{P}(\overline{\gamma})=\hat{G}_{t}e_{\gamma}$.
	But now if $t\overline{\gamma}$ is not orientable, then the right-hand side is $0$ by definition.
	Otherwise, we denote by $\nu$ its upper reduction.
	Then we have a linear map
	\begin{equation*}
		f\colon\hat{\theta}_{t}\hat{P}(\overline{\gamma})\to \hat{P}(\nu),\qquad\underline{\lambda}t\overline{\gamma}\mapsto\underline{\lambda}\overline{\nu}.
	\end{equation*} 
	This is an isomorphism as the process of upper reduction does not change orientability.
	Furthermore, it is $\K$-linear because the multiplication procedure depends only topologically on the diagram.
\end{proof}
\begin{cor}\label{gtprojleftright}
The $\K$-$\K$-bimodule $\hat{G}_{t}$ is sweet, i.e.\ projective as a right and as a left $\K$-module.
\end{cor}
\begin{proof}
	It is projective as left $\K$-module by \cref{geombimodonproj} as $\hat{G}_t=\bigoplus_{\lambda\in\Lambda}\hat{G}_t e_\lambda$.
	For projectivity as right $\K$-module use the right-hand analogue of \cref{geombimodonproj}.
\end{proof}
Our next goal is to prove an adjunction theorem between the $\hat{\theta}_{t}$.
\begin{defi}
	Given a crossingless matching $t$ we define $t^\ddagger$ to be the horizontal mirror image of $t$ shifted one to the right.

	Furthermore, we define $\phi\colon \hat{G}_{t^\ddagger}\otimes \hat{G}_{t}\to\K$ on basis vectors as 
	\begin{equation*}
		\phi(\underline{\lambda}t^\ddagger\overline{\nu}\otimes\underline{\nu'}t\overline{\mu})=\delta_{\nu\nu'}(\underline{\lambda}\ured(t^\ddagger\overline{\nu}))\cdot(\lred(\underline{\nu'}t)\overline{\mu}).
	\end{equation*}
\end{defi}
\begin{lem}\label{phiisbimodandbalanced}
	The map $\phi$ is a homogeneous $\K$-$\K$-bimodule homomorphism that is also $\K$-balanced.
\end{lem}
\begin{proof}
	We can also realize $\phi$ as the following composition
	\begin{equation*}
		\hat{G}_{t^\ddagger}\otimes \hat{G}_{t}\overset{m}{\to} \hat{G}_{t^\ddagger t}\overset{\cong}{\to}\hat{G}_{\red(t^\ddagger t)}\overset{\omega}{\to} \K
	\end{equation*}
	where $\omega$ is given by applying the surgery procedures to eliminate the middle section $\red(t^\ddagger t)$.
	This is possible because $t^\ddagger$ was defined as the horizontal mirror image shifted one to the right.
	
	Note that this composition is a bimodule map as each of the composites is by \eqref{geombimodassoc}, \cref{bimoduleisoreductiontosinglelayer} and \cref{surgproccommute}.
	Furthermore, $m$ is balanced by \cref{geombimodassoc}, and thus $\phi$ is $\K$-balanced as well.
\end{proof}
\begin{thm}\label{adjunctionhelperthm}
	There is a homogeneous $\K$-$\K$-bimodule isomorphism.
	\begin{equation*}
		\hat{\phi}\colon \hat{G}_t\to \Hom_{\K}(\hat{G}_{t^\ddagger},\K)
	\end{equation*}
	given by sending $y\in \hat{G}_t$ to $\hat{\phi}(y)\colon \hat{G}_{t^\ddagger}\to\K, x\mapsto\phi(x\otimes y)$.
\end{thm}
\begin{proof}
	First, note that this is well-defined as $\phi(\_\otimes y)$ is a left $\K$-module homomorphism by \cref{phiisbimodandbalanced}.

	To show that this is a $\K$-$\K$-bimodule homomorphism let $u\in\K$, $x\in \hat{G}_{t^\ddagger}$ and $y\in \hat{G}_t$.
	Then we have
	\begin{align*}
		(u\hat{\phi}(y))(x)&=\hat{\phi}(y)(xu)=\phi(xu\otimes y)= \phi(x\otimes uy) = (\hat{\phi}(uy))(x),\\
		(\hat{\phi}(y)u)(x)&=(\hat{\phi}(y)(x))u=\phi(x\otimes y)u=\phi(x\otimes yu) = (\hat{\phi}(yu))(x),
	\end{align*}
	and thus $\hat{\phi}(uy) = u\hat{\phi}(y)$ and $\hat{\phi}(yu)=\hat{\phi}(y)u$.

	It remains to show that $\hat{\phi}$ is a vector space isomorphism.
	For this it suffices to show that the restriction
	\begin{equation*}
		\hat{\phi}\colon e_\lambda \hat{G}_t\to\Hom_{\K}(\hat{G}_{t^\ddagger}e_\lambda, \K)
	\end{equation*}
	is an isomorphism.
	Now as $t^\ddagger$ is the mirror image and shifted one to the right with respect to $t$ we see that $e_\lambda \hat{G}_t\neq 0$ if and only if $\hat{G}_{t^\ddagger}e_\lambda\neq 0$ (recall that cap diagrams arise by shifting the horizontal mirror image of cup diagrams one to the right).
	Thus, we may assume that $e_\lambda \hat{G}_t\neq 0$.
	In this case we have by the mirrored version of \cref{actionofgeombimodonprojnuclear} $e_\lambda \hat{G}_t\cong e_\nu\K$ and $\hat{G}_{t^\ddagger}e_\lambda=\K e_\nu$ for some $\nu$ (which is the same for both).
	Under this isomorphism $\phi$ translates to multiplication.
	Thus, we need to show that $e_\nu\K\to\Hom_{\K}(\K e_\nu,\K), y\mapsto (x\mapsto xy)$ is an isomorphism, but this is obvious.
\end{proof}
\begin{cor}\label{hatthetaadjunction}
	We have an adjunction $(\hat{\theta}_{t^\ddagger},\hat{\theta}_t)$.
	In particular, $\hat{\theta}_{i+1}$ is left adjoint to $\hat{\theta}_{i}$.
\end{cor}
\begin{proof}
	We have the usual adjunction $(\hat{G}_{t^\ddagger}\otimes_{\K}\_, \Hom_{\K}(\hat{G}_{t^\ddagger},\_))$.
	As $\hat{G}_{t^\ddagger}$ is projective as a left $\K$-module by \cref{gtprojleftright}, there exists a natural isomorphism $\Hom_{\K}(\hat{G}_{t^\ddagger},\K)\otimes_{\K}M\to\Hom_{\K}(\hat{G}_{t^\ddagger},M)$.
	Precomposing this with the isomorphism from \cref{adjunctionhelperthm} gives the desired adjunction.
\end{proof}
\subsection{Relation to \texorpdfstring{$\sG$}{G}}
Next we define natural transformations $\eta_i\colon\hat{\theta}_{i+1}\hat{\theta}_i \to \K$, $\epsilon_{i}\colon\K\to \hat{\theta}_{i-1}\hat{\theta}_i$ and $\psi_{i,j}\colon\hat{\theta}_i\hat{\theta}_j\to \hat{\theta}_j\hat{\theta}_i$ that satisfy the same relations as $\sGp$.
This will be the first step in relating $\K$ with $\sGp$.

Instead of directly defining the natural transformations, we first define bimodule homomorphisms $\hat{\eta}_i\colon \hat{G}_{t^{i+1}t^i}\to\K$, $\hat{\epsilon}_i\colon\K\to \hat{G}_{t^{i-1}t^i}$ and $\hat{\psi}_{i,j}\colon \hat{G}_{t^it^j}\to \hat{G}_{t^jt^i}$.
Using \cref{multdescendstoiso} this induces then the natural transformations $\eta_i$, $\epsilon_i$ and $\psi_{i,j}$.
\begin{defi}\label{defibimodulemaps}
	We define $\hat{\eta}_i$, $\hat{\epsilon}_i$ and  $\hat{\psi}_{i,j}$ for $j\neq i$, $i\pm1$ to be the linear maps
	\begin{flalign*}
		\hat{\eta}_i\colon \hat{G}_{t^{i+1}t^i}&\to\K&
		\hat{\epsilon}_i\colon\K&\to \hat{G}_{t^{i-1}t^i}\\
		at^{i+1}t^ib&\mapsto\begin{cases}
			ab&\text{if $ab$ orient.,}\\
			0&\text{otherwise,}
		\end{cases}&
		ab&\mapsto\begin{cases}
			at^{i-1}t^ib&\text{if $at^{i-1}t^ib$ orient.,}\\
			0&\text{otherwise,}
		\end{cases}\\[0.5\baselineskip]
        \hat{\psi}_{i,j}\colon \hat{G}_{t^it^j}&\to \hat{G}_{t^jt^i}&
		\hat{\psi}_{i,i+1}\colon \hat{G}_{t^it^{i+1}}&\to \hat{G}_{t^{i+1}t^i}\\
        at^it^jb&\mapsto at^{j}t^ib&
		at^it^{i+1}b&\mapsto\begin{cases}
			at^it^{i+1}b&\text{if $at^it^{i+1}b$ orient.,}\\
			0&\text{otherwise.}
		\end{cases}\\
	\end{flalign*}
	See also \cref{figureformapsbetweenthetas} to see a pictorial description of these maps.
\end{defi}
\begin{figure}[ht]
	\centering	
	\begin{tabular}{*{2}{m{0.4\textwidth}}}
		\begin{tikzpicture}[scale=0.4]
			\CUP{1}
			\SETCOORD{1}{0}
			\LINE{0}{-1.5}
			\CUP{-1}
			\CAP{-1}
			\LINE{0}{-1.5}
			\SETCOORD{1}{0}
			\CAP{1}
			\draw[->] (2.5,-1.5)--(3.5,-1.5);
			\SETCOORD{2}{0}
			\LINE{0}{3}
			\SETCOORD{1}{0}
			\LINE{0}{-3}
			\SETCOORD{1}{0}
			\LINE{0}{3}
			\node at (3,-3.75) {$\hat{\eta}_i\colon \hat{G}_{t^{i+1}t^i}\to\K$};
			\useasboundingbox ([shift={(0mm,0mm)}]current bounding box.north east) rectangle ([shift={(0mm,-5mm)}]current bounding box.south west);
		\end{tikzpicture}&
		\begin{tikzpicture}[scale=0.4]
			\CUP{1}
			\SETCOORD{-2}{0}
			\LINE{0}{-1.5}
			\CUP{1}
			\CAP{1}
			\LINE{0}{-1.5}
			\SETCOORD{-1}{0}
			\CAP{-1}
			\draw[->] (1.5,-1.5)--(2.5,-1.5);
			\SETCOORD{4}{3}
			\CUP{1}
			\SETCOORD{1}{0}
			\LINE{0}{-1.5}
			\CUP{-1}
			\CAP{-1}
			\LINE{0}{-1.5}
			\SETCOORD{1}{0}
			\CAP{1}
			\node at (2,-3.75) {$\hat{\psi}_{i,i+1}\colon \hat{G}_{t^it^{i+1}}\to \hat{G}_{t^{i+1}t^i}$};
		\end{tikzpicture}\\
		\begin{tikzpicture}[scale=0.4]
			\LINE{0}{3}
			\SETCOORD{1}{0}
			\LINE{0}{-3}
			\SETCOORD{1}{0}
			\LINE{0}{3}
			\draw[->] (2.5,1.5)--(3.5,1.5);
			\SETCOORD{3}{0}
			\CUP{1}
			\SETCOORD{-2}{0}
			\LINE{0}{-1.5}
			\CUP{1}
			\CAP{1}
			\LINE{0}{-1.5}
			\SETCOORD{-1}{0}
			\CAP{-1}
			\node at (3,-0.75) {$\hat{\epsilon}_i\colon\K\to \hat{G}_{t^{i-1}t^i}$};
		\end{tikzpicture}&
		\begin{tikzpicture}[scale=0.4]
			\CUP{1}
			\SETCOORD{1}{0}
			\LINE{0}{-1.5}
			\CUP{1}
			\LINE{0}{1.5}
			\SETCOORD{0}{-3}
			\CAP{-1}
			\SETCOORD{-1}{0}
			\LINE{0}{1.5}
			\CAP{-1}
			\LINE{0}{-1.5}
			\draw[->] (3.5,-1.5)--(4.5,-1.5);
			\SETCOORD{5}{0}
			\CAP{1}
			\SETCOORD{1}{0}
			\LINE{0}{1.5}
			\CAP{1}
			\LINE{0}{-1.5}
			\SETCOORD{0}{3}
			\CUP{-1}
			\SETCOORD{-1}{0}
			\LINE{0}{-1.5}
			\CUP{-1}
			\LINE{0}{1.5}
			\node at (4,-3.75) {$\hat{\psi}_{i,j}\colon \hat{G}_{t^it^j}\to \hat{G}_{t^jt^i}$};
		\end{tikzpicture}
	\end{tabular}
	\caption{Pictorial description of the maps $\hat{\eta}_i$, $\hat{\epsilon}_i$ and $\hat{\psi}_{i,j}$}\label{figureformapsbetweenthetas}
\end{figure}
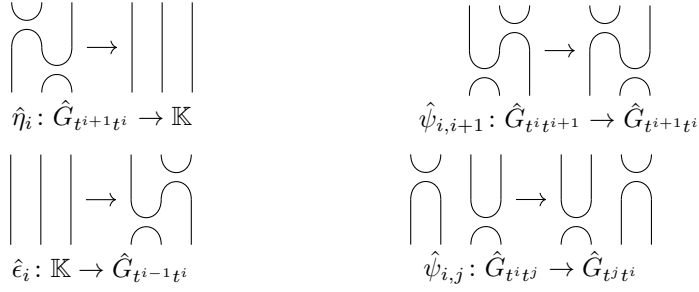
\begin{rem}
	Observe that the order of the matchings $t^i$ is reversed in comparison to the strands in $\sGp$, as we tensor from the left and add strands from the right.
\end{rem}
\begin{lem}
	The maps from \cref{defibimodulemaps} are $\K$-$\K$-bimodule homomorphisms.
\end{lem}
\begin{proof}
	For $\hat{\psi}_{i,j}$ with $j\neq i,i\pm 1$ this is clear from the definition and \cref{figureformapsbetweenthetas}	as it does not change anything regarding the orientability.
	Regarding $\hat{\eta}_i$, $\hat{\epsilon}_i$ and $\hat{\psi}_{i,i+1}$, we remark that we can interpret all these maps as a rotated surgery procedure, see \cref{figuremapsarerotationsofsurgproc} for details.
	But using \cref{surgproccommute} we see that these are actually bimodule homomorphism.
\end{proof}
\begin{figure}[h]
	\centering	
	\begin{tabular}{*{2}{m{0.4\textwidth}}}
		\begin{tikzpicture}[scale=0.4]
			\CUP{1}
			\SETCOORD{1}{0}
			\LINE{-2}{-3}
			\SETCOORD{1}{0}
			\CAP{1}
			\draw[decorate, decoration={snake},->] (2.5,-1.5)--(3.5,-1.5);
			\SETCOORD{2}{0}
			\LINE{0}{3}
			\SETCOORD{1}{0}
			\LINE{0}{-3}
			\SETCOORD{1}{0}
			\LINE{0}{3}
			\node[anchor = south] at (0,0) {$1$};
			\node[anchor = south] at (1,0) {$2$};
			\node[anchor = south] at (2,0) {$3$};
			\node[anchor = north] at (0,-3) {$6$};
			\node[anchor = north] at (1,-3) {$5$};
			\node[anchor = north] at (2,-3) {$4$};
			\node[anchor = south] at (4,0) {$1$};
			\node[anchor = south] at (5,0) {$2$};
			\node[anchor = south] at (6,0) {$3$};
			\node[anchor = north] at (4,-3) {$6$};
			\node[anchor = north] at (5,-3) {$5$};
			\node[anchor = north] at (6,-3) {$4$};
			\node[anchor = north] at (3,-4) {Normal surgery procedure};
		\end{tikzpicture}&
		\begin{tikzpicture}[scale=0.4]
			\CUP{1}
			\SETCOORD{1}{0}
			\LINE{0}{-1.5}
			\CUP{-1}
			\CAP{-1}
			\LINE{0}{-1.5}
			\SETCOORD{1}{0}
			\CAP{1}
			\draw[->] (2.5,-1.5)--(3.5,-1.5);
			\SETCOORD{2}{0}
			\LINE{0}{3}
			\SETCOORD{1}{0}
			\LINE{0}{-3}
			\SETCOORD{1}{0}
			\LINE{0}{3}
			\node[anchor = south] at (0,0) {$1$};
			\node[anchor = south] at (1,0) {$2$};
			\node[anchor = south] at (2,0) {$3$};
			\node[anchor = north] at (0,-3) {$6$};
			\node[anchor = north] at (1,-3) {$5$};
			\node[anchor = north] at (2,-3) {$4$};
			\node[anchor = south] at (4,0) {$1$};
			\node[anchor = south] at (5,0) {$2$};
			\node[anchor = south] at (6,0) {$3$};
			\node[anchor = north] at (4,-3) {$6$};
			\node[anchor = north] at (5,-3) {$5$};
			\node[anchor = north] at (6,-3) {$4$};
			\node[anchor = north] at (3,-4) {$\hat{\eta}_i\colon \hat{G}_{t^{i+1}t^i}\to\K$};
			\useasboundingbox ([shift={(0mm,0mm)}]current bounding box.north east) rectangle ([shift={(0mm,-5mm)}]current bounding box.south west);
		\end{tikzpicture}\\
		\begin{tikzpicture}[scale=0.4]
			\CUP{1}
			\SETCOORD{-2}{0}
			\LINE{0}{-1.5}
			\CUP{1}
			\CAP{1}
			\LINE{0}{-1.5}
			\SETCOORD{-1}{0}
			\CAP{-1}
			\draw[->] (1.5,-1.5)--(2.5,-1.5);
			\SETCOORD{4}{3}
			\CUP{1}
			\SETCOORD{1}{0}
			\LINE{0}{-1.5}
			\CUP{-1}
			\CAP{-1}
			\LINE{0}{-1.5}
			\SETCOORD{1}{0}
			\CAP{1}
			\begin{scope}[xshift=-1cm]
			\node[anchor = south] at (0,0) {$6$};
			\node[anchor = south] at (1,0) {$1$};
			\node[anchor = south] at (2,0) {$2$};
			\node[anchor = north] at (0,-3) {$5$};
			\node[anchor = north] at (1,-3) {$4$};
			\node[anchor = north] at (2,-3) {$3$};
			\node[anchor = south] at (4,0) {$6$};
			\node[anchor = south] at (5,0) {$1$};
			\node[anchor = south] at (6,0) {$2$};
			\node[anchor = north] at (4,-3) {$5$};
			\node[anchor = north] at (5,-3) {$4$};
			\node[anchor = north] at (6,-3) {$3$};
			\end{scope}
			\node[anchor = north] at (2,-4) {$\hat{\psi}_{i,i+1}\colon \hat{G}_{t^it^{i+1}}\to \hat{G}_{t^{i+1}t^i}$};
		\end{tikzpicture}&
		\begin{tikzpicture}[scale=0.4]
			\LINE{0}{3}
			\SETCOORD{1}{0}
			\LINE{0}{-3}
			\SETCOORD{1}{0}
			\LINE{0}{3}
			\draw[->] (2.5,1.5)--(3.5,1.5);
			\SETCOORD{3}{0}
			\CUP{1}
			\SETCOORD{-2}{0}
			\LINE{0}{-1.5}
			\CUP{1}
			\CAP{1}
			\LINE{0}{-1.5}
			\SETCOORD{-1}{0}
			\CAP{-1}
			\begin{scope}[yshift=3cm]
			\node[anchor = south] at (0,0) {$5$};
			\node[anchor = south] at (1,0) {$6$};
			\node[anchor = south] at (2,0) {$1$};
			\node[anchor = north] at (0,-3) {$4$};
			\node[anchor = north] at (1,-3) {$3$};
			\node[anchor = north] at (2,-3) {$2$};
			\node[anchor = south] at (4,0) {$5$};
			\node[anchor = south] at (5,0) {$6$};
			\node[anchor = south] at (6,0) {$1$};
			\node[anchor = north] at (4,-3) {$4$};
			\node[anchor = north] at (5,-3) {$3$};
			\node[anchor = north] at (6,-3) {$2$};
			\end{scope}
			\node[anchor = north] at (3,-1) {$\hat{\epsilon}_i\colon\K\to \hat{G}_{t^{i-1}t^i}$};
		\end{tikzpicture}
	\end{tabular}
	\caption{How to interpret $\hat{\eta}_i$, $\hat{\epsilon}_i$ and $\hat{\psi}_{i,i+1}$ as rotated surgery procedures}\label{figuremapsarerotationsofsurgproc}
\end{figure}
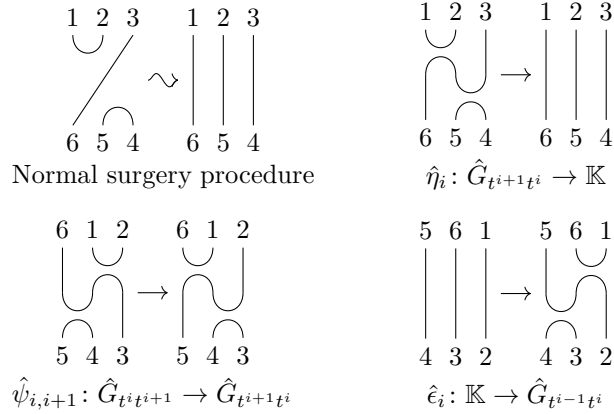
\begin{defi}
	A \emph{stretched cup diagram} is a sequence $a \bm{t}$, where $a$ is a cup diagram and $\bm{t}=t_k\cdots t_1$ is a generalized crossingless matching such that $t_j=t^{i_j}$ are special crossingless matchings for all $j$.
	A \emph{stretched cap diagram} $\bm{t}'a'$ is the horizontal mirror image of a stretched cup diagram $a\bm{t}$ shifted one to the right.
	A \emph{stretched circle diagram} $a\bm{u}\wr\bm{t}'b'$ is a stretched cup diagram $a\bm{u}$ glued underneath a stretched cap diagram $\bm{t}'b'$.
	A \emph{stretched circle diagram} $a\bm{u}\wr\bm{t}'b'$ is called orientable if it contains no circle and no non-propagating line such that both endpoints are on the same side of $0$.
\end{defi}

\begin{defi}
	Denote by $\overline{\iota}$ the unique cap diagram that has no caps.

	Let $\mathcal{F}$ be the full subcategory of $\K$-mod with objects $\hat{\theta}_{i_k}\dots\hat{\theta}_{i_1}\hat{P}(\overline{\iota})$.
	We can consider this category also as a locally finite dimensional locally unital algebra 
	\begin{equation}\label{locfindimalgebraaofK}
		A\coloneqq \bigoplus_{\bm{i}\in\bbZ^k, \bm{j}\in\bbZ^l}\Hom_{\K}(\hat{\theta}_{\bm{i}}\hat{P}(\overline{\iota}), \hat{\theta}_{\bm{j}}\hat{P}(\overline{\iota}))
	\end{equation}
	By \cref{multdescendstoiso} and \cref{hatthetaadjunction} the definition of $\hat{\theta}_i$ respectively $\hat{G}_{\bm{t}}$, the algebra $A$ has a basis given by all orientable stretched circle diagrams of the form $\underline{\iota}\bm{u}\wr\bm{t}'\overline{\iota}$, where $\bm{u}$ and $\bm{t}$ are generalized crossingless matchings build from the special ones.
	And the composition $\underline{\iota}\bm{u}\wr\bm{t}'\overline{\iota}\cdot\underline{\iota}\bm{r}\wr\bm{s}'\overline{\iota}$ is given by $0$ if $\bm{t}\neq\bm{r}$, and otherwise we draw $\underline{\iota}\bm{u}\wr\ured(\bm{t}'\overline{\iota})$ underneath $\lred(\underline{\iota}\bm{t})\wr\bm{s}'\overline{\iota}$ and apply the surgery procedure to eliminate the middle section $\ured(\bm{t}'\overline{\iota})\lred(\underline{\iota}\bm{t})$.
	\begin{defi}
		We define a functor $\Phi\colon\sGp\to\mathcal{F}$ as follows.
		\begin{gather*}
			(i_1, \dots i_k)\mapsto\hat{\theta}_{i_k}\dots\hat{\theta}_{i_1}\hat{P}(\overline{\iota})\\
			\begin{tikzpicture}[line width = \lw]
				\node at (0,0) (A){$a$}; \node at (1,0) (B) {$a-1$};\draw (A)..controls +(0.2,-0.8) and +(-0.2,-0.8)..(B);
			\end{tikzpicture}\mapsto\epsilon_a\qquad\quad
			\begin{tikzpicture}[line width = \lw]
				\node at (0,0) (A){$a$}; \node at (1,0) (B) {$a+1$};\draw (A)..controls +(0.2,0.8) and +(-0.2,0.8)..(B);
			\end{tikzpicture}\mapsto\eta_a\qquad\quad
			\begin{tikzpicture}[line width = \lw]
				\draw (0,0) -- (1,1) (1,0)--(0,1);\node[fill=white, anchor=north] at (0,0) {$a$}; \node[fill=white, anchor=north] at (1,0) {$b$};\node[fill=white, anchor=south] at (0,1) {$b$}; \node[fill=white, anchor=south] at (1,1) {$a$};
			\end{tikzpicture}\mapsto\psi_{b,a}.
		\end{gather*}
	\end{defi}
	The functor $\Phi$ intertwines $\theta_i$ and $\hat{\theta}_i$ by definition.
	The subsequent subsections will prove that this is functor is well-defined.
	Note also that adding a strand with label $i$ in $\sGp$ translates to $\hat{\theta}_i$ under $\Phi$.
\end{defi}
\begin{prop}
	The functor $\Phi$ is well-defined.
\end{prop}
\begin{proof}
	This is \cref{phirespectsrest,phirespectstangle,phirespectsbraid} in the appendix.
\end{proof}
\subsubsection{The functor $\Phi$ is faithful}
In order to show that $\Phi$ is faithful, we need the following technical result.
\begin{prop}\label{udtableauxIFForientablecirclediag}
	The sequence $(i_1, \dots, i_k)$ is a target residue sequence of an up-down tableau of shape $\emptyset$ if and only if $\underline{\iota}t^{i_k}\cdots t^{i_1}\overline{\iota}$ is an orientable generalized circle diagram.
\end{prop}
\begin{proof}
	Let $(i_1, \dots, i_k)$ be a target residue sequence of an up-down tableau of shape $\emptyset$.
	First assume that it does not contain a subsequence $(a,a\pm1,a)$ (i.e.\ a consecutive subsequence up to swapping entries of difference $>1$).
	Suppose furthermore that $\underline{\iota}t^{i_k}\cdots t^{i_1}\overline{\iota}$ is not orientable.
	This means that it either has a circle or a non-propagating line ending on both sides of $0$.
	But as we assumed that we have no subsequence $(a,a\pm1,a)$ we do not have a subpicture of the form (or its vertical mirror image)
	
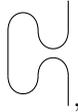
\begin{equation*}
		\begin{tikzpicture}[scale=0.4, baseline=(current bounding box.south)]
			\LINE{0}{-1}
			\CUP{-1}
			\CAP{-1}
			\LINE{0}{-1.5}
			\CUP{1}
			\CAP{1}
			\LINE{0}{-1}
		\end{tikzpicture}\,,
	\end{equation*}
	so a circle would consist of only one cup and cap and the non-propagating line only of one cup or cap.
	Now the circle would give a subsequence $(a,a)$ which cannot happen for residue sequences of up-down tableaux.
	For the non-propagating line we can move the entry corresponding to the cup (resp.\ cap) to the front (resp.\ back) but as both endpoints lie on the same side of $0$ the sequence then either does not start with $0$ or does not end with $-1$.
	But every target residue sequence of an up-down tableau of shape $\emptyset$ has to start with $0$ and end with $-1$.
	Therefore, $\underline{\iota}t^{i_k}\cdots t^{i_1}\overline{\iota}$ is orientable.

	Now we do an induction on the number of subsequences of the form $(a,a\pm1,a)$.
	The paragraph before established the base case.
	Now we will take such a residue sequence $(i_1, \dots, i_k)$ with a subsequence $(a,a\pm1,a)$.
	We will show that after replacing $(a,a\pm1,a)$ by $(a)$ we are still left with a residue sequence of an up-down tableau of shape $\emptyset$.
	Note that this reduction process (up to vertical mirror image) looks like 
	
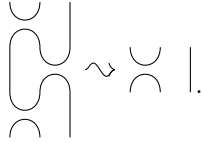
\begin{equation*}
		\begin{tikzpicture}[scale=0.4, baseline={([yshift=0.6cm]current bounding box.south)}]
			\CUP{1}
			\SETCOORD{1}{0}
			\LINE{0}{-1.5}
			\CUP{-1}
			\CAP{-1}
			\LINE{0}{-1.5}
			\CUP{1}
			\CAP{1}
			\LINE{0}{-1.5}
			\SETCOORD{-1}{0}
			\CAP{-1}
			\draw[->, decorate, decoration={snake}] (2.5,-2.25)--(3.5,-2.25);
			\SETCOORD{4}{1.5}
			\CAP{1}\SETCOORD{1}{0}
			\LINE{0}{1.5}
			\SETCOORD{-2}{0}
			\CUP{1}
		\end{tikzpicture}\,.
	\end{equation*}
	So this does not change orientability of the associated stacked circle diagram.
	Therefore, once we proved that the reduced sequence is a residue sequence of $\emptyset$ we know by induction that the associated stacked circle diagram is orientable and then reversing the reduction process there gives still an orientable diagram.
	Hence, it suffices to show that replacing a subsequence of the form $(a,a\pm1,a)$ by $(a)$ gives the target residue sequence of an up-tableau of shape $\emptyset$.
	The following are the possibilities how a subsequence $(a,a-1,a)$ can look in terms of up-down tableau.
	\ytableausetup{boxsize=1.5em}
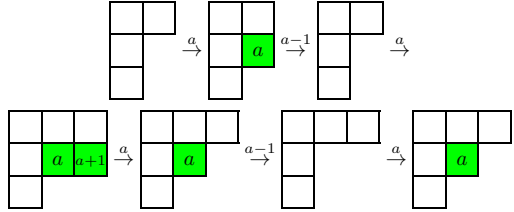
	\begin{gather*}
		\scalebox{0.8}{$\ydiagram{2,1,1}\overset{a}{\rightarrow}\ytableaushort{\none,\none a,\none}*{2,2,1}*[*(green)]{0,1+1}\overset{a-1}{\rightarrow}\ydiagram{2,1,1}\overset{a}{\rightarrow}$}\\
		\scalebox{0.8}{$\ytableaushort{\none,\none a{\scriptstyle{a+1}},\none}*{3,3,1}*[*(green)]{0,1+2}\overset{a}{\rightarrow}\ytableaushort{\none,\none a,\none}*{3,2,1}*[*(green)]{0,1+1}\overset{a-1}{\rightarrow}\ydiagram{3,1,1}\overset{a}{\rightarrow}\ytableaushort{\none,\none a,\none}*{3,2,1}*[*(green)]{0,1+1}$}
	\end{gather*}
	We easily see that in each case we have that $a-1$ removes a box that was either added by the $a$ before or is added again by the $a$ afterwards.
	Thus, replacing this sequence by $a$ we still have a residue sequence of an up-down tableau of shape $\emptyset$.
	Given a subsequence $(a,a+1,a)$ we have the following possibilities.
	\begin{gather*}
		\scalebox{0.8}{$\ydiagram{3,1,1}\overset{a}{\rightarrow}\ytableaushort{\none,\none a,\none}*{3,2,1}*[*(green)]{0,1+1}\overset{a+1}{\rightarrow}\ytableaushort{\none,\none a{\scriptstyle{a+1}},\none}*{3,3,1}*[*(green)]{0,1+2}\overset{a}{\rightarrow}\ytableaushort{\none,\none a,\none}*{3,2,1}*[*(green)]{0,1+1}$}\\
		\scalebox{0.8}{$\ytableaushort{\none,\none {\scriptstyle{a+1}},\none}*{3,2,1}*[*(green)]{0,1+1}\overset{a}{\rightarrow}\ydiagram{3,1,1}\overset{a+1}{\rightarrow}\ytableaushort{\none,\none {\scriptstyle{a+1}},\none}*{3,2,1}*[*(green)]{0,1+1}\overset{a}{\rightarrow}$}
	\end{gather*}
	Similar to before $a+1$ adds a box which is removed by one of the $a$'s. 
	Thus, the reduced sequence is also a target residue sequence of an up-down tableau of shape $\emptyset$.
	This proves then that such a residue sequence gives rise to an oriented stacked circle diagram.

	Now assume that $\underline{\iota}t^{i_k}\cdots t^{i_1}\overline{\iota}$ is orientable.
	We will prove that any diagram without subpictures of the form
	
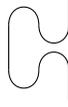
\begin{equation*}
		\begin{tikzpicture}[scale=0.4]
			\LINE{0}{-1}
			\CUP{-1}
			\CAP{-1}
			\LINE{0}{-1.5}
			\CUP{1}
			\CAP{1}
			\LINE{0}{-1}
		\end{tikzpicture}
	\end{equation*}
	or its vertical mirror image gives rise to a target residue sequence of shape $\emptyset$.
	Furthermore, we will prove that given a target residue sequence of an up-down tableau, replacing any entry $(a)$ by $(a,a\pm1,a)$ still gives a target residue sequence of the same shape.

	Now given any diagram, we can use the reduction process from the second paragraph to obtain a picture without these subpictures.
	This is then a target residue sequence by our first claim, and then we can reverse the reduction process and add the sequences $(a,a\pm1,a)$ again and by the second claim this stays in the desired form.

	We will first prove the second claim for $(a,a-1,a)$.
	Suppose we are given such a residue sequence $(i_1, \dots, i_k)$ and let $i_l=a$ be some entry.
	If there exists an entry $a$ that corresponds to adding a box that is not removed until step $l$ we can remove this with $a-1$ and add it again with $a$.
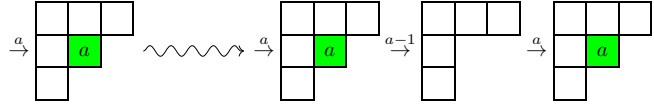
	\begin{equation*}
		\scalebox{0.8}{
		\begin{tikzcd}[ampersand replacement = \&]
			\overset{a}{\rightarrow}\ytableaushort{\none,\none {a},\none}*{3,2,1}*[*(green)]{0,1+1}\arrow[rr, decorate, decoration={snake}]\&\&
			\overset{a}{\rightarrow}\ytableaushort{\none,\none {a},\none}*{3,2,1}*[*(green)]{0,1+1}\overset{a-1}{\rightarrow}\ydiagram{3,1,1}\overset{a}{\rightarrow}\ytableaushort{\none,\none {a},\none}*{3,2,1}*[*(green)]{0,1+1}
		\end{tikzcd}
		}
	\end{equation*}
	If no such entry exists, this means that $a$ removes a box of residue $a+1$ in the first column.
	But in this case we could also add this $a$ and then remove the two boxes with $a-1$ and $a$.
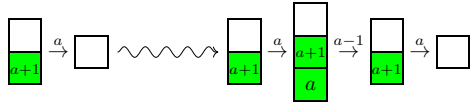
	\begin{equation*}
		\scalebox{0.8}{
			\begin{tikzcd}[ampersand replacement = \&]
				\ytableaushort{\none, {\scriptstyle{a+1}}}*{1,1}*[*(green)]{0,1}\overset{a}{\rightarrow}\ydiagram{1}\arrow[rr,decorate,decoration={snake}]\&\&
				\ytableaushort{\none, {\scriptstyle{a+1}}}*{1,1}*[*(green)]{0,1}\overset{a}{\rightarrow}\ytableaushort{\none, {\scriptstyle{a+1}},a}*{1,1,1}*[*(green)]{0,1,1}\overset{a-1}{\rightarrow}\ytableaushort{\none, {\scriptstyle{a+1}}}*{1,1}*[*(green)]{0,1}\overset{a}{\rightarrow}\ydiagram{1}
			\end{tikzcd}
		}
	\end{equation*}
	Next we prove the second claim for $(a,a+1,a)$.
	Suppose we are given such a residue sequence $(i_1, \dots, i_k)$ and let $i_l=a$ be some entry.
	If after step $l$ there exists an addable box of content $a+1$, we can add this with $a+1$ and remove this with $a$.
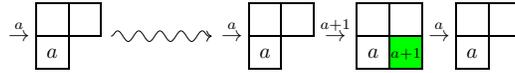
	\begin{equation*}
		\scalebox{0.8}{
		\begin{tikzcd}[ampersand replacement = \&]
			\overset{a}{\rightarrow}\ytableaushort{\none,{a}}*{2,1}\arrow[rr, decorate, decoration={snake}]\&\&
			\overset{a}{\rightarrow}\ytableaushort{\none,{a}}*{2,1}\overset{a+1}{\rightarrow}\ytableaushort{\none,{a}{\scriptstyle{a+1}}}*{2,2}*[*(green)]{0,1+1}\overset{a}{\rightarrow}\ytableaushort{\none,{a}}*{2,1}
		\end{tikzcd}
		}
	\end{equation*}
	If there is no addable box of content $a+1$, this means that $a$ added a box and the situation looks as follows.
	\begin{equation*}
		\scalebox{0.8}
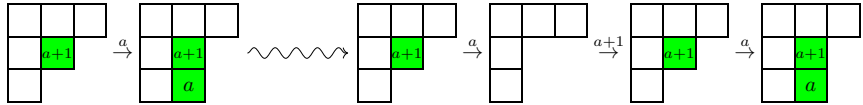
{
		\begin{tikzcd}[ampersand replacement = \&]
			\ytableaushort{\none,\none {\scriptstyle{a+1}},\none}*{3,2,1}*[*(green)]{0,1+1}\overset{a}{\rightarrow}\ytableaushort{\none,\none {\scriptstyle{a+1}},\none {a}}*{3,2,2}*[*(green)]{0,1+1,1+1}\arrow[rr, decorate, decoration={snake}]\&\&
			\ytableaushort{\none,\none {\scriptstyle{a+1}},\none}*{3,2,1}*[*(green)]{0,1+1}\overset{a}{\rightarrow}\ydiagram{3,1,1}\overset{a+1}{\rightarrow}\ytableaushort{\none,\none {\scriptstyle{a+1}},\none}*{3,2,1}*[*(green)]{0,1+1}\overset{a}{\rightarrow}\ytableaushort{\none,\none {\scriptstyle{a+1}},\none {a}}*{3,2,2}*[*(green)]{0,1+1,1+1}
		\end{tikzcd}
		}
	\end{equation*}
	Hence, we have proved the second claim.

	It remains to prove that any orientable diagram $\underline{\iota}t^{i_k}\cdots t^{i_1}\overline{\iota}$ without subpictures of the form
	
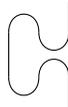
\begin{equation*}
		\begin{tikzpicture}[scale=0.4]
			\LINE{0}{-1}
			\CUP{-1}
			\CAP{-1}
			\LINE{0}{-1.5}
			\CUP{1}
			\CAP{1}
			\LINE{0}{-1}
		\end{tikzpicture}
	\end{equation*}
	or its vertical mirror image gives rise to a target residue sequence of an up-down tableau of shape $\emptyset$.
	In this case we have the following properties of the integer sequence $(i_1, \dots i_k)$.
	\begin{enumerate}
		\item If $i_r=a>0$, then there exists $r'<r$ with $i_{r'}=a-1$.
		\item If $i_r=a<0$, then there exists $r'<r$ with $i_{r'}=a+1$.
		\item If $i_{r'}=i_r=a$ for $r'<r$, then there exists $r'<l,l'<r$ such that $i_l=a+1$ and $i_{l'}=a-1$.
		\item The first two properties hold for the sequence $(i_k+1, i_{k-1}+1, \dots, i_1+1)$.
	\end{enumerate}
	If the first two would not be satisfied we could create a non-propagating line ending at the top with both endpoints on the same side of $0$.
	The last one ensures the same for non-propagating lines ending at the bottom (remember the shift by $-1$ for cup diagrams).
	And the third condition prevents the existence of the subpictures from the top as well as circles.

	All these together imply that $\bm{i}$ is the residue sequence of an $(n+1)\times n$ rectangle for which it is easy to see that this can be interpreted as a residue sequence of an up-down tableau of shape $\emptyset$.
	This is done by adding boxes to reach the partition $(n,n-1,n-2,\dots 1)$ and then removing boxes until one ends up at $\emptyset$.
\end{proof}
\begin{lem}\label{phifaithful}
	The functor $\Phi$ is faithful.
\end{lem}
\begin{proof}
	The category $\sGp$ has a basis given by $\Psi_{\ta{s}\ta{t}}$, where $\ta{s}$ and $\ta{t}$ are two up-down-tableaux of the same shape.
	By adjunction and using that $\Phi$ is compatible with this adjunction, we can assume that $\ta{s}=\emptyset$ is the trivial up-down-tableaux.
	All these basis vectors are build by applying KLR-cups (to distinguish them from the cups in $\K$).
	Using \eqref{tangletwo}, we can get rid of the non-distant crossings that might be involved in the KLR-cups.
	So $\Psi_{\emptyset\ta{t}}$ can be cut into a sequence of distant crossings and KLR-cups on neighbored strands.
	In terms of the stretched circle diagram basis for $A$ from \eqref{locfindimalgebraaofK} this means that we start with the diagram $\underline{\iota}\wr\overline{\iota}$ and successively apply the crossings and KLR-cups.
	But in every step we reach a pair $(\emptyset,\ta{t}')$, where $\ta{t}'$ is an up-down-tableau of shape $\emptyset$.
	By \cref{udtableauxIFForientablecirclediag} this means that in every step we get an orientable stretched circle diagram, and thus by definition of $\epsilon_i$ and $\psi_{i,j}$ the result is nonzero in every step.
	But this means that $\Psi_{\emptyset\ta{t}}$ is mapped to something nonzero and by adjunction this holds for all basis vectors.

	Using that the dimension of morphism spaces in $\sGp$ is $\leq 1$ we conclude that $\Phi$ is faithful.
\end{proof}
\subsubsection{The functor $\Phi$ is an isomorphism of categories}
\begin{thm}\label{phiisomorphism}
	The functor $\Phi$ is an isomorphism of categories.
\end{thm}
\begin{proof}
	By \cref{phifaithful} we know that $\Phi$ is faithful.
	On the other hand \cref{udtableauxIFForientablecirclediag} and the adjunction $(\hat{\theta}_i,\hat{\theta}_{i-1})$ from \cref{hatthetaadjunction} imply $\dim_{\sGp}(\bm{i},\bm{j})=\dim_{\mathcal{F}}(\hat{\theta}_{\bm{i}}\hat{P}(\overline{\iota}),\hat{\theta}_{\bm{j}}\hat{P}(\overline{\iota}))$.
	Furthermore, $\Phi$ is a bijection on objects given by $\bm{i}\mapsto\hat{\theta}_{\bm{i}}\hat{P}(\overline{\iota})$, so $\Phi$ is an isomorphism of categories.
\end{proof}
\section{The algebra \texorpdfstring{$\K_n$}{Kn}: definition and basic properties}\label{kndefisection}
Overall, we want to relate the algebra $\K$ with finite dimensional representations of $\lie{p}(n)$.
As $\K$ is Morita equivalent to $\sGp$, this algebra is a bit too big to be equivalent to $\lie{p}(n)\mmod$.
Therefore, we introduce a new algebra $\K_n$ that will turn out to satisfy our purpose.
This algebra is defined as a quotient of an idempotent truncation of $\K$.
The general idea is to first pick those idempotents that correspond to projective representations of $\lie{p}(n)$ and then quotient out by all the morphisms that vanish for the Lie superalgebra.

With \cref{defidominttocupdiag} we can associate to every ($\rho$-shifted) dominant integral weight $\lambda$ a cup (or cap) diagram with $n$ cups, and we obtain a bijection between dominant integral weights and $\Lambda_n$.
Recall, that a weight is called typical if $\lambda_i-\lambda_{i+1}>1$.
This means for the associated cup (or cap) diagram that there are no nested cups (or caps).
\begin{defi}
	We then define $e\K e$ to be the idempotent truncation of $\K$ at $\Lambda_n$, i.e.\ at all cup diagrams with $n$ cups.
	The algebra $e\K e$ has a basis given by orientable circle diagrams with exactly $n$ cups and caps.

	We denote by $\mathbb{I}_n$ the subspace spanned by all orientable circle diagrams $\underline{\lambda}\overline{\mu}$ with $n$ cups and caps such that there exists at least one non-propagating line.
\end{defi}
\begin{lem}
	The space $\mathbb{I}_n$ is a two-sided ideal of $\K$.
\end{lem}
\begin{proof}
	From \cref{detailedanalysis} it is clear that every surgery procedure either preserves non-propagating lines or produces $0$.
\end{proof}
\begin{lem}\label{nonpropifffactorthroughhigherlayer}
	Given $f\in e\K e$ the following holds:
	\begin{equation}
		f\in \mathbb{I}_n\Rightarrow f\text{ factors through an object }\nu\in I_{n+1}
	\end{equation}
\end{lem}
\begin{proof}
	It suffices to prove this statement for all circle diagrams $\underline{\lambda}\overline{\mu}\in e\K e$ that contain a non-propagating line.
	As the number of cups in $\underline{\lambda}$ is the same as the number of caps in $\overline{\mu}$ we see that we have as many non-propagating lines ending at the bottom as we have at the top.
	So consider a non-propagating line $L$ ending at the bottom and we choose $L$ such that its left endpoint is minimal. 
	This has one cap more than cup.
	Now let $\underline{\lambda'}$ be the cup diagram that is the same as $\underline{\lambda}$ except that the left endpoint of the non-propagating line and the next ray to its left form a cup instead of two rays (see also \cref{figureexlambdaprimenonproplines}).
	As $L$ was chosen minimal, $\underline{\lambda'}\overline{\mu}$ is orientable and it has one more cup than caps.
	Furthermore, $\underline{\lambda}\overline{\lambda'}$ is also orientable as the additional cap connects exactly the same endpoints as $L$ in $\underline{\lambda}\overline{\mu}$ by construction.
	We claim that $\underline{\lambda}\overline{\lambda'}\cdot\underline{\lambda'}\overline{\mu}=\underline{\lambda}\overline{\mu}$.
	As $\underline{\lambda}$ and $\underline{\lambda'}$ agree up to one cup, we see that $\underline{\lambda}\overline{\lambda'}$ is very close to the idempotent $e_\lambda$.
	We may choose the surgery procedure coming from the additional cup as the first one.
	But this surgery procedure produces by definition of $\underline{\lambda'}$ the diagram $\underline{\lambda}\overline{\lambda}$ drawn underneath $\underline{\lambda}\overline{\mu}$ which is clearly orientable (see also \cref{exampleformultiplicationlamdaprimenonproplines}).
	After this we essentially compute $e_\lambda\cdot\underline{\lambda}\overline{\mu}$ which gives clearly $\underline{\lambda}\overline{\mu}$.
	As $\underline{\lambda'}\in I_{n+1}$, this proves the lemma.
	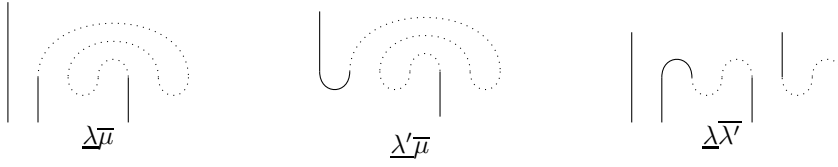
\begin{figure}
		\centering
		\begin{tabular}{m{0.3\textwidth}m{0.3\textwidth}m{0.3\textwidth}}
		\begin{tikzpicture}[scale=0.4]
			\LINE{0}{4}
			\SETCOORD{1}{-4}
			\LINE{0}{1.5}
			\CAP[dotted]{5}
			\CUP[dotted]{-1}
			\CAP[dotted]{-3}
			\CUP[dotted]{1}
			\CAP[dotted]{1}
			\LINE{0}{-1.5}
			\node[anchor=south] at (3,-1.2) {$\underline{\lambda}\overline{\mu}$};
		\end{tikzpicture}&
		\begin{tikzpicture}[scale=0.4]
			\SETCOORD{0}{4}
			\LINE{0}{-2}
			\CUP{1}
			\CAP[dotted]{5}
			\CUP[dotted]{-1}
			\CAP[dotted]{-3}
			\CUP[dotted]{1}
			\CAP[dotted]{1}
			\LINE{0}{-1.5}
			\node[anchor=south] at (3,-1.2) {$\underline{\lambda'}\overline{\mu}$};
		\end{tikzpicture}&
		\begin{tikzpicture}[scale=0.4]
			\draw[white] (0,0)--(0,4);
			\LINE{0}{3}
			\SETCOORD{1}{-3}
			\LINE{0}{1.5}
			\CAP{1}
			\CUP[dotted]{1}
			\CAP[dotted]{1}
			\LINE{0}{-1.5}
			\SETCOORD{3}{0}
			\LINE{0}{1.5}
			\CAP[dotted]{-1}
			\CUP[dotted]{-1}
			\LINE{0}{1.5}
			\node[anchor=south] at (3,-1.2) {$\underline{\lambda}\overline{\lambda'}$};
		\end{tikzpicture}
	\end{tabular}
	\caption{Example of how $\lambda'$ is constructed.}
	\label{figureexlambdaprimenonproplines}
	\end{figure}
	\begin{figure}
		\centering
		\begin{tikzpicture}[scale=0.4]
			\LINE{0}{-1.5}
			\CUP{1}
			\CAP[dotted]{5}
			\CUP[dotted]{-1}
			\CAP[dotted]{-3}
			\CUP[dotted]{1}
			\CAP[dotted]{1}
			\LINE{-4}{-2.5}
			\LINE{0}{-1.5}
			\SETCOORD{1}{0}
			\LINE{0}{1.5}
			\CAP{1}
			\CUP[dotted]{1}
			\CAP[dotted]{1}
			\LINE{0}{-1.5}
			\SETCOORD{3}{5.5}
			\LINE{0}{-1.5}
			\LINE{-2}{-2.5}
			\CUP[dotted]{1}
			\CAP[dotted]{1}
			\LINE{0}{-1.5}
			\draw[->, decorate, decoration={snake}] (7.5,-2.75)--(8.5,-2.75);
			\SETCOORD{2}{5.5}
			\LINE{0}{-5.5}
			\SETCOORD{1}{0}
			\LINE{0}{4}
			\CAP[dotted]{5}
			\CUP[dotted]{-1}
			\CAP[dotted]{-3}
			\CUP[dotted]{1}
			\CAP[dotted]{1}
			\LINE{-2}{-2.5}
			\CUP[dotted]{1}
			\CAP[dotted]{1}
			\LINE{0}{-1.5}
			\SETCOORD{3}{5.5}
			\LINE{0}{-1.5}
			\LINE{-2}{-2.5}
			\CUP[dotted]{1}
			\CAP[dotted]{1}
			\LINE{0}{-1.5}
		\end{tikzpicture}
		\caption{Example for the first surgery procedure of $\underline{\lambda}\overline{\lambda'}\cdot\underline{\lambda'}\overline{\mu}$}
		\label{exampleformultiplicationlamdaprimenonproplines}
	\end{figure}
\end{proof}
\begin{defi}
	We define the algebra $\K_n$ to be $e\K e/\mathbb{I}_n$.

	When talking about orientable circle diagrams in the context of $\K_n$ we mean orientable circle diagrams for $\K$ without any non-propagating lines.
\end{defi}
\begin{rem}
	For $\K_n$ the multiplication is given exactly as for $\K$ but we declare the result to be $0$ whenever a non-propagating line appears.
	Thus, comparing with \cref{detailedanalysis} we see that the split and reconnect produce $0$ in every case and straightening produces something non-zero.
\end{rem}
\begin{lem}\label{KNisessfinite}
	The algebra $\K_n$ is an essentially finite locally unital algebra $\K$ with idempotents $e_\lambda$ indexed by $\lambda\in\Lambda_n$.
\end{lem}
\begin{proof}
	By \cref{klocuntial} $K$ is a locally unital algebra, so clearly $\K_n$ is as well.
	It is essentially finite dimensional, i.e.\ $\dim\K_ne_\lambda<\infty$, as we do not allow non-propagating lines.
	This means that if we fix a cap diagram $\overline{\lambda}$ and $\underline{\mu}\overline{\lambda}$ is orientable, then all the cups of $\underline{\mu}$ are close to the caps of $\overline{\lambda}$.
	In particular every cup diagram of this form fits into a finite strip where it is non-trivial, thus $\dim\K_ne_\lambda<\infty$.
	And similarly also $\dim e_\lambda\K_n<\infty$.
\end{proof}
\begin{defi}
	We define an anti-involution $\ast$ on $\K_n$, which is given by rotating a circle diagram around $\frac{1}{2}$.
\end{defi}
From now on we will assume that any $\K_n$-module $M$ is compatible with the locally unital structure, i.e.\ $M=\bigoplus_{\lambda\in\Lambda_n}e_\lambda M$.
\subsection{A triangular basis and quasi-hereditaryness}
In this section we explore the structure of $\K_n$ further. 
In particular, we will prove that $\K_n$ is an essentially finite based quasi-hereditary algebra in the sense of \cite{BS21}.
\subsubsection{\texorpdfstring{$\Delta$ and $\nabla$}{Delta and Nabla} orientations}
We introduce orientations of circle diagrams that allow us to split our diagrams and provide then some cellular like properties.
These additional orientations behave very much like orientations for the Khovanov algebras of type $A$ respectively $B$, see \cite{BS11} for type $A$ and \cite{ES16a} for type $B$.
The main difference is that for type $P$ the orientations are not needed to define the algebra structure, they are only important for the more involved structure.
\begin{defi}
	A $\Delta$-orientation of a cap diagram $\overline{\lambda}$ associates to each cap with endpoints $a<b\in\bbZ+\frac{1}{2}$ an integer $k$ such that $a<k<b$ and for every other cap with endpoints $a'<k<b'$ we have $a'<a<b<b'$.

	A $\nabla$-orientation of a cup diagram $\overline{\lambda}$ associates to each cup with endpoints $a<b\in\bbZ+\frac{1}{2}$ an integer $k\in\{a-\frac{1}{2},b+\frac{1}{2}\}$.

	Furthermore, we require for both orientations that the integers are pairwise distinct.
	We draw this as $n$ black dots at the corresponding positions in the cap (resp.\ cup) diagram.
\end{defi}
\begin{rem}
	One might argue that \enquote{orientation} is not a descriptive name for this.
	However, we chose this wording (similar to surgery procedures) as these mimic the orientations in \cite{BS11a} and \cite{ES16a} for the Khovanov algebras of other types.
	In their setup, these orientations provide filtrations of indecomposable projective modules via standard and costandard modules.
	Our orientations play the exact same role as shown in \cref{knisquher} below.
\end{rem}
\begin{rem}\label{remorientationsareweights}
	We can think of $\Delta$- and $\nabla$-orientations also as integral dominant weights for $\lie{p}(n)$.
	If the set $\{k_1>\dots> k_n\}$ describes the orientation we can associate the integral dominant weight $k_1-1>\dots> k_n-1$.
	
	So cup diagrams are shifted by $\frac{1}{2}$, orientations by $1$ and cap diagrams by $\frac{3}{2}$.
\end{rem}
\begin{rem}
	Spelling out the definition of $\Delta$-orientation, we associate to each cap a point that lies below this cap but not below any other cap inside this cap.
	Informally speaking, we move the point from the cap somewhere to the inside.
	Especially if we have a small cap there is only one point we can associate.
	
	So for a typical cap diagram there is only one $\Delta$-orientation.
\end{rem}
\begin{rem}\label{alternativenablaorientation}
	Instead of associating the integer $a-\frac{1}{2}$ or $b+\frac{1}{2}$ to a cup, we could also demand for an integer $k$ such that $k\notin[a,b]$ and for another cup (here two neighbored rays are also considered as a cup) with endpoints $a'<a<b<b'$ we have $a'<k<b'$.
	This allows a priori for more possibilities for each cup diagram, but the interested reader can easily verify that these two definitions are equivalent.
	This other definition resembles more the definition of a $\Delta$-orientation meaning that we can move every point to the outside but not outside another cup, but the original one is more practical.
\end{rem}
\begin{defi}
	An orientation of an orientable circle diagram $\underline{\lambda}\overline{\mu}$ is a $\Delta$-ori\-en\-ta\-tion of $\overline{\mu}$ and a $\nabla$-orientation $\underline{\lambda}$ such that the sets of the associated integers agree. 
	Using \cref{remorientationsareweights} we also write $\underline{\lambda}\nu\overline{\mu}$ if the orientation of $\underline{\lambda}\overline{\mu}$ is exactly $\nu$.
\end{defi}
\begin{lem}\label{circlediaghasuniqueorientation}
	Every orientable circle diagram admits a unique orientation.
\end{lem}
\begin{proof}
	We prove this via induction on the number $k$ of caps, with $k=0$ being trivial.
	
	So let $c$ be an orientable circle diagram.
	If $k>0$, there exists a small cap, which is connected in one of the following ways.
	
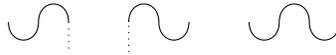
\begin{equation*}
		\begin{tikzpicture}[scale=0.4]
			\CUP{1}\CAP{1}\LINE[dotted]{0}{-1}
			\SETCOORD{2}{0}
			\LINE[dotted]{0}{1}\CAP{1}\CUP{1}
			\SETCOORD{2}{0}
			\CUP{1}\CAP{1}\CUP{1}
		\end{tikzpicture}
	\end{equation*}
	The dotted ray indicates that it is either a ray or a cup that contains the undotted cup.
	Any orientation of $c$ necessarily needs to have a black dot at the unique position inside this small cap, which needs to be associated to one of the undotted cups in the above picture.
	In the first two cases we can remove the cup/cap pair and in the last we replace the picture with a cup.
	In all three cases we obtain an orientable diagram with one cap less.
	By induction, we now that this admits a unique orientation.
	 
	In the first two cases we directly get a unique orientation for $c$ by putting the cup/cap pair back and adding a dot below the small cap.
	In the third case note that the new cup that we added necessarily has a dot either to its left or right.
	If it is on the left, we associate this with the left cup and the dot inside the small cap with the right cup and vice versa.
	In any case, we can build a unique orientation out of the smaller diagram.
\end{proof}
\subsubsection{A triangular basis}
\begin{defi}
	For $\lambda,\mu\in\Lambda_n$ we define the sets $X(\lambda, \mu)$ and $Y(\lambda, \mu)$ as
	\begin{align*}
		X(\lambda, \mu)&\coloneqq \begin{cases}\{\underline{\lambda}\lambda\overline{\mu}\}&\text{if $\lambda$ is a valid $\Delta$-orientation of $\overline{\mu}$,}\\\emptyset&\text{otherwise,}\end{cases}\\
		Y(\lambda, \mu)&\coloneqq \begin{cases}\{\underline{\lambda}\mu\overline{\mu}\}&\text{if $\mu$ is a valid $\nabla$-orientation of $\underline{\lambda}$,}\\\emptyset&\text{otherwise.}\end{cases}
	\end{align*}
	We also set $Y(\lambda)\coloneqq \bigcup_{\mu\in\Lambda_n}Y(\mu,\lambda)$ and $X(\lambda)\coloneqq \bigcup_{\mu\in\Lambda_n}X(\lambda,\mu)$.
\end{defi}
\begin{lem}\label{matchingorientationsgivescircdiag}
	If $\nu$ is a valid $\Delta$-orientation of $\overline{\mu}$ and a valid $\nabla$-orientaion of $\underline{\lambda}$, then $\underline{\lambda}\overline{\mu}$ is orientable.
\end{lem}
\begin{proof}
	Suppose that $\underline{\lambda}\overline{\mu}$ is not orientable.
	Then it contains either a non-propagating line ending at the bottom or a circle.
	In both cases look at the $k$ caps of this component. 
	As $\nu$ is a valid $\Delta$-orientation of $\overline{\mu}$ we have $k$ dots in between these $k$ caps.
	But this means that we also have $k$ dots in $k$ cups if we have a circle or $k$ dots for $k-1$ cups bounded by rays if we have a non-propagating line ending at the bottom.
	But in either case $\nu$ is not a valid $\nabla$-orientation of $\underline{\lambda}$.
\end{proof}
\begin{thm}\label{thmpropbasedquher}
	The following properties hold.
	\begin{enumerate}
		\item\label{thmpropbasedquheri} The products $(yx)$ for $(y,x)\in\sqcup_{\lambda\in\Lambda_n}Y(\lambda)\times X(\lambda)$ form a basis of $\K_n$.
		\item\label{thmpropbasedquherii} For $\lambda,\mu\in\Lambda_n$, the sets $Y(\mu,\lambda)$ and $X(\lambda, \mu)$ are empty unless $\mu\leq\lambda$.
		\item\label{thmpropbasedquheriii} We have that $Y(\lambda,\lambda)=X(\lambda,\lambda)=\{e_\lambda\}$ for each $\lambda\in\Lambda_n$.
	\end{enumerate}
\end{thm}
\begin{proof}
For \cref{thmpropbasedquherii} note that $Y(\mu, \lambda)$ (resp.\ $X(\lambda,\mu)$) is non-empty only if $\lambda$ is a $\nabla$-orientation (resp.\ $\Delta$-orientation) of $\underline{\mu}$ (resp.\ $\overline{\mu}$).
Observe that if one always puts the dot at the rightmost position possible one obtains the $\nabla$-orientation (resp.\ $\Delta$-orientation) $\mu$.
All the others are obtained from this by moving single dots to the left, which makes the weight bigger in our order from \cref{defipartialordering}.

For \cref{thmpropbasedquheriii} observe that $\lambda$ is clearly a $\Delta$-orientation of $\overline{\lambda}$ and a $\nabla$-orientation of $\underline{\lambda}$.

Thus, it only remains to proof \cref{thmpropbasedquheri}.
We will prove that if $\underline{\lambda}\nu\overline{\mu}$ is an orientation, then $\underline{\lambda}\overline{\mu}=\underline{\lambda}\overline{\nu}\cdot\underline{\nu}\overline{\mu}$.
We clearly have $\underline{\lambda}\overline{\nu}\in Y(\lambda,\nu)$ and $\underline{\nu}\overline{\mu}\in X(\nu,\mu)$ and together with \cref{circlediaghasuniqueorientation} this proves \cref{thmpropbasedquheri}.

Denote by $\underline{\nu_k'}$ the cup diagram, where we remove the $k$ cups of $\underline{\lambda}$ which are associated to the $k$ leftmost dots of $\nu$.
Then let $\underline{\nu_k}$ be obtained from $\underline{\nu_k'}$ by adding $k$ cups such that their endpoints correspond to the $k$ leftmost dots of $\nu$.
In other words $\nu_0=\lambda$ and $\nu_n=\nu$.
By construction, we know that $\nu$ is a $\nabla$-orientation of $\underline{\nu_k}$.
Therefore, $\underline{\nu_k}\overline{\mu}$ is orientable by \cref{matchingorientationsgivescircdiag}.
Furthermore, by construction we also have that $\nu_k$ is a $\nabla$-orientation of $\underline{\lambda}$ and hence $\underline{\lambda}\overline{\nu_k}$ is orientable as well.

Now we compute some surgery procedures for $\underline{\lambda}\overline{\nu_k}\cdot\underline{\nu_k}\overline{\mu}$.
Refer for this also to \cref{orientationsforKn}, where we describe explicitly how orientations behave under multiplication.
When multiplying these diagrams we have $n$ surgery procedures in total, $n-k$ of them come from $\lambda$ and the last $k$ from $\nu$.
By construction, we may choose to apply the $n-k$ procedures for $\lambda$ first.
As stated before $\nu_k$ is an orientation of $\underline{\lambda}\overline{\nu_k}$ and the $n-k$ rightmost dots in $\nu_k$ all also appear in $\lambda$.
But this means that the first $n-k$ surgery procedures are similar to multiplying with $e_{\nu_k}$ instead of $\underline{\lambda}\overline{\nu_k}$ and each of these looks just as in the proof of \cref{multwithidempotents}, i.e.\ all theses are straightenings.
Thus, none of these change orientability and after $n-k$ steps we have reached an orientable diagram.
But this diagram is exactly the diagram that we obtain after $n-k$ surgery procedures (for one specific order) of the multiplication $\underline{\lambda}\overline{\nu}\cdot\underline{\nu}\overline{\mu}$.
Therefore, with the correct order (namely from right to left) of surgery procedures we see that after applying each of them we receive an orientable diagram.
Thus, by definition of the multiplication we get $\underline{\lambda}\overline{\mu}=\underline{\lambda}\overline{\nu}\cdot\underline{\nu}\overline{\mu}$.
\end{proof}
\begin{cor}\label{knisquher}
	The algebra $\K_n$ is an essentially finite based quasi-hereditary algebra in the sense of \cite{BS21}.
\end{cor}
\begin{proof}
	This is immediate from \cref{KNisessfinite} and \cref{thmpropbasedquher}.
\end{proof}
\subsubsection{An alternative description of $\K_n$ using orientations}\label{orientationsforKn}
Alternatively we can describe $\K_n$ as the algebra with a basis given by all oriented circle diagrams $\underline{\lambda}\nu\overline{\mu}$.
And the multiplication is given by exactly the same procedure we just have to define what happens with the orientation during a surgery procedure.

In order to describe this, we first define what such an orientation ought to be.
\begin{defi}
	Suppose we are given a generalized stacked circle diagram $\underline{\lambda}\bm{t}\overline{\mu}$ with $t=t_k\cdots t_1$.
	An orientation of this diagram is a sequence of weights $\nu_k, \dots, \nu_0$ such that 
	\begin{enumerate}
		\item $\nu_k$ is a $\nabla$-orientation of $\underline{\lambda}$,
		\item $\nu_0$ is a $\Delta$-orientation of $\overline{\mu}$,
		\item and every of $\nu_it_i\nu_{i-1}$ is oriented (where we think of the orientations $\nu_i$ as dots), by which we mean
		\begin{itemize} 
			\item every cap has associated a unique dot of $\lambda_i$ inside it that is not contained in any other cap,
			\item every cup has associated a unique dot of $\lambda_{i-1}$ directly to its left or right,
			\item and every other dot of $\lambda_i$ that is not contained in any cap lies in a region bounded by two rays.
				For each of these dots, there exists a unique dot of $\lambda_{i-1}$ that is in the same region and not associated to a cup.
		\end{itemize}
	\end{enumerate}
\end{defi}
\begin{rem}
	Associating the dots to cups and or caps is not part of the data. 
	We require only the existence of such a pairing.

	In \cref{exorientationscrossmatching} it is easy to see that there are many possibilities to define this pairing, but we consider this to be only one orientation.
\end{rem}
\begin{ex}
	\Cref{exorientationscrossmatching} provides examples for some orientations of crossingless matchings.
	\begin{figure}[h]
		\centering
		\begin{tikzpicture}[scale=0.5]
			\CUP{1}
			\SETCOORD{1}{0}
			\CUP{-3}
			\SETCOORD{-2}{0}
			\CUP{-1}
			\SETCOORD{-1}{0}
			\LINE{0}{-3.5}
			\SETCOORD{1}{0}
			\LINE{2}{3.5}
			\SETCOORD{5}{0}
			\LINE{0}{-3.5}
			\SETCOORD{-1}{0}
			\CAP{-5}
			\SETCOORD{1}{0}
			\CAP{1}
			\SETCOORD{1}{0}
			\CAP{1}
			\node at (-0.5,0) {$\bullet$};
			\node at (2.5,0) {$\bullet$};
			\node at (-1.5,0) {$\bullet$};
			\node at (-2.5,0) {$\bullet$};
			\node at (-4.5,0) {$\bullet$};
			\node at (0.5,-3.5) {$\bullet$};
			\node at (-0.5,-3.5) {$\bullet$};
			\node at (-1.5,-3.5) {$\bullet$};
			\node at (-3.5,-3.5) {$\bullet$};
			\node at (-4.5,-3.5) {$\bullet$};
		\end{tikzpicture}
		\caption{Examples of orientations of crossingless matchings}
		\label{exorientationscrossmatching}
\end{figure}
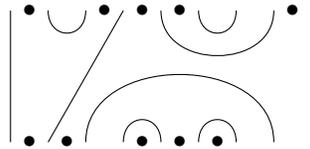
\end{ex}
\begin{lem}
	Any orientable diagram of the form $\underline{\lambda}\bm{t}\overline{\mu}$ admits a unique orientation (and if it admits an orientation it is orientable).
\end{lem}
\begin{proof}
	This is proven by exactly the same argument as for \cref{circlediaghasuniqueorientation}.
\end{proof}
Observe that all diagrams that are obtained using surgery procedures are of the form $\underline{\lambda}t\overline{\mu}$ for some crossingless matching $t$.
We have now all the ingredients to describe how orientations behave under surgeries.
If the surgery we apply is a split or a reconnect the result is $0$, so we only have to look at straightening.
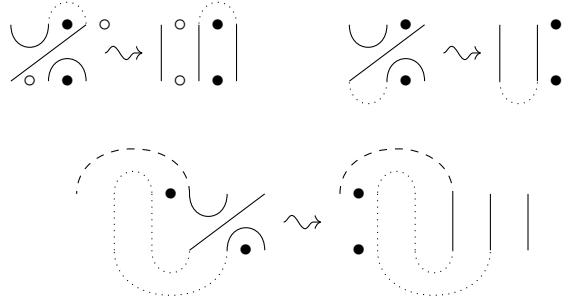
\begin{figure}[h]
	\centering
	 \begin{tikzpicture}[scale=0.5]
			\CUP{1}
			\CAP[dotted]{1}
			\LINE{-2}{-1.5}
			\SETCOORD{1}{0}
			\CAP{1}
			\SETCOORD{9}{0}
			\CAP{-1}
			\CUP[dotted]{-1}
			\LINE{2}{1.5}
			\SETCOORD{-1}{0}
			\CUP{-1}
			\draw[->, decorate, decoration={snake}] (2.5, -0.75) -- (3.5, -0.75);
			\draw[->, decorate, decoration={snake}] (11.5, -0.75) -- (12.5, -0.75);
			\node at (1.5,0) {$\bullet$};
			\node at (1.5,-1.5) {$\bullet$};
			\node at (2.5,0) {$\circ$};
			\node at (0.5,-1.5) {$\circ$};
			\node at (5.5,0) {$\bullet$};
			\node at (5.5,-1.5) {$\bullet$};
			\node at (4.5,0) {$\circ$};
			\node at (4.5,-1.5) {$\circ$};
			\node at (10.5,-1.5) {$\bullet$};
			\node at (10.5,0) {$\bullet$};
			\node at (14.5,-1.5) {$\bullet$};
			\node at (14.5,0) {$\bullet$};
			\SETCOORD{4}{0}
			\LINE{0}{-1.5}
			\CUP[dotted]{1}
			\LINE{0}{1.5}
			\SETCOORD{1}{0}
			\LINE{0}{-1.5}
			\SETCOORD{-9}{0}
			\LINE{0}{1.5}
			\CAP[dotted]{-1}
			\LINE{0}{-1.5}
			\SETCOORD{-1}{0}
			\LINE{0}{1.5}
			
			\SETCOORD{2.75}{-6}
			\CAP{-1}
			\CUP[dotted]{-3}
			\LINE[dotted]{0}{1.5}
			\CAP[dotted]{1}
			\LINE[dotted]{0}{-1.5}
			\CUP[dotted]{1}
			\LINE{2}{1.5}
			\SETCOORD{-1}{0}
			\CUP{-1}
			\CAP[dashed]{-3}
			\SETCOORD{7}{0}
			\CAP[dashed]{3}
			\LINE{0}{-1.5}
			\CUP[dotted]{-1}
			\LINE[dotted]{0}{1.5}
			\CAP[dotted]{-1}
			\LINE[dotted]{0}{-1.5}
			\CUP[dotted]{3}
			\LINE{0}{1.5}
			\SETCOORD{1}{0}
			\LINE{0}{-1.5}
			\node at (9.25,-6) {$\bullet$};
			\node at (6.25,-6) {$\bullet$};
			\node at (4.25,-4.5) {$\bullet$};
			\node at (9.25,-4.5) {$\bullet$};
			\draw[->, decorate, decoration={snake}] (7.25, -5.25) -- (8.25, -5.25);
											\end{tikzpicture}
		\caption{How orientations behave under straightening}
		\label{figureStraighteningwithorient}
\end{figure}
Note that by definition of orientation these are all the cases that appear.
In every case we move the dots as depicted in \cref{figureStraighteningwithorient} (the white dots may or may not be there)

This preserves the orientation as the resulting dots lie in the same region.

After every surgery procedure we have a trivial middle section and the two orientations agree.
Then we collapse the middle section and identify the two orientations.
\begin{ex}
	Let $n=3$ and $\lambda=(1,0,-1)$, $\mu=(4,3,-1)$ and $\nu=(3,2,-1)$. 
	We then compute in \cref{exmultwithoris} $\underline{\nu}\overline{\mu}\cdot\underline{\mu}\overline{\lambda}$ including the orientations.
	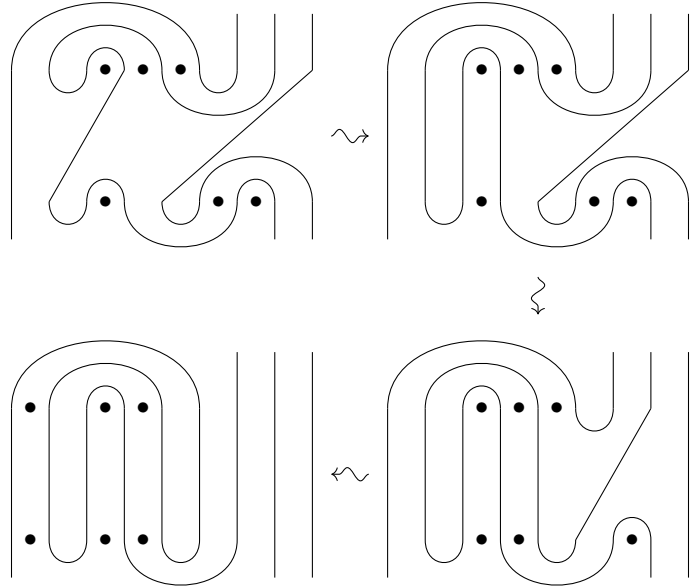
\begin{figure}[h]
		\centering
		\begin{tikzpicture}[scale=0.5]
			\LINE{0}{4.5}
			\CAP{5}
			\CUP{1}
			\LINE{0}{1.5}
			\SETCOORD{1}{0}
			\LINE{0}{-1.5}
			\CUP{-3}
			\CAP{-3}
			\CUP{1}
			\CAP{1}
			\LINE{-2}{-3.5}
			\CUP{1}
			\CAP{1}
			\CUP{3}
			\CAP{1}
			\LINE{0}{-1}
			\SETCOORD{1}{0}
			\LINE{0}{1}
			\CAP{-3}
			\CUP{-1}
			\LINE{4}{3.5}
			\LINE{0}{1.5}
			\node at (2.5,1) {$\bullet$};
			\node at (5.5,1) {$\bullet$};
			\node at (6.5,1) {$\bullet$};
			\node at (2.5,4.5) {$\bullet$};
			\node at (3.5,4.5) {$\bullet$};
			\node at (4.5,4.5) {$\bullet$};
			\draw[->, decorate, decoration={snake}] (8.5,2.75)--(9.5,2.75); 
			\SETCOORD{2}{-6}
			\LINE{0}{4.5}
			\CAP{5}
			\CUP{1}
			\LINE{0}{1.5}
			\SETCOORD{1}{0}
			\LINE{0}{-1.5}
			\CUP{-3}
			\CAP{-3}
			\LINE{0}{-3.5}
			\CUP{1}
			\LINE{0}{3.5}
			\CAP{1}
			\LINE{0}{-3.5}
			\CUP{3}
			\CAP{1}
			\LINE{0}{-1}
			\SETCOORD{1}{0}
			\LINE{0}{1}
			\CAP{-3}
			\CUP{-1}
			\LINE{4}{3.5}
			\LINE{0}{1.5}
			\node at (12.5,1) {$\bullet$};
			\node at (15.5,1) {$\bullet$};
			\node at (16.5,1) {$\bullet$};
			\node at (12.5,4.5) {$\bullet$};
			\node at (13.5,4.5) {$\bullet$};
			\node at (14.5,4.5) {$\bullet$};
			\draw[->, decorate, decoration={snake}] (14, -1) -- (14, -2);
			\SETCOORD{-8}{-15}
			\LINE{0}{4.5}
			\CAP{5}
			\CUP{1}
			\LINE{0}{1.5}
			\SETCOORD{1}{0}
			\LINE{0}{-1.5}
			\LINE{-2}{-3.5}
			\CUP{-1}
			\LINE{0}{3.5}
			\CAP{-3}
			\LINE{0}{-3.5}
			\CUP{1}
			\LINE{0}{3.5}
			\CAP{1}
			\LINE{0}{-3.5}
			\CUP{3}
			\CAP{1}
			\LINE{0}{-1}
			\SETCOORD{1}{0}
			\LINE{0}{6}
			\node at (12.5,-8) {$\bullet$};
			\node at (13.5,-8) {$\bullet$};
			\node at (16.5,-8) {$\bullet$};
			\node at (12.5,-4.5) {$\bullet$};
			\node at (13.5,-4.5) {$\bullet$};
			\node at (14.5,-4.5) {$\bullet$};
			\draw[->, decorate, decoration={snake}] (9.5, -6.25) -- (8.5, -6.25);
			\SETCOORD{-18}{-6}
			\LINE{0}{4.5}
			\CAP{5}
			\LINE{0}{-3.5}
			\CUP{-1}
			\LINE{0}{3.5}
			\CAP{-3}
			\LINE{0}{-3.5}
			\CUP{1}
			\LINE{0}{3.5}
			\CAP{1}
			\LINE{0}{-3.5}
			\CUP{3}
			\LINE{0}{5}
			\SETCOORD{1}{0}
			\LINE{0}{-6}
			\SETCOORD{1}{0}
			\LINE{0}{6}
			\node at (2.5,-8) {$\bullet$};
			\node at (3.5,-8) {$\bullet$};
			\node at (0.5,-8) {$\bullet$};
			\node at (2.5,-4.5) {$\bullet$};
			\node at (3.5,-4.5) {$\bullet$};
			\node at (0.5,-4.5) {$\bullet$};
		\end{tikzpicture}\label{exmultwithoris}
		\caption{Example of multiplication including orientations}
	\end{figure}
\end{ex}
\section{Geometric bimodules and adjunction}\label{geombimodknsection}
This section adapts the definition of geometric bimodules for $\K$ to the subquotient $\K_n$ and proving similar results as in \cref{geombimodK}.
\subsection{Definition}
\begin{defi}
	For a generalized crossingless matching $\bm{t}$ we can define the $\K_n$-$\K_n$-bimdoule $G_{\bm{t}}\coloneqq e\hat{G}_{\bm{t}}e/\mathbb{I}_{\bm{t}}$ where $\mathbb{I}_{\bm{t}}$ is the subvectorspace spanned by all generalized circle diagrams with non-propagating lines.
	That this is indeed a two-sided submodule can be verified analogously to \cref{nonpropifffactorthroughhigherlayer}.
	
	We then define $\theta_{\bm{t}}$ as tensoring over $\K_n$ with $G_{\bm{t}}$ and if $\bm{t}=t^i$ we also write $\theta_i$ for $\theta_{t^i}$ and call these $\theta_{\bm{t}}$ \emph{projective functors}.
\end{defi}
\begin{rem}
	As in \cref{bimoduleisoreductiontosinglelayer}, we still have $G_{\bm{t}}\cong G_{\red(\bm{t})}$, thus we restrict ourselves from now on to crossingless matchings of one layer.
	\end{rem}	
Observe that for these bimodules the appropriate replacement for \eqref{geombimodassoc} also holds.
On the other hand given a generalized crossingless matching $\bm{t}$ we define $\bm{t}^\ast$ as the rotation of $\bm{t}$ around $\frac{1}{2}$.
And we can define a linear map $\ast\colon \hat{G}_{\bm{t}}\to \hat{G}_{\bm{t}^\ast},a\bm{t}b\mapsto b^\ast\bm{t}^\ast a^\ast$, i.e.\ $\ast$ is given by rotating the total picture around $\frac{1}{2}$.
Then the map $\ast$ is anti-multiplicative in the sense that
\begin{equation}\label{geombimodantimultnuclear}
	\begin{tikzcd}
		\hat{G}_{\bm{t}}\otimes \hat{G}_{\bm{u}}\arrow[d, "m"]\arrow[r, "P\circ(\ast\otimes\ast)"]&\hat{G}_{\bm{u}^\ast}\otimes \hat{G}_{\bm{t}^\ast}\arrow[d, "m"]\\
		\hat{G}_{\bm{t}\bm{u}}\arrow[r, "\ast"]&\hat{G}_{\bm{u}^\ast\bm{t}^\ast}
	\end{tikzcd}
\end{equation}
commutes, where $P$ is the normal flip map.
\subsection{Adjunction}
The next proposition states the equivalent of \cref{geombimodonproj} for $\K_n$.
\begin{prop}\label{actionofgeombimodonprojnuclear}
	We have 
	\begin{equation*}
		\theta_{t}P(\overline{\gamma})=\begin{cases}
			0&\text{if $t\overline{\gamma}$ is not orientable,}\\
			P(\overline{\nu})&\text{otherwise, where $\overline{\nu}=\ured(t\overline{\gamma})$.}
		\end{cases}
	\end{equation*}
\end{prop}
\begin{proof}
	This is proven in the same way as \cref{geombimodonproj} by noticing that now any non-propagating line is killed by definition of $\theta_{t}$.
\end{proof}
\begin{cor}\label{gtprojleftrightnuclear}
	The $\K_n$-$\K_n$-bimodule $G_{t}$ is sweet, i.e.~projective as a left and as a right $\K_n$-module.
\end{cor}
\begin{proof}
	We have $G_{t}=\bigoplus_{\lambda\in\Lambda_n}G_{t}e_\lambda$ is projective as a left $\K_n$-module by \cref{actionofgeombimodonprojnuclear}.
	Furthermore, using \eqref{geombimodantimultnuclear} we see that $G_{t}$ is projective as a right $\K_n$-module if and only if $G_{t^*}$ is projective as a left $\K_n$-module which we know.
\end{proof}
\begin{cor}\label{projfunctorsareexact}
	Projective functors are exact and preserve finitely generated modules.
\end{cor}
\begin{proof}
	Use \cref{actionofgeombimodonprojnuclear} and \cref{gtprojleftrightnuclear}.
\end{proof}
Our next goal is to prove an adjunction theorem between the $\theta_{t}$.
\begin{thm}\label{thetaadjunction}
	We have an adjunction $(\theta_{t^\ddagger},\theta_t)$.
	Thus, $\theta_{i+1}$ is left adjoint to $\theta_{i}$. 
\end{thm}
\begin{proof}
	This is proven exactly as \cref{hatthetaadjunction}.
\end{proof}
\begin{thm}\label{geombimodduality}
	Given any crossingless matching $t$ and any finite dimensional $\K_n$-module $M$ there exists a natural isomorphism $\theta_{t^\dagger}M^\dual\cong(\theta_{t}M)^\dual$, where $t^\dagger$ is the vertical reflection of $t$ at $0$.
	In particular, $(\theta_iM)^\dual\cong\theta_{-i}M^\dual$.
\end{thm}
\begin{proof}
	It suffices to construct a natural isomorphism $G_{t^\dagger}\otimes_{\K_n}M^\dual\cong(G_t\otimes_{\K_n}M)^\dual$.
	For this we define the auxiliary map
	\begin{equation*}
		s\colon G_{t^\dagger}\otimes M^\dual\otimes G_t\otimes M\to\bbC
	\end{equation*}
	by sending $x\otimes f\otimes y\otimes m$ to $f(\phi(x^*\otimes y)m)$.
	Observe that $(t^\dagger)^*$ is the rotation around $\frac{1}{2}$ of the vertical mirror of $t$ at $0$.
	But this is the same as the horizontal mirror image shifted one to the right, i.e.\ $(t^\dagger)^*=t^\ddagger$.
	Hence, we can apply $\phi$ in the definition of $s$, and thus $s$ is well-defined.

	Now let $u\in\K_n$.
	Using the linearity properties from \cref{phiisbimodandbalanced} (or its $\K_n$-equivalent) we have 
	\begin{align*}
		s(xu\otimes f\otimes y\otimes m)&=f(\phi((xu)^*\otimes y)m)=f(u^*\phi(x^*\otimes y)m)=s(x\otimes uf\otimes y \otimes m),\\
		s(x\otimes f\otimes yu\otimes m)&=f(\phi(x^*\otimes yu)m)=f(\phi(x^*\otimes y)um)=s(x\otimes uf\otimes y \otimes um).
	\end{align*}
	This means that $s$ factors over 
	\begin{equation*}
		s\colon G_{t^\dagger}\otimes_{\K_n} M^\dual\otimes G_t\otimes_{\K_n} M\to\bbC,
	\end{equation*}
	and thus we get an induced map
	\begin{equation*}
		\tilde{s}\colon G_{t^\dagger}\otimes_{\K_n} M^\dual\to(G_t\otimes_{\K_n} M)^\dual
	\end{equation*}
	by sending $x\otimes f$ to $(y\otimes \mapsto s(x\otimes f\otimes y\otimes m))$.

	It remains to show that $\tilde{s}$ is $\K_n$-linear and a vector space isomorphism.

	For linearity let $u\in\K_n$.
	\begin{align*}
		(u\tilde{s}(x\otimes f))(y\otimes m) &= \tilde{s}(x\otimes f)(u^*y\otimes m)=f(\phi(x^*\otimes u^*y)m)=f(\phi((ux)^*\otimes y)m)\\
		&=\tilde{s}(ux\otimes f)(y\otimes m)
	\end{align*}
	where we used that $\phi$ is $\K_n$-balanced.

	In order to show the vector space isomorphism it suffices to restrict for each $\lambda\in\Lambda_n$ at
	\begin{equation*}
		e_\lambda G_{t^\dagger}\otimes_{\K_n} M^\dual\to e_\lambda(G_t\otimes_{\K_n} M)^\dual.
	\end{equation*}
	We can identify $e_\lambda(G_t\otimes_{\K_n} M)^\dual$ with $(e_\lambda^*G_t\otimes_{\K_n} M)^\dual$.
	Now observe that $e_\lambda^*$ is the same as vertical mirroring $e_\lambda$ at $0$.
	Therefore, we have $e_\lambda G_{t^\dagger}\neq 0$ if and only if $e_\lambda^*G_t\neq 0$, so we may assume that these are non-zero.

	By \cref{actionofgeombimodonprojnuclear} we have that $e_\lambda G_{t^\dagger}=e_\nu\K_n$ for some $\nu\in\Lambda_n$ and for the same $\nu$ we also have $e_\lambda^*G_t=e_\nu^*\K_n$.

	Therefore,it suffices to (see also the proof of \cref{adjunctionhelperthm}) show that
	\begin{equation*}
		e_\nu M^\dual\to (e_\nu^*M)^*,\qquad e_\nu f=f(e_\nu^*\cdot\_)\mapsto (e_\nu^*m\mapsto f(e_\nu^* m))
	\end{equation*}
	is an isomorphism, but this is clear.
\end{proof}
\begin{rem}
	One should observe that the index shift in the theorem (i.e.~$\theta_i$ turning into $\theta_{-i}$ under the duality) agrees exactly with \cref{dualityonKLR}.
\end{rem}
\subsection{Action on special classes of modules}
We already have proven in \cref{actionofgeombimodonprojnuclear} how geometric bimodules act on projective modules. 
In this section we are going to determine the effect of geometric bimodules on standard, costandard and irreducible modules.
But first, we use \cref{knisquher} to define standard and costandard modules.
\begin{defi}\label{defistandardcostandard}
	Let $\lambda\in\Lambda_n$ and define $\K_n^{\leq\lambda}$ to be the quotient of $\K_n$ by the two-sided ideal generated by all $e_\mu$ for $\mu\nleq\lambda$.
	We will also write $\bar{x}$ for the image of $x\in\K_n$ in this quotient.
	We then define the left $\K_n$-modules $\Delta_n(\lambda)\coloneqq \K_n^{\leq\lambda}\bar{e}_\lambda$ and $\nabla_n(\lambda)\coloneqq (\bar{e}_\lambda\K_n^{\leq\lambda})^*$ where $(\_)^*$ denotes the usual vector space dual.
	The module $\Delta(\lambda)$ has a basis given by $\{(y\bar{e}_\lambda)\mid y\in Y(\lambda)\}$.
	Furthermore, the vectors $(\bar{e}_\lambda x)$ for $x\in X(\lambda)$ give a basis for $\bar{e}_\lambda\K_n^{\leq\lambda}$ and its dual basis gives a basis for $\nabla(\lambda)$. 
\end{defi}
\begin{prop}\label{geombimodonstandard}
	Let $t$ be a crossingless matching and $\gamma\in\Lambda_n$. 
	\begin{enumerate}
		\item\label{geombimodonstandardi} The $\K_n$-module $\theta_t\Delta(\overline{\gamma})$ has a filtration
		\begin{equation*}
			\{0\}=M_0\subset M_1\subset\dots\subset M_r=\theta_t\Delta(\overline{\gamma})
		\end{equation*}
		such that $M_i/M_{i-1}\cong\Delta(\overline{\mu_i})$.
		In this case $\mu_1,\dots,\mu_r$ denote the elements of the set $\{\mu\in\Lambda_n\mid\mu t\gamma\text{ oriented}\}$ ordered such that $\mu_j>\mu_i$ implies $j<i$.
		\item\label{geombimodonstandardii} The module is nonzero if and only if no cup of $t\gamma$ contains more dots than cups and for every cup there has to be a dot directly to its left or right (these are chosen pairwise distinct for all cups).
		\item\label{geombimodonstandardiii} Assuming \cref{geombimodonstandardii}, the module $\theta_t\Delta(\overline{\gamma})$ is indecomposable with irreducible head $L(\overline{\lambda})$, where $\overline{\lambda}$ is given by the upper reduction of $t\overline{\gamma}$. 
	\end{enumerate}
\end{prop}
\begin{proof}
	The module $\Delta(\overline{\gamma})$ is the quotient of $P(\overline{\gamma})$ by the subspace spanned by all oriented circle diagrams $\underline{\nu}\eta\overline{\gamma}$ with $\eta\neq \gamma$.
	Therefore, $\theta_t\Delta(\overline{\gamma})=G_t\otimes_{\K_n}\Delta(\overline{\gamma})$ is obtained as the quotient of $G_te_\gamma$ by the subspace spanned by all circle diagrams $\underline{\nu}\mu t\eta\overline{\gamma}$ with $\eta\neq\gamma$.
	Hence, $\theta_t\Delta(\overline{\gamma})$ has a basis given by the images of $\underline{\nu}\mu t\gamma\overline{\gamma}$ under the quotient map.

	Now let $M_0=\{0\}$ and define inductively $M_i$ to be generated by $M_{i-1}$ and $\{\underline{\nu}\mu_i t\gamma\overline{\gamma}\mid \text{for all oriented cup diagrams $\underline{\nu}\mu_i$}\}$.
	Every $M_i$ is a $\K_n$-submodule by \cite{BS21}*{Lemma 5.5}.
	Furthermore, the map 
	\begin{equation*}
		M_i/M_{i-1}\to \Delta(\overline{\mu_i}), \quad \underline{\nu}\mu_i t\gamma\overline{\gamma}\mapsto\underline{\nu}\mu_i\overline{\mu_i}
	\end{equation*}
	defines an isomorphism of $\K_n$-modules (in the definition of the map we mean the images of the elements under the quotient maps rather than the elements themselves).
	This proves \cref{geombimodonstandardi}.

	For \cref{geombimodonstandardii} observe that if any of the mentioned conditions is not satisfied then there exists no oriented diagram of the form $\mu t\gamma$. 
	Thus, in this case $\theta_t\Delta(\lambda)=0$.
	For the converse observe that $\lambda t \gamma$ is oriented with $\lambda$ defined as in \cref{geombimodonstandardiii}.

	Now the functors $\theta_t$ are exact by \cref{gtprojleftright}, and thus $\theta_t\Delta{\overline{\gamma}}$ is a quotient of $\theta_tP(\overline{\gamma})$.
	Thus, it is either $0$ or indecomposable with irreducible head $L(\overline{\lambda})$ (see \cref{actionofgeombimodonprojnuclear}).
\end{proof}
\newpage
\begin{prop}\label{geombimodoncostandard}
	Let $t$ be a crossingless matching and $\gamma\in\Lambda_n$. 
	\begin{enumerate}
		\item\label{geombimodoncostandardi} The $\K_n$-module $\theta_t\nabla(\overline{\gamma})$ has a filtration
		\begin{equation*}
			\{0\}=M_0\subset M_1\subset\dots\subset M_r=\theta_t\nabla(\overline{\gamma})
		\end{equation*}
		such that $M_i/M_{i-1}\cong\nabla(\overline{\mu_i})$.
		In this case $\mu_1,\dots,\mu_r$ denote the elements of the set $\{\mu\in\Lambda_n\mid\gamma t^\ddagger\mu\text{ oriented}\}$ ordered such that $\mu_j>\mu_i$ implies $j>i$.
		\item\label{geombimodoncostandardii} The module is nonzero if and only if each cap of $\gamma t^\ddagger$ contains exactly as many dots as caps and for every cap there cannot be a dot directly to its left and right.
		\item\label{geombimodoncostandardiii} Assuming \cref{geombimodonstandardii}, the module $\theta_t\nabla(\overline{\gamma})$ is indecomposable with irreducible socle $L(\overline{\lambda})$, where $\underline{\lambda}$ is given by the lower reduction of $\underline{\gamma}t^\ddagger$. 
	\end{enumerate}
\end{prop}
\begin{proof}
	First we look at the right $\K_n$-modules $\nabla^*(\nu)=\bar{e}_\lambda\K_n^{\leq\lambda}$ as in the notation of \cref{defistandardcostandard}.
	With verbatim the same proof as \cref{geombimodonstandard} we can prove the following.
	\begin{enumerate}
		\item The right $\K_n$-module $\nabla^*(\overline{\gamma})\otimes_{\K_n}G_{t^\ddagger}$ has a filtration
		\begin{equation*}
			\{0\}=M_0\subset M_1\subset\dots\subset M_r=\theta_t\nabla(\overline{\gamma})
		\end{equation*}
		such that $M_i/M_{i-1}\cong\nabla^*(\overline{\mu_i})$.
		In this case $\mu_1,\dots,\mu_r$ denote the elements of the set $\{\mu\in\Lambda_n\mid\gamma t^\ddagger\mu\text{ oriented}\}$ ordered such that $\mu_j>\mu_i$ implies $j<i$.
		\item The module is nonzero if and only if each cap of $\gamma t^\ddagger$ contains exactly as many dots as caps and for every cap there cannot be a dot directly to its left and right.
		\item Assuming \cref{geombimodonstandardii}, the right module $\nabla^*(\overline{\gamma})\otimes_{\K_n}G_{t^\ddagger}$ is indecomposable with irreducible head $L(\overline{\lambda})$, where $\underline{\lambda}$ is given by the lower reduction of $\underline{\gamma}t^\ddagger$. 
	\end{enumerate}
	
	Now using the proof and arguments as in the proof of \cref{geombimodduality}, we can show that for a right $\K_n$-module $M$ we have 
	\begin{equation*}
		G_t\otimes_{\K_n}M^*\cong (M\otimes_{\K_n}G_{t^\ddagger})^*
	\end{equation*}
	where $*$ denotes the usual vector space dual that turns the right module $M$ into a left module via $(xf)(m)=f(mx)$.

	By definition, we have that $\nabla(\overline{\nu})^*=\nabla^*(\overline{\nu})$, and thus $\theta_t\nabla(\overline{\nu})=(\nabla^*(\overline{\nu})\otimes_{\K_n}G_{t^\ddagger})^*$.
	Under this duality the filtration is turned upside down, and thus we have proven the proposition.
\end{proof}

\begin{thm}\label{geombimodonirred}
Let $t$ be a generalized crossingless matching and $\lambda\in\Lambda_n$.
Then
\begin{enumerate}
	\item\label{thetaonirredi} in the Grothendieck group of $\K_n\mmod$
	\begin{equation*}
		[\theta_tL(\overline{\lambda})]=\sum_{\mu} [L(\overline{\mu})]
	\end{equation*}
	where we sum over all $\mu$ such that 
	\begin{enumerate}
		\item\label{thetaonirredia} $t^\ddagger\overline{\mu}$ contains neither circles nor non-propagating lines ending at the top,
		\item\label{thetaonirredib} $\overline{\lambda}$ is the upper reduction of $t^\ddagger\overline{\mu}$.
	\end{enumerate}
	\item\label{thetaonirredii} $\theta_tL(\overline{\lambda})$ is nonzero if and only if $t\overline{\lambda}$ has neither circles nor non-propagating lines ending at the top and $\overline{\lambda}$ is the upper reduction of $t^\ddagger t\overline{\lambda}$.
	\item\label{thetaonirrediii} Under the assumptions from \cref{thetaonirredii} define $\overline{\nu}$ to be the upper reduction of $t\overline{\lambda}$.
	In this case $\theta_tL(\overline{\lambda})$ is an indecomposable module with irreducible head $L(\overline{\nu})$.
	\item\label{thetaonirrediv} Under the assumptions from \cref{thetaonirredii} define $\underline{\nu'}$ to be the lower reduction of $\underline{\lambda}t^\ddagger$.
	In this case $\theta_tL(\overline{\lambda})$ is an indecomposable module with irreducible socle $L(\overline{\nu'})$.
\end{enumerate}
\end{thm}
\begin{proof}
	In order to prove \cref{thetaonirredi} we compute (using the adjunction $(\theta_{t^\ddagger},\theta_t)$ from \cref{thetaadjunction})
	\begin{equation}\label{thetaonirredproofeq}
		\Hom_{\K_n}(P(\overline{\mu}),\theta_tL(\overline{\lambda}))=\Hom_{\K_n}(\theta_{t^\ddagger}P(\overline{\mu}),L(\overline{\lambda}))
	\end{equation}
	Now $\theta_{t^\dagger}P(\overline{\mu})\neq 0$ if and only if \cref{thetaonirredia} is satisfied, in which case it is isomorphic to $P(\overline{\beta})$ where $\overline{\beta}$ denotes the upper reduction of $t^\dagger\overline{\mu}$.
	Thus, \eqref{thetaonirredproofeq} is nonzero if and only if \cref{thetaonirredia} and \cref{thetaonirredib} are satisfied and in which case it is isomorphic to $\bbC\langle k\rangle$.

	For \cref{thetaonirredii} and \cref{thetaonirrediii} note that $\theta_tL(\overline{\lambda})$ is a quotient of $\theta_tP(\lambda)$ as $\theta_t$ is exact by \cref{projfunctorsareexact}.
	But $\theta_tP(\lambda)$ is non-zero if and only if $t\overline{\lambda}$ has neither circles nor non-propagating lines ending at the top and in which case it is isomorphic to $P(\overline{\nu})$ by \cref{actionofgeombimodonprojnuclear}.
	Thus, $\theta_t L(\overline{\lambda})$ is either zero or has irreducible head $L(\overline{\nu})$.
	But $L(\overline{\nu})$ can only appear if $\lambda$ is the upper reduction of $t^\ddagger\overline{\nu}$ or equivalently $t^\ddagger t\overline{\lambda}$.
	For \cref{thetaonirrediv} note that $L(\overline{\lambda})$ is the socle of $\nabla(\lambda)$.
	Again using \cref{projfunctorsareexact} we see that $\theta_tL(\overline{\lambda})$ is a submodule of $\theta_t\nabla(\lambda)$, which is indecomposable and has irreducible socle $L(\overline{\nu'})$ ($\nu'$ as in the statement of the lemma).
	So if $\theta_tL(\overline{\lambda})\neq 0$, it has the same socle.
\end{proof}
\section{Equivalence between \texorpdfstring{$\K_n$-mod}{Kn-mod} and \texorpdfstring{$\lie{p}(n)\mmod$}{p(n)-mod}}\label{sectionmainthm}
In this section we will prove the main equivalence between $\K_n$-mod and $\lie{p}(n)$-mod.
We will achieve this by identifying $\K_n$ with the endomorphism ring of a projective generator for $\lie{p}(n)$-mod.
\subsection{Generating projective objects using translation functors}
Recall that we can associate to each $\mathcal{P}(\lambda)$ a cap diagram with $n$ caps via \cref{defidominttocupdiag}.
On the other hand $\mathcal{P}(\lambda)$ also arises as the direct summand of some (shift of) $V^{\otimes d}$, and thus as $\Theta^+_{i_k}\cdots\Theta^+_{i_1}\bbC$.
\Cref{KLRtopnp,phiisomorphism} tells us that $\mathcal{P}(\lambda)$ is actually the image of $\hat{\theta}_{i_k}\cdots\hat{\theta}_{i_1}\hat{P}(\overline{\iota})$ in $\mathcal{F}$.
But $\hat{\theta}_{i_k}\cdots\hat{\theta}_{i_1}\hat{P}(\iota)$ is isomorphic to some $\hat{P}(\overline{\nu})$ for some cap diagram $\overline{\nu}$ by \cref{geombimodonproj}.
This gives us two different ways to associate a cup diagram to an integral dominant weight $\lambda$ and our next goal is to show that these two notions agree (cf.~\cref{combinatorialdiagram}).

The next proposition will prove that the action of $\Theta_i$ on $\mathcal{P}(\lambda)$ follows the same rule as \cref{geombimodonproj}.
\begin{prop}\label{translationfunctoronprojperiplectic}
	Given a dominant integral weight $\lambda$, we have
	\begin{equation*}
		\Theta_i\mathcal{P}(\lambda)=\begin{cases}
			0&\text{if $t^i\overline{\lambda}$ contains a circle or a non-propagating line,}\\
			\mathcal{P}(\nu)&\text{otherwise, where $\nu$ denotes the upper reduction of $t^i\overline{\lambda}$.}
		\end{cases}
	\end{equation*}
\end{prop}
\begin{proof}
	It suffices to match the combinatorics from \cref{geombimodonproj} with the black dot combinatorics from \cite{BDE19}
	For this observe that our cap diagrams are obtained from their weight diagrams by connecting pairs of $\circ\bullet$ by a cap and then shifting the total diagram $\frac{3}{2}$ to the right.
	The statement then follows from \cite{BDE19}*{Lemma 7.2.1 + Lemma 7.2.3}.
\end{proof}
\begin{lem}\label{deltanprojexplicit}
	Given a residue sequence $(i_1, \dots i_k)$ of an up-tableau of shape $\delta_n$, then
	\begin{equation*}
		\Theta_{i_k}\dots\Theta_{i_1}\bbC=\mathcal{P}(n-2,n-4,\dots,-n).
	\end{equation*}
\end{lem}
\begin{proof}
	We prove this only for the residue sequence $(0,1,\dots, n-2,-1,\dots,n-4,\dots,-n+2,n-1,n+1,\dots,-n+1)$.
	This is the residue sequence that first builds $\delta_{n-1}$ and then adds the missing boxes.
	All the other ones are obtained via swapping entries $i$ and $j$ with $\abs{i-j}>1$, which gives isomorphic $\lie{p}(n)$-modules.

	First look at the residue sequence $\bm{j}=(0,1,\dots, n-2,-1,\dots,n-4,\dots,-n+2)$ of $\delta_{n-1}$.
	We claim that $\Theta_{\bm{j}}\bbC=\nabla(0,-1,\dots,-n+1)=\mathcal{L}(n-1,n-3,\dots,-n+1)$.
	Using the adjunction $(\Theta_{i+1},\Theta_i)$ we find
	\begin{align*}
		[\Theta_{\bm{j}}\bbC : \mathcal{L}(\mu)]&=\dim\Hom_{\lie{p}(n)}(\mathcal{P}(\mu), \Theta_{\bm{j}}\bbC)\\
		&=\dim\Hom_{\lie{p}(n)}(\Theta_{\bm{j}'}\mathcal{P}(\mu), \mathcal{L}(n-1,n-2,\dots,0))\\
	\end{align*}
	where $\bm{j}'=(-n+3,-n+5,-n+4,-n+7,-n+6,-n+5,\dots,n-1,\dots,1)$.
	This means that we need to find all $\mu$ such that $\Theta_{\bm{j}'}\mathcal{P}(\mu)\cong \mathcal{P}(n-1,n-2,\dots,0)$.
	Note that the (reduction of the) generalized crossingless matching looks like:
	
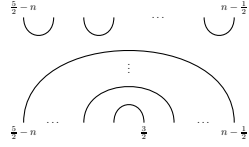
\begin{equation*}
		\begin{tikzpicture}[scale=0.4, transform shape]
			\CUP{1}
			\SETCOORD{1}{0}
			\CUP{1}
			\SETCOORD{3}{0}
			\CUP{1}
			\SETCOORD{0}{-3.5}
			\CAP{-7}
			\SETCOORD{2}{0}
			\CAP{3}
			\SETCOORD{-1}{0}
			\CAP{-1}
			\node at (4.5,0) {$\dots$};
			\node at (1,-3.5) {$\dots$};
			\node at (6,-3.5) {$\dots$};
			\node at (3.5,-1.6) {$\vdots$};
			\node[anchor=south] at (0,0) {$\frac{5}{2}-n$};
			\node[anchor=south] at (7,0) {$n-\frac{1}{2}$};
			\node[anchor=north] at (0,-3.5) {$\frac{5}{2}-n$};
			\node[anchor=north] at (7,-3.5) {$n-\frac{1}{2}$};
			\node[anchor=north] at (4,-3.5) {$\frac{3}{2}$};
		\end{tikzpicture}
	\end{equation*}
	The cap diagram associated to $\mathcal{P}(n-1,n-2,\dots,0)$ has right endpoints of caps at positions $\{n+\frac{1}{2},n-\frac{1}{2},\dots,\frac{3}{2}\}$, hence its cap diagram looks like
	
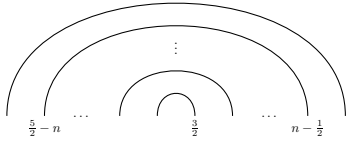
\begin{equation*}
		\begin{tikzpicture}[scale=0.5, transform shape]
			\SETCOORD{-1}{0}
			\CAP{9}
			\SETCOORD{-1}{0}
			\CAP{-7}
			\SETCOORD{2}{0}
			\CAP{3}
			\SETCOORD{-1}{0}
			\CAP{-1}
			\node at (1,0) {$\dots$};
			\node at (6,0) {$\dots$};
			\node at (3.5,1.9) {$\vdots$};
			\node[anchor=north] at (0,0) {$\frac{5}{2}-n$};
			\node[anchor=north] at (7,0) {$n-\frac{1}{2}$};
			\node[anchor=north] at (4,0) {$\frac{3}{2}$};
		\end{tikzpicture}
	\end{equation*}
	Thus, there exists only one $\mu$ such that the upper reduction process returns the cap diagram of $(n-1,n-2,\dots,0)$.
	Namely, the cap diagram where the right endpoints of caps are given by $\{n+\frac{1}{2},n-\frac{1}{2},\dots,-n+\frac{5}{2}\}$.
	And this is associated with $\mu=(n-1,n-3,\dots,-n+1)$ so $\Theta_{\bm{j}}\bbC=\mathcal{L}(n-1,n-3,\dots,-n+1)=\nabla(n-1,n-3,\dots,-n+1)$ as claimed.
	
	For the next step we prove that 
	\begin{equation*}
		\Theta_{-n+1}\Theta_{-n+3}\cdots\Theta_{n-1}\nabla(n-1,n-3,\dots,-n+1)=\mathcal{P}(n-2,n-4,\dots,-n).
	\end{equation*}
	
	By \cite{BDE19}*{Prop.\ 5.2.2} we know that 
	\begin{equation*}
		\Theta_{n-1-2i}\nabla(n-2,n-4,\dots,n-2i,n-2i-1,n-2i-3,\dots,-n+1)
	\end{equation*} has a quotient isomorphic to $\nabla(n-2,n-4,\dots,n-2i,n-2i-2,n-2i-3,\dots,-n+1)$. 
	But repeating this step we see that $\Theta_{-n+1}\Theta_{-n+3}\cdots\Theta_{n-1}\nabla(n-1,n-3,\dots,-n+1)$ has a quotient isomorphic to $\nabla(n-2,n-4,\dots,-n)=\mathcal{L}(n-2,n-4,\dots,-n)$.

	Furthermore, $\Theta_{-n+1}\Theta_{-n+3}\dots\Theta_{n-1}\nabla(0,-1,\dots,-n+1)$ is projective by \cref{partitioncriterionforproj}.
	Hence, it has to be isomorphic to $\mathcal{P}(n-2,n-4,\dots,-n)$.
\end{proof}
\begin{prop}\label{notionsofcupdiagramsagree}
	The diagram \cref{combinatorialdiagram} from the introduction commutes.
\end{prop}
\begin{proof}
	Observe that the cap diagram associated to the dominant integral weight $(n-2,n-4,\dots,-n)$ has $n$ caps and the positions of their right endpoints are $n-2i+\frac{3}{2}$ for $1\leq i\leq n$.
	On the other hand to a residue sequence $(i_1,\dots, i_k)$ of $\delta_n=(n,n-1,\dots,1)$  we associate the cap diagram that is obtained by the upper reduction $\overline{\nu}$ of $t^{i_k}\cdots t^{i_1}\overline{\iota}$.
	Now $t^{i_k}\cdots t^{i_1}\overline{\iota}$ looks (topologically) like the following.
	
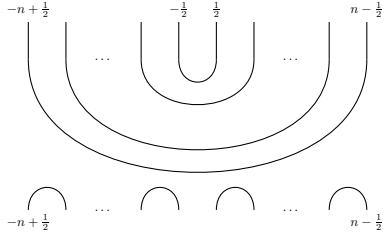
\begin{equation*}
		\begin{tikzpicture}[scale=0.5, transform shape]
			\SETCOORD{0}{1}
			\LINE{0}{-1}
			\CUP{1}
			\LINE{0}{1}
			\SETCOORD{1}{0}
			\LINE{0}{-1}
			\CUP{-3}
			\LINE{0}{1}
			\SETCOORD{-2}{0}
			\LINE{0}{-1}
			\CUP{7}
			\LINE{0}{1}
			\SETCOORD{1}{0}
			\LINE{0}{-1}
			\CUP{-9}
			\LINE{0}{1}
			\SETCOORD{0}{-5}
			\CAP{1}
			\SETCOORD{2}{0}
			\CAP{1}
			\SETCOORD{1}{0}
			\CAP{1}
			\SETCOORD{2}{0}
			\CAP{1}
			\node at (3,0) {$\dots$};
			\node at (-2,0) {$\dots$};
			\node at (3,-4) {$\dots$};
			\node at (-2,-4) {$\dots$};
			\node[anchor=south] at (0,1) {$-\frac{1}{2}$};
			\node[anchor=south] at (1,1) {$\frac{1}{2}$};
			\node[anchor=south] at (-4,1) {$-n+\frac{1}{2}$};
			\node[anchor=south] at (5,1) {$n-\frac{1}{2}$};
			\node[anchor=north] at (-4,-4) {$-n+\frac{1}{2}$};
			\node[anchor=north] at (5,-4) {$n-\frac{1}{2}$};
		\end{tikzpicture}
	\end{equation*}
	Hence, the right endpoints of the caps in $\overline{\nu}$ are exactly at positions $-n+\frac{3}{2},-n+\frac{5}{2},\dots,n-\frac{1}{2}$, which agrees with the cup diagram associated to $(n-2,n-4,\dots,-n)$.
	Therefore, for this particular weight the two notions agree and using \cref{translationfunctoronprojperiplectic} we see that these notions agree for all dominant integral weights.
	Hence, we obtain the commutativity of \cref{combinatorialdiagram} from the introduction.
\end{proof}	
\subsection{Dimension of homomorphism spaces}
When proving the main theorem we will need some estimate on the dimensions of homomorphism spaces to conclude that the ideal $\mathbb{I}_n$ is big enough.
The next theorem provides this estimate. 
Note that we only state one inequality here as it suffices for the proof of the main theorem but from the main theorem it is clear that it is actually an equality.
\begin{thm}\label{correctdimensionformaintheorem}
	We have the following equality.
	\begin{equation}
		\dim\Hom_{\lie{p}(n)}(\mathcal{P}(\lambda), \mathcal{P}(\mu))\geq\dim e_\lambda\K_n e_\mu
	\end{equation}
\end{thm}
\begin{proof}
	Note that the right-hand side is given by circle diagrams with $n$ cups and caps that have no non-propagating line.
	First assume that $\mu$ is very typical, i.e.\ $\mu_i\geq\mu_{i+1}+4$.
	This means that the associated cap diagram contains no nested caps and all caps have by assumption at least to rays between them.
	Therefore, any orientable circle diagram necessarily looks locally like one of the following two
	
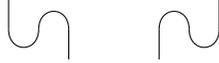
\begin{equation*}
		\begin{tikzpicture}[scale=0.4]
			\LINE{0}{-1}
			\CUP{1}
			\CAP{1}
			\LINE{0}{-1}
			\SETCOORD{3}{0}
			\LINE{0}{1}
			\CAP{1}
			\CUP{1}
			\LINE{0}{1}
		\end{tikzpicture}
	\end{equation*} 
	and as all caps are far apart from one another these local picture do not interact.
	This gives rise to $2^n$ valid orientable circle diagrams and the cup diagrams are associated to all weights of the form $(\mu_1+2\eps_1,\dots,\mu_n+2\eps_n)$ where $(\eps_1,\dots,\eps_n)\in\{0,1\}^n$.
	
	On the other hand the proof of \cite{BDE19}*{Proposition 8.1.1} we know that
	\begin{equation}
		\dim\Hom_{\lie{p}(n)}(\mathcal{P}(\lambda), \mathcal{P}(\mu))=\abs{\blacktriangle(\mu)\cap\blacktriangledown(\lambda)}
	\end{equation}
	in their notation.
	As $\lambda$ is also a typical weight, it is clear from their definition that $\blacktriangledown(\lambda)$ contains all weights of the form $(\lambda_1-2\eps_1,\dots,\lambda_n-2\eps_n)$ as well as $\mu\in\blacktriangle(\mu)$.
	So we proved in this case that the inequality holds.

	If $\mu$ is not very typical, we can use translation functors and \cite{BDE19}*{Theorem 7.1.1} to obtain $\mathcal{P}(\mu)$ as $\Theta_{i_k}\cdots\Theta_{i_1}\mathcal{P}(\mu')$ where $\mu'$ is a very typical weight.
	Then we have 
	\begin{align*}
		&\dim\Hom_{\lie{p}(n)}(\mathcal{P}(\lambda), \mathcal{P}(\mu))\\
		&=\dim\Hom_{\lie{p}(n)}(\mathcal{P}(\lambda), \Theta_{i_k}\cdots\Theta_{i_1}\mathcal{P}(\mu'))\\
		&=\dim\Hom_{\lie{p}(n)}(\Theta_{i_1+1}\cdots\Theta_{i_k+1}\mathcal{P}(\lambda), \mathcal{P}(\mu'))\\
		&=\dim\Hom_{\lie{p}(n)}(\mathcal{P}(\lambda'), \mathcal{P}(\mu'))\\
		&\geq\dim\Hom_{\K_n}(P(\overline{\lambda'}),P(\overline{\mu'}))\\
		&=\dim\Hom_{\K_n}(\theta_{i_1+1}\cdots\theta_{i_k+1}P(\overline{\lambda}),P(\overline{\mu'}))\\
		&=\dim\Hom_{\K_n}(P(\overline{\lambda}),\theta_{i_k}\cdots\theta_{i_1}P(\overline{\mu'}))
	\end{align*}
	where we used the adjunction $(\Theta_i,\Theta_{i-1})$ (resp.\ $(\theta_i,\theta_{i-1})$) and that translation functor act topologically on projective modules by \cref{translationfunctoronprojperiplectic} and \cref{actionofgeombimodonprojnuclear}.
\end{proof}
\begin{rem}
	In fact, in \cref{correctdimensionformaintheorem} equality holds.
	This will follow from \cref{mainthm} below.
	We will however use this inequality in the proof of \cref{mainthm} to argue that the ideal $\mathbb{I}_n$ is \enquote{big enough}.
\end{rem}
\subsection{Main theorem}
\begin{thm}[Main theorem]\label{mainthm}
	There is an equivalence of categories 
	\begin{equation}
		\Psi\colon\K_n\mmod\to\lie{p}(n)\mmod
	\end{equation}
	such that $\Psi\circ\theta_i\cong\Theta_i\circ\Psi$.
	Furthermore, $\Psi$ identifies the highest weight structures on both sides.
\end{thm}
\begin{proof}
	By \cref{phiisomorphism} we know that the categories $\mathcal{F}$ and $\sGp$ are isomorphic.
	Thus, we get an induced functor $\mathcal{F}\to\Cp$, where $\Cp$ was defined as the category with objects $\Theta_{i_k}\cdots\Theta_{i_1}\bbC$.
	This functor is full and by definition essentially bijective on objects and it intertwines the translation functors.
	Now $\Theta_{i_k}\cdots\Theta_{i_1}\bbC$ is a direct summand of $V^{\otimes k}$ and $V$ is a projective generator for $\lie{p}(n)$.
	This means that every indecomposable projective module $\mathcal{P}(\lambda)$ appears as $\Theta_{i_k}\cdots\Theta_{i_1}\bbC$.
	By \cref{sequenceiseitheruptableauorreducible} we may assume that $(i_1, \dots, i_k)$ is the residue sequence of some up-tableau of shape $\Gamma$.
	Using \cref{partitioncriterionforzero} and \cref{partitioncriterionforproj} we see that $\delta_n\subseteq\Gamma\subsetneq\delta_{n+1}$.
	By \cref{geombimodonproj} we see that this corresponds to the module $\hat{P}(\overline{\nu})$ where $\overline{\nu}$ is the upper reduction of $t^{i_k}\dots t^{i_1}\overline{\iota}$.
	But this is equal to $\overline{\lambda}$ (which has $n$ caps) by \cref{notionsofcupdiagramsagree}, and thus we have a surjective algebra homomorphism
	\begin{equation}\label{ekesurjectsontopn}
		e\K e\cong \bigoplus_{\lambda,\mu}\Hom_{\K}(\hat{P}(\overline{\lambda}),\hat{P}(\overline{\mu}))\twoheadrightarrow\bigoplus_{\lambda,\mu}\Hom_{\lie{p}(n)}(\mathcal{P}({\lambda}),\mathcal{P}({\mu})),
	\end{equation}
	where the sum runs over all $\lambda$ and $\mu$ which are dominant integral weights for $\lie{p}(n)$.
	
	The left-hand side has a basis given by all orientable circle diagrams with exactly $n$ cups and caps.
	Now by \cref{nonpropifffactorthroughhigherlayer} we see that every circle diagram that has a non-propagating line maps factors through an object with more than $n$ caps (which means that the shape of the corresponding partition contains $\delta_{n+1}$).
	This means that its image in $\lie{p}(n)\mmod$ is $0$ by \cref{partitioncriterionforzero}.
	Therefore, we get an induced surjective morphism 
	\begin{equation}\label{knisopn}
		\Psi\colon\K_n=e\K e/\mathbb{I}_n\twoheadrightarrow\bigoplus_{\lambda,\mu}\Hom_{\lie{p}(n)}(\mathcal{P}({\lambda}),\mathcal{P}({\mu})).
	\end{equation}
	Using \cref{correctdimensionformaintheorem} we see that this is actually an isomorphism.
	The right-hand side is now equivalent to $\lie{p}(n)\mmod$.

	The functor $\mathcal{F}\to\Cp$ was compatible with the translation functors by definition.
	But the inclusion $e\K e\to \K$ is not compatible with translation functors as these can create new cups in $\K$ but not in $e\K e$.
	But from \cref{geombimodonproj} it is clear that if $\hat{\theta}_i \hat{P}(\overline{\lambda})=\hat{P}(\overline{\mu})$, then $\overline{\mu}$ has the same number of caps as $\overline{\lambda}$ or one more.
	This last case is exactly the reason why the inclusion $e\K e\to \K$ is not compatible with translation functors as $e\hat{\theta}_i e$ would produce $0$.
	But if the number of caps increases, the image in the right-hand side of \eqref{ekesurjectsontopn} is $0$, therefore \eqref{ekesurjectsontopn} is still compatible with translation functors.
	So we see especially that $\Psi\circ\theta_i\cong\Theta_i\circ\Psi$.
	Finally, it is compatible with the highest weight structures as the combinatorics describing multiplicities of standard and costandard modules in projectives agree.
\end{proof}
\section{Applications and consequences}\label{sectionapplications}
\subsection{Duals of irreducible modules}\label{dualityonsimples}
\begin{thm}
	Up to parity shift we have $\mathcal{L}(\lambda)^*\cong \mathcal{L}(\lambda^\#)$, where $\underline{\lambda^\#}=(\overline{\lambda}^\ast)$, i.e.\ it is obtained from $\lambda$ by rotating its cap diagram around $\frac{1}{2}$.
\end{thm}
\begin{proof}
	We begin by recalling a combinatorial procedure from \cite{BDE19}*{Proposition 5.3.1}.
	Namely, they look at a ($\rho$-shifted) dominant integral weight $\lambda_1>\lambda_2>\dots>\lambda_n$.
	Then they ordered the set $(i,j)$ for $i,j$ lexicographically and then look iteratively at the pairs $(i,j)$, starting with $(1,2)$ and ending with $(n-1,n)$.
	In each step one creates a new weight as follows. 
	If $\abs{\lambda_{k+1}-\lambda_k}>1$ for $k=i,j$, we increase $\lambda_i$ and $\lambda_j$ by $1$, otherwise we change nothing.
	Using this iterative procedure we receive a new weight diagram $\lambda^\dagger$.
	They prove then that $\mathcal{L}(\lambda)^*$ has highest weight given by the reflection of $\lambda^\dagger$ at $\frac{n-1}{2}$.

	We will prove now that this procedure gives the same result as rotating the associated cap diagram around $\frac{1}{2}$.

	We will prove this via induction on $n$.
	If $n=1$, the procedure from \cite{BDE19} sends a weight $(k)$ to $(-k)$.
	We would associate to $(k)$ the cap diagram with one cap that has endpoints $k+\frac{1}{2}$ and $k+\frac{3}{2}$.
	Rotating this around $\frac{1}{2}$ gives the cup diagram with one cup that has endpoints at $-k+\frac{1}{2}$ and $-k-\frac{1}{2}$.
	The right endpoint of this cup is at $-k+\frac{1}{2}$ and when translating back into weight diagrams we have to shift by $-\frac{1}{2}$, which results in $(-k)$.
	Observe here that our cup diagrams are shifted with respect to the cap diagrams, i.e.\ when translating a weight into a cap diagram we shift by $\frac{3}{2}$ and by $\frac{1}{2}$ for cup diagrams.

	Now let $n>1$ and assume that the statement has been proven for all smaller values.
	We have two different cases to distinguish, either we have one big outer cap containing all smaller caps or we have multiple outer caps.
	Suppose first that we have multiple outer caps.
	We denote the rightmost outer cap and all the caps contained in it by $C_1$ and the rest by $C_2$ and let $d_1$ be the number of caps in $C_1$ and $d_2$ the same for $C_2$.
	Then the rotation is the same as rotating each of these components separately around $\frac{1}{2}$.
	We claim now that the same holds true for the procedure of \cite{BDE19}.
	For the procedure of \cite{BDE19} we have to evaluate all the pairs $(i,j)$. 
	We first only have pairs where both belong to $C_1$, then mixed terms and in the end all the ones from $C_2$.
	We claim that any of the mixed terms changes the weight non-trivially.
	This holds true as $C_1$ is given by an outer cap. This means that for every right endpoint there is also one left endpoint before $C_2$.
	Hence, in terms of weights that there is enough space in between so that every of the mixed pairs is non-trivial.
	But therefore we first shift everything according to $C_1$, then we have the non-trivial mixed pairs, which shift $C_1$ to the right $d_2$ steps and $C_2$ by $d_1$.
	And then we shift everything according to $C_2$.
	After this we reflect at $\frac{n-1}{2}=\frac{d_1+d_2-1}{2}$. 
	But shifting first by $d_2$ and then reflecting at $\frac{d_1+d_2-1}{2}$ is the same as reflecting at $\frac{d_1-1}{2}$ and similar for $d_1$.
	But this means that the procedure of \cite{BDE19} splits this into two, applies their procedure there and glues them back together.
	Buy induction we know that for the shorter weights this procedure is the same as rotating around $\frac{1}{2}$ and so it also agrees for the long one.

	The last case is that we have one big outer cap, and we want to relate the procedure here to the one for the big cap removed.
	So we compare the procedure for the weight $(\lambda_1, \dots, \lambda_n)$ with $(\lambda'_2, \dots, \lambda'_n)$
	We denote the pairs for the smaller component by $(i',j')$ starting from $(2',3')$.
	We first look at the pairs $(1,i)$. 
	Note first that $(1,2)$ is trivial.
	But this means that the pair $(1,i)$ is non-trivial if and only if the pair $(2',i')$ is.
	Therefore, the $\lambda_1$ is in the end exactly one higher as the integer obtained for $\lambda_2'$.
	Furthermore, we know that $1\leq\abs{\lambda_2-\lambda_3}\leq 2$ as the cap diagram is completely nested.
	If $\abs{\lambda_2-\lambda_3}=2$ the pair $(1,3)$ non-trivial, and thus after executing all pairs $(1,i)$ we have $\abs{\lambda_2-\lambda_3}=1$.
	But then we see that $(2,3)$ is trivial and the other $(2,i)$ are non-trivial if and only if $(3',i')$ is non-trivial.
	Therefore, $\lambda_2$ ends up exactly one higher than $\lambda_3'$ would.
	Repeating this argument over and over we see that $\lambda_i$ for $i<n$ end up one bigger than $\lambda_{i+1}'$ and $\lambda_n$ lies directly next to $\lambda_{n-1}$.
	Reflecting $\lambda_i$ for $i<n$ at $\frac{n-1}{2}$ is the same as reflecting $\lambda'_{i+1}$ at $\frac{n-2}{2}$ and the additional $\lambda_n$ is after reflection directly to the right of $\lambda_{n-1}$.
	On the other hand rotating this cap diagram means rotating the inner part and adding a right endpoint of a cup directly to the right of everything.
	Thus, we can apply our induction hypothesis and see that both procedures give the same result.
\end{proof}
\begin{rem}
	There is no contravariant duality $d$ on $\lie{p}(n)$ preserving irreducibles and satisfying $d\theta_i\cong\theta_id$.
	For instance in $\lie{p}(1)$:
	\begin{align*}
		&\theta_1\theta_0d\bbC=\theta_1\theta_0\bbC=\mathcal{P}(0)\\
		&d\theta_1\theta_0\bbC = d\mathcal{P}(0) = \mathcal{I}(0)= \mathcal{P}(-2),
	\end{align*}
	which are not isomorphic.
\end{rem}
\begin{rem}
	We do not expect a direct geometric realization of $\lie{p}(n)\mmod$ in contrast to ordinary Brauer algebras (see e.g.~\cite{SW19}) since there one would expect such a (Verdier) duality.
\end{rem}
In many related representation theoretic contexts, a geometric interpretation implies Koszulity, \cite{BGS96}, which we will see now also fails for $\lie{p}(n)$.
\subsection{Gradings and Non-Koszulity}\label{nonskoszulsection}
Let $n=2$ and consider the indecomposable projective $\lie{p}(2)$-module $P$ corresponding to the cap diagram \begin{equation*}
	\begin{tikzpicture}[myscale=0.7]
		\setcounter{count}{0}
		\foreach \x in {3,1} {
			\node[anchor=north] at (\value{count},0) {$-\frac{\x}{2}$};
			\stepcounter{count}
		}
		\foreach \x in {1,3,5,7,9} {
			\node[anchor=north] at (\value{count},0) {$\frac{\x}{2}$};
			\stepcounter{count}
		}
		\LINE{0}{1}
		\SETCOORD{1}{-1}
		\CAP{1}
		\SETCOORD{1}{0}
		\LINE{0}{1}
		\SETCOORD{1}{-1}
		\CAP{1}
		\SETCOORD{1}{0}
		\LINE{0}{1}
		\node at (-1,0.5) {$\dots$};
		\node at (7,0.5) {$\dots$};
	\end{tikzpicture}.
\end{equation*}
By \cref{mainthm} we can carry out the computation inside $\K_n$ and $P\cong \mathcal{P}(2,-1)$ has a basis given by 
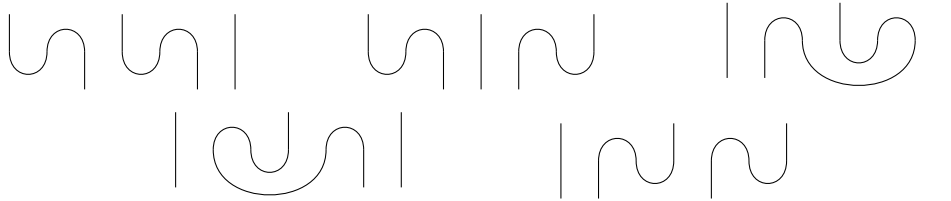
\begin{gather*}
	\begin{tikzpicture}[myscale=0.5]
		\LINE{0}{-1}
		\CUP{1}
		\CAP{1}
		\LINE{0}{-1}
		\SETCOORD{1}{2}
		\LINE{0}{-1}
		\CUP{1}
		\CAP{1}
		\LINE{0}{-1}
		\SETCOORD{1}{2}
		\LINE{0}{-2}
	\end{tikzpicture}\qquad\qquad\quad
	\begin{tikzpicture}[myscale=0.5]
		\LINE{0}{-1}
	\CUP{1}
	\CAP{1}
	\LINE{0}{-1}
	\SETCOORD{1}{2}
	\LINE{0}{-2}
	\SETCOORD{1}{0}
	\LINE{0}{1}
	\CAP{1}
	\CUP{1}
	\LINE{0}{1}
\end{tikzpicture}\qquad\qquad\quad
\begin{tikzpicture}[myscale=0.5]
	\LINE{0}{-2}
	\SETCOORD{3}{2}
	\LINE{0}{-1}
	\CUP{1}
	\CAP{1}
	\CUP{-3}
	\CAP{-1}
	\LINE{0}{-1}
	\SETCOORD{5}{2}
	\LINE{0}{-2}
\end{tikzpicture}\\
\begin{tikzpicture}[myscale=0.5]
	\LINE{0}{-2}
	\SETCOORD{3}{2}
	\LINE{0}{-1}
	\CUP{-1}
	\CAP{-1}
	\CUP{3}
	\CAP{1}
	\LINE{0}{-1}
	\SETCOORD{1}{2}
	\LINE{0}{-2}
\end{tikzpicture}\qquad\qquad\qquad
\begin{tikzpicture}[myscale=0.5]
	\LINE{0}{-2}
	\SETCOORD{1}{0}
	\LINE{0}{1}
	\CAP{1}
	\CUP{1}
	\LINE{0}{1}
	\SETCOORD{1}{-2}
	\LINE{0}{1}
	\CAP{1}
	\CUP{1}
	\LINE{0}{1}
\end{tikzpicture}
\end{gather*}
Using this and the multiplication rule inside $\K_n$, it is not hard to deduce the structure of $P$.
Namely, one gets for the radical and socle filtration
\begin{equation*}
	\begin{tabular}{lr}
		\multicolumn{2}{c}{$\mathcal{L}(2,-1)$}\\
		$\mathcal{L}(4,-1)$&$\mathcal{L}(3,2)$\\
		&$\mathcal{L}(2,1)$\\
		\multicolumn{2}{c}{$\mathcal{L}(4,1)$}
	\end{tabular}\qquad\text{and}\qquad
	\begin{tabular}{lr}
		\multicolumn{2}{c}{$\mathcal{L}(2,-1)$}\\
		&$\mathcal{L}(3,2)$\\
		$\mathcal{L}(4,-1)$&$\mathcal{L}(2,1)$\\
		\multicolumn{2}{c}{$\mathcal{L}(4,1)$}
	\end{tabular}
\end{equation*}
respectively.
In particular, these do not agree.
By \cite{BGS96}*{Proposition 2.4.1}, we see that $\K_n$ cannot be non-negatively graded such that $(\K_n)_0$ is semisimple and $\K_n$ is generated by $(\K_n)_1$.
For $n>2$, the same phenomenon occurs when considering e.g.~$\mathcal{P}(2,-1,-5,-9,-13,\dots)$.

Hence, $\lie{p}(n)$ does not admit a Koszul grading for $n\geq2$.
\subsection{Irreducible summands of \texorpdfstring{$V^{\otimes d}$}{V\textasciicircum d}}\label{irredsummandsinvod}
We want to defer our attention now to irreducible modules appearing as direct summands of $V^{\otimes d}$.
For $\lie{gl}(m|n))$ and $\lie{osp}(r|2n)$, there is a rather rich class of irreducible summands, the so-called \emph{Kostant modules}, see e.g.~\cite{Hei17} and \cite{HNS23} (or \cite{Neh21}) respectively.
In particular, there is exactly one irreducible summand for every block.
The situation is similar for $\lie{p}(n)$ in the sense that we will show that there exist exactly $n$ irreducible summands, each belonging to a unique one of the $n+1$ blocks.
However, representations of $\lie{gl}(m|n))$ and $\lie{osp}(r|2n)$ decompose into very small blocks, giving rise to many Kostant modules, contrasting the $n$ for $\lie{p}(n)$.
\begin{thm}
	The irreducible $\lie{p}(n)$-module $\mathcal{L}(\mu)$ appears as a summand of $V^{\otimes d}$ for some $d$ if and only if $\mu=(n-1,n-2,\dots,k,k-2,k-4,\dots,-k)$ for some $0\leq k\leq n-1$.
	Moreover, this summand appears as the image of the indecomposable object $\delta_k$ of $\Rep(P)$ under Schur--Weyl duality.
\end{thm}
\begin{proof}
	Denote by $\overline{\nu}$ the cap diagram corresponding to the ($\rho$-shifted) highest weight $(n-1,n-2,\dots,0)$ of the trivial module.
	Then,
	\begin{equation*}\overline{\nu}=
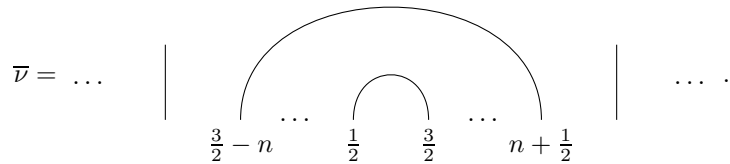
		\begin{tikzpicture}[baseline={(0,0.5)}]
			\node[anchor=north] at (3.5,0) {$\frac{3}{2}$};
			\node[anchor=north] at (2.5,0) {$\frac{1}{2}$};
			\node[anchor=north] at (5,0) {$n+\frac{1}{2}$};
			\node[anchor=north] at (1,0) {$\frac{3}{2}-n$};
			\LINE{0}{1}
			\SETCOORD{1}{-1}
			\CAP{4}
			\SETCOORD{1}{0}
			\LINE{0}{1}
			\SETCOORD{-3.5}{-1}
			\CAP{1}
			\node at (-1,0.5) {$\dots$};
			\node at (7,0.5) {$\dots$};
			\node at (4.25,0) {$\dots$};
			\node at (1.75,0) {$\dots$};
		\end{tikzpicture}\,.
	\end{equation*}
	We have to find all generalized crossingless matchings $t$ such that $\theta_tL(\overline{\nu})$ is irreducible.
	By means of \cref{geombimodonirred}, for any of these $t$ there exists exactly one $\overline{\mu}$ such that $t^\ddagger\overline{\mu}$ is orientable and $\overline{\nu}=\ured(t^\ddagger\overline{\mu})$.

	Hence, $t^\ddagger$ has to look something like
	
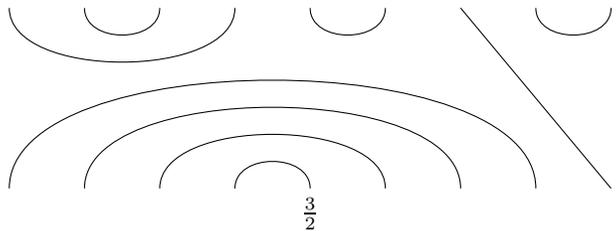
\begin{equation*}
		\begin{tikzpicture}[yscale=0.6, baseline = {(0,-2.4)}]
			\CUP{1}
			\SETCOORD{-2}{0}
			\CUP{3}
			\SETCOORD{1}{0}
			\CUP{1}
			\SETCOORD{2}{0}
			\CUP{1}
			\SETCOORD{-2}{0}
			\LINE{2}{-4}
			\SETCOORD{-1}{0}
			\CAP{-7}
			\SETCOORD{1}{0}
			\CAP{5}
			\SETCOORD{-1}{0}
			\CAP{-3}
			\SETCOORD{1}{0}
			\CAP{1}
			\node[anchor=north] at (3,-4) {$\frac{3}{2}$};
		\end{tikzpicture}\,.
	\end{equation*}
	The condition $\overline{\nu}=\ured(t^\ddagger\overline{\mu})$ implies that all the caps are nested with the right endpoint of the innermost cap at $\frac{3}{2}$.
	
	If there are any nested cups at the top, there cannot exist exactly one $\overline{\mu}$ because we could either connect the two left or right endpoints with a cup in $\overline{\mu}$ and can extend these to cap diagrams satisfying the conditions above.
	We also cannot have a line in between two cups, as we can connect this line either to the left or right cup with a cap in $\overline{\mu}$ and extend to two different admissible cap diagrams.
	But this means that $t^\ddagger$ is of the form
	
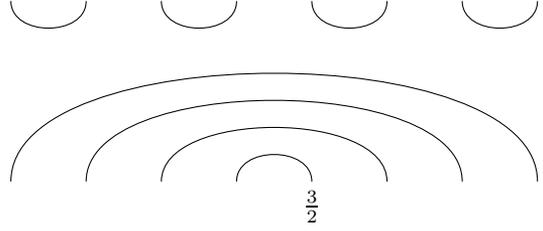
\begin{equation*}
		\begin{tikzpicture}[yscale=0.6, baseline = {(0,-2.4)}]
			\CUP{1}
			\SETCOORD{1}{0}
			\CUP{1}
			\SETCOORD{1}{0}
			\CUP{1}
			\SETCOORD{1}{0}
			\CUP{1}
			\SETCOORD{0}{-4}
			\CAP{-7}
			\SETCOORD{1}{0}
			\CAP{5}
			\SETCOORD{-1}{0}
			\CAP{-3}
			\SETCOORD{1}{0}
			\CAP{1}
			\node[anchor=north] at (4,-4) {$\frac{3}{2}$};
		\end{tikzpicture}\,,
	\end{equation*}
	where the only freedom is given by the number of nested caps $k$.
	If $k>n$ the result is $\theta_tL(\overline{\nu})=0$ as $t^\ddagger\overline{\mu}$ will always contain a non-propagating line.
	If $k=n$, \cref{deltanprojexplicit} computed $\theta_tL(\overline{\nu})$ to be an indecomposable projective module and not irreducible.
	In all other cases (i.e.~$0\leq k\leq n-1$) the result is irreducible and given by the cap diagram
	\begin{equation*}
		\overline{\mu_k}=
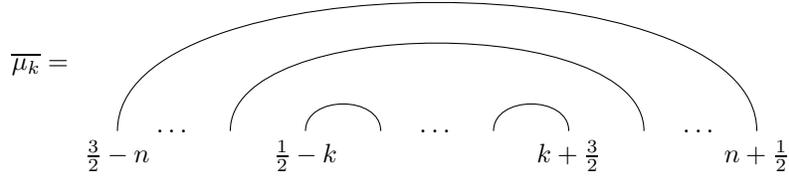
\begin{tikzpicture}[yscale=0.6]
			\CAP{8.5}
			\SETCOORD{-1.5}{0}
			\CAP{-5.5}
			\SETCOORD{1}{0}
			\CAP{1}
			\SETCOORD{1.5}{0}
			\CAP{1}
			\node at (0.75,0) {$\dots$};
			\node at (4.25,0) {$\dots$};
			\node at (7.75,0) {$\dots$};
			\node[anchor=north] at (0,0) {$\frac{3}{2}-n$};
			\node[anchor=north] at (2.5,0) {$\frac{1}{2}-k$};
			\node[anchor=north] at (6,0) {$k+\frac{3}{2}$};
			\node[anchor=north] at (8.5,0) {$n+\frac{1}{2}$};
		\end{tikzpicture}
	\end{equation*}
	corresponding to the weight $(n-1,n-2,\dots,k,k-2,k-4,\dots,-k)$ as claimed.
\end{proof}
\subsection{Extensions between irreducibles and Ext-quivers}\label{extbetweensimplessection}
The main goal of this section is to compute the dimension of 
\begin{equation*}
	\Ext_{\lie{p}(n)}(\mathcal{L}(\lambda), \mathcal{L}(\mu)).
\end{equation*}
Using a projective resolution of $\mathcal{L}(\lambda)$, it is easy to see that it is at most one dimensional as $\mathcal{P}(\lambda)$ is multiplicity free.
Furthermore, it is one dimensional if and only if $\mathcal{L}(\mu)$ appears in the head of the radical of $\mathcal{P}(\lambda)$.
The computation of the head of the radical of $\mathcal{P}(\lambda)$ can be done in $\K_n\mmod$ by \cref{mainthm}, meaning that we need to determine all $\underline{\mu}\overline{\lambda}$ (with $\mu\neq\lambda$) such that
\begin{equation}\label{nonsplittable}
	\underline{\mu}\overline{\lambda} = \underline{\mu}\overline{\kappa}\cdot\underline{\kappa}\overline{\lambda}\quad\implies\quad\kappa\in\{\mu,\lambda\}.
\end{equation}
\begin{defi}\label{deficircdiagprimitive}
	An orientable circle diagram $\underline{\mu}\overline{\lambda}$ with $\mu\neq\lambda$ is called \emph{$\Delta$-primitive} if $\mu\overline{\lambda}$ a $\Delta$-orientation such that 
	\begin{enumerate}[label=($\Delta$\arabic*)]
		\item\label{deltatwooff} There is at most one cap in $\overline{\lambda}$ such that its corresponding dot is not at the rightmost possible position,
		\item\label{deltaonedotmovestwice} it is \begin{tikzpicture}[scale=0.5]
			\CAP{5}\SETCOORD{-1}{0}\CAP{-1}\SETCOORD{-1}{0}\CAP{-1}\node at (0.5,0) {$\bullet$};
		\end{tikzpicture} avoiding, and
		\item\label{deltaextra} it is \begin{tikzpicture}[scale=0.5]
			\CAP{1}\SETCOORD{1}{0}\CAP{3}\SETCOORD{-1}{0}\CAP{-1}\node at (2.5,0) {$\bullet$};
		\end{tikzpicture} avoiding.
	\end{enumerate}
	It is called $\nabla$-primitive if $\underline{\mu}\lambda$ is a $\nabla$-orientation such that
	\begin{enumerate}[label=($\nabla$\arabic*)]
		\item\label{nablatwooff} There is at most one cup in $\overline{\lambda}$ such that its corresponding dot is not at the rightmost possible position, and
		\item\label{nablaextra} it is \begin{tikzpicture}[scale=0.5]
			\CUP{3}\SETCOORD{-1}{0}\CUP{-1}\node at (0.5,0) {$\bullet$};
		\end{tikzpicture} avoiding, where the outer cup is not contained in any other cup.
	\end{enumerate}

	We call $\underline{\mu}\overline{\lambda}$ \emph{primitive} if it is either $\Delta$- or $\nabla$-primitive.
\end{defi}
\begin{thm}
	Let $\underline{\mu}\overline{\lambda}$ be an orientable circle diagram. 
	Then $\underline{\mu}\overline{\lambda}$ satisfies \eqref{nonsplittable} if and only if it is primitive.
\end{thm}
\begin{proof}
	First we prove that every primitive circle diagram satisfies \eqref{nonsplittable} by contraposition.
	
	For this consider an oriented circle diagram 
	\begin{equation*}
		\underline{\mu}\delta\overline{\lambda} = \underline{\mu}\nu\overline{\eta}\cdot\underline{\eta}\kappa\overline{\lambda}
	\end{equation*}
	that violates \eqref{nonsplittable}.
	We have to show that $\underline{\mu}\delta\overline{\lambda}$ is neither $\Delta$- nor $\nabla$-primitive.

	By \cref{thmpropbasedquher}\cref{thmpropbasedquheri} we know that 
	\begin{equation*}
		\underline{\mu}\delta\overline{\lambda} = \underline{\mu}\nu\overline{\nu}\cdot\underline{\nu}\kappa'\overline{\lambda}
	\end{equation*}
	for some $\kappa'$.
	Now using \cref{thmpropbasedquher}\cref{thmpropbasedquherii} we know that $\kappa'\geq\lambda$, $\nu\geq\mu$ and $\delta\geq\kappa',\nu$ (this last inequality uses \cite{BS21}*{Lemma 5.5}).
	If $\kappa'\neq\lambda$ and $\nu\neq\mu$ we have $\delta\neq\mu,\lambda$, and thus it can neither be $\Delta$- nor $\nabla$-primitive as any orientation is unique by \cref{circlediaghasuniqueorientation}.
	So either $\kappa'=\lambda$ or $\nu = \mu$.
	\begin{itemize}
		\item Suppose that $\kappa'=\lambda$ and $\nu\neq\mu$.
			This means we have 
			\begin{equation*}
				\underline{\mu}\delta\overline{\lambda} = \underline{\mu}\nu\overline{\nu}\cdot\underline{\nu}\lambda\overline{\lambda}.
			\end{equation*}
			Observe furthermore that $\lambda$,$\nu$,and $\mu$ are pairwise distinct by assumption, and thus $\delta\geq\lambda>\nu>\mu$ by \cref{thmpropbasedquher}\cref{thmpropbasedquherii}.

			As $\delta>\mu$ this cannot be $\Delta$-primitive.
			If $\delta$ and $\mu$ differ by at least two dots it cannot be $\nabla$-primitive as it violates \cref{nablatwooff}.
			If they differ by exactly one dot, observe that this dot is uniquely defined (there will be a sequence of neighbored cups that have their associated dots in between them (so one is missing), the additional dot has to be to the left of this sequence).
			In case this sequence contains more than one cup, we violate \cref{nablatwooff}.
			If this sequence contains exactly one cup, we arrive at a contradiction as there cannot exist a $\mu<\nu<\delta$.
	\end{itemize}
	So we may assume that $\nu=\mu$, i.e.\ we are reduced to the case
	\begin{equation*}
		\underline{\mu}\delta\overline{\lambda} = \underline{\mu}\mu\overline{\eta}\cdot\underline{\eta}\kappa\overline{\lambda}
	\end{equation*}
	Using again \cref{thmpropbasedquher}\cref{thmpropbasedquheri} we have
	\begin{equation*}
		\underline{\mu}\delta\overline{\lambda} = \underline{\mu}\eta'\overline{\kappa}\cdot\underline{\kappa}\kappa\overline{\lambda}
	\end{equation*}
	for some $\eta'$.
	By the same reasoning as above we have $\delta\geq\eta',\kappa$, $\eta'\geq\mu$ and $\kappa\geq\lambda$.
	Thus, if $\eta'\geq\mu$ and $\kappa\geq\lambda$ the circle diagram cannot be primitive.
	So either $\eta'=\mu$ or $\kappa=\lambda$.
	\begin{itemize}
		\item Suppose that $\eta'=\mu$ and $\kappa\neq\lambda$.
			This means we have \begin{equation*}
				\underline{\mu}\delta\overline{\lambda} = \underline{\mu}\mu\overline{\kappa}\cdot\underline{\kappa}\kappa\overline{\lambda}
			\end{equation*}
			Observe furthermore that $\lambda$, $\kappa$ and $\mu$ are pairwise distinct, and thus $\delta\geq\mu>\kappa>\lambda$ by \cref{thmpropbasedquher}\cref{thmpropbasedquherii}.
			Therefore, this cannot be $\nabla$-primitive.
			If $\delta$ and $\lambda$ differ by at least two dots, it violates \cref{deltatwooff}, so it cannot be $\Delta$-primitive.
			If they differ by exactly one dot look at the cap in $\overline{\lambda}$ that corresponds to this dot.
			As $\delta>\kappa>\lambda$, this dot has to move two times along a valid $\Delta$-orientaion, and thus it violates \cref{deltaonedotmovestwice} and cannot be $\Delta$-primitive.
	\end{itemize}
	Thus, we can also assume that $\kappa=\lambda$, meaning that we are reduced to showing that 
	\begin{equation*}
		\underline{\mu}\delta\overline{\lambda} = \underline{\mu}\mu\overline{\eta}\cdot\underline{\eta}\lambda\overline{\lambda}
	\end{equation*}
	can be neither $\Delta$- nor $\nabla$-primitive.
	Note that by assumption $\eta\neq\lambda,\mu$ and furthermore $\mu\neq\lambda$ as otherwise it is not primitive by definition.
	Furthermore, if $\delta\neq\lambda,\mu$ it cannot be primitive by definition.
	\begin{itemize}
		\item Suppose that $\delta=\lambda>\mu$.
		In this case $\underline{\mu}\overline{\lambda}$ cannot be $\Delta$-primitive.
		We may also assume that it satisfies \cref{nablatwooff} as otherwise it cannot be $\nabla$-primitive either.
		This means that in the multiplication process no dot in the top number line moves.
		As $\mu\neq\eta$ we find a situation \begin{tikzpicture}[scale=0.5]
			\CAP{3}\SETCOORD{-1}{0}\CAP{-1}\node at (0.5,0) {$\bullet$};
		\end{tikzpicture} in $\mu\overline{\eta}$.
		This means that considering the total diagram $\underline{\mu}\mu\overline{\eta}$ it looks like
		
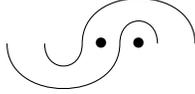
\begin{equation*}
			\begin{tikzpicture}[scale=0.5]
				\CUP{1}
				\CAP{3}
				\SETCOORD{-1}{0}
				\CAP{-1}
				\CUP{-3}
				\node at (1.5, 0) {$\bullet$};
				\node at (2.5, 0) {$\bullet$};
			\end{tikzpicture}
		\end{equation*}
		We may assume that this outer cup is not contained in any other cup, as otherwise we would find the same pattern for the cup that contains this.
		Now consider the surgery procedure involving the outer cap.
		This cannot be of the second kind from \cref{figureStraighteningwithorient} as the inner cap afterwards would be of the first kind with the additional white dots.
		This would mean that a dot in the top number line moves which contradicts our assumption.
		Therefore, it has to be of the third kind, which means that the left dot moves to the other side of the inner cup (and stays there).

		But then $\underline{\mu}\lambda\overline{\lambda}$ violates \cref{nablaextra}.
		\item Suppose that $\delta=\mu>\lambda$.
		In this case $\underline{\mu}\overline{\lambda}$ cannot be $\nabla$-primitive.
		We may also assume that it satisfies \cref{deltatwooff} and \cref{deltaonedotmovestwice} as otherwise it cannot be $\nabla$-primitive either.
		This means that in the multiplication process no dot in the bottom number line moves.
		As $\lambda\neq\eta$ there exists a subpicture like \begin{tikzpicture}[scale=0.5]
			\CUP{1}\node at (-0.5,0) {$\bullet$};
		\end{tikzpicture} where there appears not dot to the right of this cup.
		Observe now that no surgery procedure of the third kind from \cref{figureStraighteningwithorient} can occur, as no dot in the bottom number line can move by assumption.
		But the other two kinds of straightenings require a dot to the right of the cup.
		This means that the one dot difference between $\lambda$ and $\mu$ has to move via a surgery procedure of the first kind (as a white dot) and move to the right of this cup.
		Therefore, our situation looks like
		
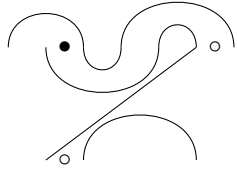
\begin{equation*}
			\begin{tikzpicture}[scale=0.5]
				\CAP{2}
				\CUP{1}
				\CAP{3}
				\SETCOORD{-5}{0}
				\CUP{3}
				\CAP{1}
				\LINE{-4}{-3}
				\SETCOORD{1}{0}
				\CAP{3}
				\node at (1.5,0) {$\bullet$};
				\node at (5.5,0) {$\circ$};
				\node at (1.5,-3) {$\circ$};
			\end{tikzpicture}
		\end{equation*}
		Now the surgery procedure of the first kind moves the white dot to the other side of the cap (and the dot stays there), so that $\underline{\mu}\mu\overline{\lambda}$ violates \cref{deltaextra}, therefore it cannot be $\Delta$-primitive.
	\end{itemize}
	This finishes the first half of the proof.
	We are left to show that any circle diagram satisfying \eqref{nonsplittable} is primitive.
	We prove this also by contraposition, i.e.\ we show that any non-primitive diagram does not satisfy \eqref{nonsplittable}.
	For this, we split the proof into three different cases.
	\begin{enumerate}
		\item Suppose that $\lambda$ is not a $\nabla$-orientation of $\underline{\mu}$ and $\mu$ is not a $\Delta$-orientation of $\overline{\lambda}$.
		Now let $\eta$ be the orientation of $\underline{\mu}\eta\overline{\lambda}$.
		By our assumption we have $\eta\neq\lambda,\mu$.
		By \cref{thmpropbasedquher}\cref{thmpropbasedquheri} we have 
		\begin{equation*}
			\underline{\mu}\eta\overline{\lambda} = \underline{\mu}\eta\overline{\eta}\cdot\underline{\eta}\eta\overline{\lambda}
		\end{equation*}
		violating \eqref{nonsplittable}.
		\item Suppose that $\lambda$ is a valid $\nabla$-orientation of $\underline{\mu}$.
		\begin{enumerate}[label=(\alph*)]
			\item Suppose first that this violates \cref{nablatwooff}.
				Let $\kappa$ be the same weight as $\lambda$ except that we replace the rightmost dot of $\lambda$ that corresponds to a left endpoint of a cup with the dot to the right of the same cup.
				By construction $\underline{\kappa}\lambda\overline{\lambda}$ is oriented (and only one dot in $\kappa$ and $\lambda$ differ).
				Furthermore, $\underline{\mu}\kappa\overline{\kappa}$ is also oriented and $\kappa$ and $\mu$ have one dot more in common than $\lambda$ and $\mu$.
				We claim that 
				\begin{equation*}
					\underline{\mu}\kappa\overline{\kappa}\cdot\underline{\kappa}\lambda\overline{\lambda}=\underline{\mu}\lambda\overline{\lambda}
				\end{equation*}
				For this it suffices to show that any appearing surgery procedure is a straightening.
				First look at the surgery procedures that involve a cup that has a dot to its right.
				This means there has to be a cap above this dot and because the orientation of $\underline{\kappa}\overline{\lambda}$ is $\lambda$, we have to have a straightening.
				There is exactly one surgery procedure without a dot to the right of the cup.
				The situation then looks like
				
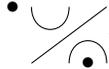
\begin{equation*}
					\begin{tikzpicture}[scale=0.5]
						\CUP{1}
						\SETCOORD{1}{0}
						\LINE{-2}{-1.5}
						\SETCOORD{1}{0}
						\CAP{1}
						\node at (-0.5,0) {$\bullet$};
						\node at (1.5,-1.5) {$\bullet$};
					\end{tikzpicture}.
				\end{equation*}
				By construction $\kappa$ and $\lambda$ agree up to one dot (which is exactly the one depicted above), and thus all the other dots agree.
				But now there cannot be a dot between the line and the cap as it would have to correspond to one to the right of the line at the top number line, which would mean that $\kappa$ and $\lambda$ differ by at least two dots.
				Additionally, by construction (as the dot on the top is to the left of the cup) the dot on the bottom has to be to the right of a cup.
				As there is no dot between the cap and the line, this means that this cup connects the line with the cap, meaning that we have a straightening.
				Thus, the multiplication is nonzero, therefore violating \cref{nonsplittable}.
			\item Suppose now that \cref{nablatwooff} holds but \cref{nablaextra} is violated.
				Then our situation looks like 
				
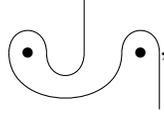
\begin{equation*}
					\begin{tikzpicture}[scale=0.5]
						\LINE{0}{-1.5}
						\CUP{-1}
						\CAP{-1}
						\CUP{3}
						\CAP{1}
						\LINE{0}{-1.5}
						\node at (-1.5,-1.5) {$\bullet$};
						\node at (1.5,-1.5) {$\bullet$};
					\end{tikzpicture},
				\end{equation*}
				and it is easy to check that 
				
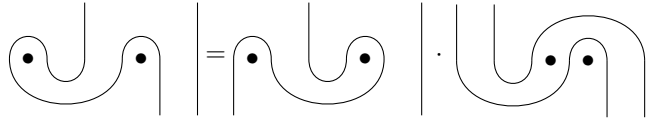
\begin{equation*}
						\begin{tikzpicture}[scale=0.5]
							\LINE{0}{-1.5}
							\CUP{-1}
							\CAP{-1}
							\CUP{3}
							\CAP{1}
							\LINE{0}{-1.5}
							\SETCOORD{1}{3}
							\LINE{0}{-3}
							\node at (-1.5,-1.5) {$\bullet$};
							\node at (1.5,-1.5) {$\bullet$};
						\end{tikzpicture}
						=
						\begin{tikzpicture}[scale=0.5]
							\LINE{0}{-1.5}
							\CUP{1}
							\CAP{1}
							\CUP{-3}
							\CAP{-1}
							\LINE{0}{-1.5}
							\SETCOORD{5}{3}
							\LINE{0}{-3}
							\node at (-1.5,-1.5) {$\bullet$};
							\node at (1.5,-1.5) {$\bullet$};
						\end{tikzpicture}
						\;\cdot\;
						\begin{tikzpicture}[scale=0.5]
							\LINE{0}{-1.5}
							\CUP{3}
							\CAP{1}
							\LINE{0}{-1.5}
							\SETCOORD{-3}{3}
							\LINE{0}{-1.5}
							\CUP{1}
							\CAP{3}
							\LINE{0}{-1.5}
							\node at (2.5,-1.5) {$\bullet$};
							\node at (3.5,-1.5) {$\bullet$};
						\end{tikzpicture}
				\end{equation*}
				provides a counterexample to \eqref{nonsplittable}.
		\end{enumerate}
		\item Suppose now that $\mu$ is a valid $\Delta$-orientation of $\overline{\lambda}$
		\begin{enumerate}[label=(\alph*)]
			\item First assume that \cref{deltatwooff} is violated.
			As \cref{deltatwooff} does not hold, there exists a dot in $\mu$ that is not at the right end below the corresponding cap $C$ in $\overline{\lambda}$.
			Of these dots we choose the rightmost one such that any caps contained in $C$ have their dot at the rightmost position.
			Let $\kappa$ be the same as $\mu$ except that we replace this dot by the rightmost position in $C$.
			By \cref{deltatwooff} we have $\kappa\neq\lambda,\mu$. 
			The weight $\kappa$ is a $\Delta$-orienation of $\overline{\lambda}$ by construction, and thus $\underline{\kappa}\kappa\overline{\lambda}$ is oriented.
			Furthermore, our choice ensures that $\mu$ is a valid $\Delta$-orientation of $\overline{\kappa}$, and thus $\underline{\mu}\mu\overline{\kappa}$ is oriented.
			We claim now that
			\begin{equation*}
				\underline{\mu}\mu\overline{\kappa}\cdot\underline{\kappa}\kappa\overline{\lambda}= \underline{\mu}\mu\overline{\lambda}
			\end{equation*}
			For this we show that no split or reconnect can occur as a surgery procedure.
			If a split occurred, this would look like one of the following two possibilities:
			\begin{center}
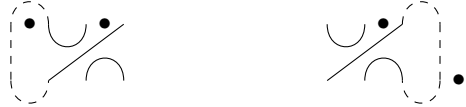
				\begin{tabularx}{0.7\textwidth}{>{\centering\arraybackslash}X >{\centering\arraybackslash}X}
					\begin{tikzpicture}[scale=0.5]
						\CUP{-1}
						\CAP[dashed]{-1}
						\LINE[dashed]{0}{-1.5}
						\CUP[dashed]{1}
						\LINE{2}{1.5}
						\SETCOORD{-1}{-1.5}
						\CAP{1}
						\node at (-1.5,0) {$\bullet$};
						\node at (0.5,0) {$\bullet$};
					\end{tikzpicture}&
					\begin{tikzpicture}[scale=0.5]
						\CUP{1}
						\SETCOORD{-1}{-1.5}
						\LINE{2}{1.5}
						\CAP[dashed]{1}
						\LINE[dashed]{0}{-1.5}
						\CUP[dashed]{-1}
						\CAP{-1}
						\node at (1.5,0) {$\bullet$};
						\node at (3.5,-1.5) {$\bullet$};
					\end{tikzpicture}
				\end{tabularx}
			\end{center}
			The first case cannot appear as every dot is to the right of a cup by assumption.
			This means that the left dot in the first case cannot be matched with the cup to its right, so it has to be matched to a dot at the bottom number line.
			This would be contained in the dashed cup, which is not possible by definition of orientation.

			The second case also cannot appear, as the bottom dot is not contained in a cap, and thus has to be matched to a dot at the top.
			But then these two dots at the top need to lie inside the same cap, which contradicts the definition of orientation.
			Therefore, there cannot appear any split.

			Now if a reconnect would appear the situation would look like
			
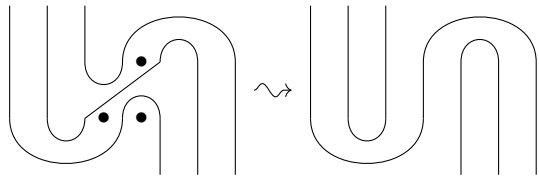
\begin{equation*}
				\begin{tikzpicture}[scale=0.5]
					\LINE{0}{-3}
					\CUP{3}
					\CAP{1}
					\LINE{0}{-1.5}
					\SETCOORD{-3}{4.5}
					\LINE{0}{-3}
					\CUP{1}
					\LINE{2}{1.5}
					\CAP{1}
					\LINE{0}{-3}
					\SETCOORD{-3}{4.5}
					\LINE{0}{-1.5}
					\CUP{1}
					\CAP{3}
					\LINE{0}{-3}
					\SETCOORD{2}{4.5}
					\draw[->, decorate, decoration={snake}] (6.5, -2.25) -- (7.5, -2.25);
					\LINE{0}{-3}
					\CUP{3}
					\LINE{0}{1.5}
					\CAP{3}
					\LINE{0}{-3}
					\SETCOORD{-5}{4.5}
					\LINE{0}{-3}
					\CUP{1}
					\LINE{0}{3}
					\SETCOORD{2}{-4.5}
					\LINE{0}{3}
					\CAP{1}
					\LINE{0}{-3}
					\node at (3.5,-1.5) {$\bullet$};
					\node at (2.5,-3) {$\bullet$};
					\node at (3.5,-3) {$\bullet$};
				\end{tikzpicture}
			\end{equation*}
			But any further surgery procedure that we could apply (even though it is already not orientable) would preserve at least one of the non-propagating lines (and as this number has to be even), this would imply that $\underline{\mu}\overline{\lambda}$ contain non-propagating lines which contradicts our assumption.
			Therefore, no reconnect can appear.
			Thus, only straightenings can occur, and thus the multiplication is nonzero.
			\item Suppose now that \cref{deltatwooff} holds but \cref{deltaonedotmovestwice} is violated.
			Then our situation looks like
			
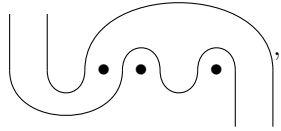
\begin{equation*}
				\begin{tikzpicture}[scale=0.5]
					\LINE{0}{-1.5}
					\CUP{1}
					\CAP{5}
					\LINE{0}{-1.5}
					\SETCOORD{-7}{3}
					\LINE{0}{-1.5}
					\CUP{3}
					\CAP{1}
					\CUP{1}
					\CAP{1}
					\LINE{0}{-1.5}
					\node at (1.5, -1.5) {$\bullet$};
					\node at (2.5, -1.5) {$\bullet$};
					\node at (4.5, -1.5) {$\bullet$};
				\end{tikzpicture},
			\end{equation*}
			and it is easy to check that 
			
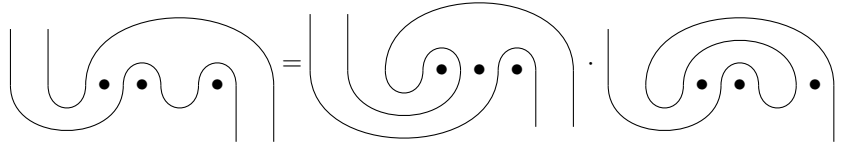
\begin{equation*}
				\begin{tikzpicture}[scale=0.5]
					\LINE{0}{-1.5}
					\CUP{1}
					\CAP{5}
					\LINE{0}{-1.5}
					\SETCOORD{-7}{3}
					\LINE{0}{-1.5}
					\CUP{3}
					\CAP{1}
					\CUP{1}
					\CAP{1}
					\LINE{0}{-1.5}
					\node at (1.5, -1.5) {$\bullet$};
					\node at (2.5, -1.5) {$\bullet$};
					\node at (4.5, -1.5) {$\bullet$};
				\end{tikzpicture}
				=
				\begin{tikzpicture}[scale=0.5]
					\LINE{0}{-1.5}
					\CUP{5}
					\CAP{1}
					\LINE{0}{-1.5}
					\SETCOORD{-5}{3}
					\LINE{0}{-1.5}
					\CUP{3}
					\CAP{-1}
					\CUP{-1}
					\CAP{5}
					\LINE{0}{-1.5}
					\node at (3.5, -1.5) {$\bullet$};
					\node at (4.5, -1.5) {$\bullet$};
					\node at (5.5, -1.5) {$\bullet$};
				\end{tikzpicture}
				\;\cdot\;
				\begin{tikzpicture}[scale=0.5]
					\LINE{0}{-1.5}
					\CUP{3}
					\CAP{1}
					\CUP{1}
					\CAP{-3}
					\CUP{-1}
					\CAP{5}
					\LINE{0}{-1.5}
					\SETCOORD{1}{3}
					\LINE{0}{-3}
					\node at (2.5, -1.5) {$\bullet$};
					\node at (3.5, -1.5) {$\bullet$};
					\node at (5.5, -1.5) {$\bullet$};
				\end{tikzpicture}
			\end{equation*}
			provides a counterexample to \eqref{nonsplittable}
			\item Lastly suppose that \cref{deltatwooff} and \cref{deltaonedotmovestwice} hold but \cref{deltaextra} does not.
			Then the circle diagram looks like
			
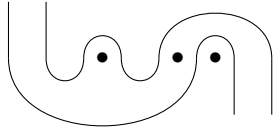
\begin{equation*}
				\begin{tikzpicture}[scale=0.5]
					\LINE{0}{-1.5}
					\CUP{5}
					\CAP{1}
					\LINE{0}{-1.5}
					\SETCOORD{-5}{3}
					\LINE{0}{-1.5}
					\CUP{1}
					\CAP{1}
					\CUP{1}
					\CAP{3}
					\LINE{0}{-1.5}
					\node at (2.5, -1.5) {$\bullet$};
					\node at (4.5, -1.5) {$\bullet$};
					\node at (5.5, -1.5) {$\bullet$};
				\end{tikzpicture}
			\end{equation*}
			and a counterexample for \eqref{nonsplittable} is given by
			
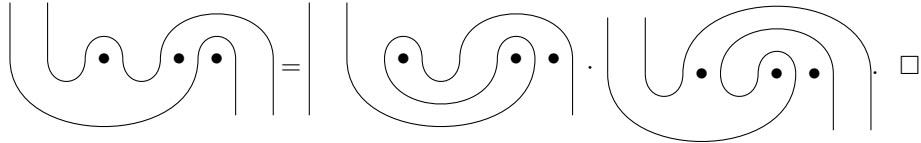
\begin{equation*}
				\begin{tikzpicture}[scale=0.5]
					\LINE{0}{-1.5}
					\CUP{5}
					\CAP{1}
					\LINE{0}{-1.5}
					\SETCOORD{-5}{3}
					\LINE{0}{-1.5}
					\CUP{1}
					\CAP{1}
					\CUP{1}
					\CAP{3}
					\LINE{0}{-1.5}
					\node at (2.5, -1.5) {$\bullet$};
					\node at (4.5, -1.5) {$\bullet$};
					\node at (5.5, -1.5) {$\bullet$};
				\end{tikzpicture}
				=
				\begin{tikzpicture}[scale=0.5]
					\LINE{0}{-3}
					\SETCOORD{1}{3}
					\LINE{0}{-1.5}
					\CUP{5}
					\CAP{-1}
					\CUP{-3}
					\CAP{1}
					\CUP{1}
					\CAP{3}
					\LINE{0}{-1.5}
					\node at (2.5, -1.5) {$\bullet$};
					\node at (5.5, -1.5) {$\bullet$};
					\node at (6.5, -1.5) {$\bullet$};
				\end{tikzpicture}
				\;\cdot\;
				\begin{tikzpicture}[scale=0.5]
					\LINE{0}{-1.5}
					\CUP{5}
					\CAP{-1}
					\CUP{-1}
					\CAP{3}
					\LINE{0}{-1.5}
					\SETCOORD{-5}{3}
					\LINE{0}{-1.5}
					\CUP{1}
					\CAP{5}
					\LINE{0}{-1.5}
					\node at (2.5, -1.5) {$\bullet$};
					\node at (4.5, -1.5) {$\bullet$};
					\node at (5.5, -1.5) {$\bullet$};
				\end{tikzpicture}.\qedhere
			\end{equation*}
		\end{enumerate}
	\end{enumerate}
\end{proof}
This explicit description of diagrams satisfying \eqref{nonsplittable} gives us immediately the following corollary.
\begin{cor}
	We have 
	\begin{equation*}
		\dim\Ext_{\lie{p}(n)}(\mathcal{L}(\lambda), \mathcal{L}(\mu))=\begin{cases}
			1&\text{if $\underline{\mu}\overline{\lambda}$ is primitive,}\\
			0&\text{otherwise.}
		\end{cases}
	\end{equation*}
\end{cor}
\subsubsection{\texorpdfstring{$\lie{p}(1)$}{p(1)} as a quiver with relations}
The irreducible modules of $\lie{p}(1)\mmod$ are labelled by $\bbZ$.
The category decomposes into $2$ blocks $B_0$, $B_1$, where $\mathcal{L}(i)\in B_0$ if $i$ is even and $\mathcal{L}(i)\in B_1$ if $i$ is odd.
The module $\mathcal{P}(i)$ has irreducible head $\mathcal{L}(i)$ and irreducible socle $\mathcal{L}(i+2)$ and nothing else.
Both blocks can be described by the same quiver
\begin{equation*}	
	\begin{tikzpicture}[baseline={([yshift=-.8ex]current bounding box.center)}, scale=0.75]
		\foreach \x in {0,2,4,6,8} \node at (\x, 0) {$\bullet$};
		\foreach \x in {0,2,4,6} \draw[->, shorten >=0.3cm, shorten <=0.3cm] (\x, 0)--(\x+2,0);
		\node[xshift=-0.6cm] at (0,0) {$\dots$};
		\node[xshift=0.6cm] at (8,0) {$\dots$};
	\end{tikzpicture}
\end{equation*}
with relation $\tikz[baseline={([yshift=-.8ex]current bounding box.center)}, scale=0.75]{\draw[->, shorten >=0.3cm, shorten <=0.3cm] (0, 0)--(2,0);\draw[->, shorten >=0.3cm, shorten <=0.3cm] (2, 0) node {$\bullet$}--(4,0);}=0$.
\subsubsection{\texorpdfstring{$\lie{p}(2)$}{p(2)} as a quiver with relations}
The irreducibles for $\lie{p}(2)$ are labelled by two integers $(i,j)$ with $i>j$.
The category decomposes into three blocks $B_0$, $B_1$ and $B_2$ depending on how many odd entries there are.
We only consider the block $B_1$ as it is the most irregular one ($B_0$ and $B_1$ can be obtained by removing the leftmost diagonal).
It can be described as the quiver
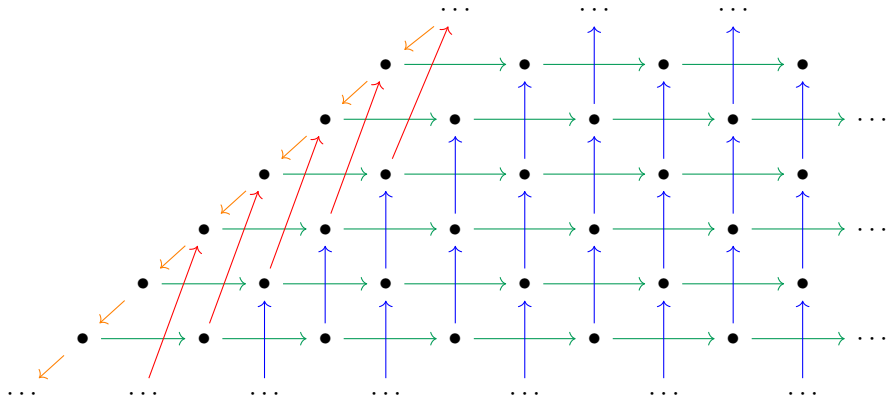
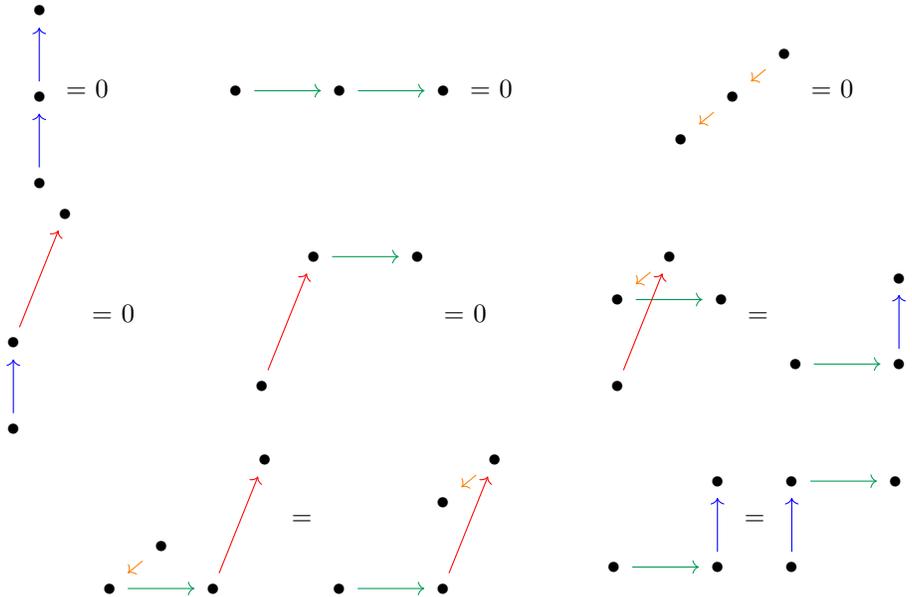
\begin{center}
\begin{tikzcd}[row sep = small, column sep = tiny]
	&&&&&&&\cdots\arrow[dl, orange]&&\cdots&&\cdots&&\\
	&&&&&&\bullet\arrow[dl, orange]\arrow[rr, ForestGreen]&&\bullet\arrow[rr, ForestGreen]&&\bullet\arrow[rr, ForestGreen]&&\bullet&\\
	&&&&&\bullet\arrow[dl, orange]\arrow[rr, ForestGreen]&&\bullet\arrow[rr, ForestGreen]&&\bullet\arrow[rr, ForestGreen]\arrow[uu, blue]&&\bullet\arrow[rr, ForestGreen]\arrow[uu, blue]&&\cdots\\
	&&&&\bullet\arrow[dl, orange]\arrow[rr, ForestGreen]&&\bullet\arrow[uuur, red]\arrow[rr, ForestGreen]&&\bullet\arrow[rr, ForestGreen]\arrow[uu, blue]&&\bullet\arrow[rr, ForestGreen]\arrow[uu, blue]&&\bullet\arrow[uu, blue]&\\
	&&&\bullet\arrow[dl, orange]\arrow[rr, ForestGreen]&&\bullet\arrow[uuur, red]\arrow[rr, ForestGreen]&&\bullet\arrow[rr, ForestGreen]\arrow[uu, blue]&&\bullet\arrow[rr, ForestGreen]\arrow[uu, blue]&&\bullet\arrow[rr, ForestGreen]\arrow[uu, blue]&&\cdots\\
	&&\bullet\arrow[dl, orange]\arrow[rr, ForestGreen]&&\bullet\arrow[uuur, red]\arrow[rr, ForestGreen]&&\bullet\arrow[rr, ForestGreen]\arrow[uu, blue]&&\bullet\arrow[rr, ForestGreen]\arrow[uu, blue]&&\bullet\arrow[rr, ForestGreen]\arrow[uu, blue]&&\bullet\arrow[uu, blue]&\\
	&\bullet\arrow[dl, orange]\arrow[rr, ForestGreen]&&\bullet\arrow[uuur, red]\arrow[rr, ForestGreen]&&\bullet\arrow[rr, ForestGreen]\arrow[uu, blue]&&\bullet\arrow[rr, ForestGreen]\arrow[uu, blue]&&\bullet\arrow[rr, ForestGreen]\arrow[uu, blue]&&\bullet\arrow[rr, ForestGreen]\arrow[uu, blue]&&\cdots\\
	\cdots&&\cdots\arrow[uuur, red]&&\cdots\arrow[uu, blue]&&\cdots\arrow[uu, blue]&&\cdots\arrow[uu, blue]&&\cdots\arrow[uu, blue]&&\cdots\arrow[uu, blue]&
 \end{tikzcd}
\end{center}
with relations

 \begin{tabularx}{\textwidth}{c>{\centering\arraybackslash}Xc>{\centering\arraybackslash}Xc}
   \begin{tikzcd}[ampersand replacement = \&, row sep = tiny, column sep = tiny]
	 \bullet\\\phantom{\bullet}\\\bullet\arrow[uu, blue]\\\phantom{\bullet}\\\bullet\arrow[uu, blue]
   \end{tikzcd} = 0&&
   \begin{tikzcd}[ampersand replacement = \&, row sep = tiny, column sep = tiny]
	 \bullet\arrow[rr, ForestGreen]\&\phantom{\bullet}\&\bullet\arrow[rr, ForestGreen]\&\phantom{\bullet}\&\bullet
   \end{tikzcd} = 0&&
   \begin{tikzcd}[ampersand replacement = \&, row sep = tiny, column sep = tiny]
	\&\&\bullet\arrow[dl, orange]\\
	\&\bullet\arrow[dl, orange]\\
	\bullet
  \end{tikzcd} = 0\\
	\begin{tikzcd}[ampersand replacement = \&, row sep = tiny, column sep = tiny]
	  \&\bullet\\\phantom{\bullet}\\\phantom{\bullet}\\
	  \bullet\arrow[uuur, red]\\\phantom{\bullet}\\
	  \bullet\arrow[uu, blue]
	\end{tikzcd} = 0&&
	\begin{tikzcd}[ampersand replacement = \&, row sep = tiny, column sep = tiny]
	  \&\bullet\arrow[rr, ForestGreen]\&\phantom{\bullet}\&\bullet\\\phantom{\bullet}\\\phantom{\bullet}\\
	  \bullet\arrow[uuur, red]
	\end{tikzcd} = 0&&
	\begin{tikzcd}[ampersand replacement = \&, row sep = tiny, column sep = tiny]
	 \&\bullet\arrow[dl, orange]\\
	 \bullet\arrow[rr, ForestGreen]\&\phantom{\bullet}\&\bullet\\
	 \phantom{\bullet}\\
	 \bullet\arrow[uuur, red]
   \end{tikzcd} = \begin{tikzcd}[ampersand replacement = \&, row sep = tiny, column sep = tiny]
	 \&\phantom{\bullet}\&\bullet\\
	 \phantom{\bullet}\\
	 \bullet\arrow[rr, ForestGreen]\&\&\bullet\arrow[uu, blue]
   \end{tikzcd}\\
   \multicolumn{3}{r}{\begin{tikzcd}[ampersand replacement = \&, row sep = tiny, column sep = tiny]
	\&\phantom{\bullet}\&\phantom{\bullet}\&\bullet\\
	\phantom{\bullet}\\
	\&\bullet\arrow[dl, orange]\\
	\bullet\arrow[rr, ForestGreen]\&\&\bullet\arrow[uuur, red]
  \end{tikzcd} = \begin{tikzcd}[ampersand replacement = \&, row sep = tiny, column sep = tiny]
	\&\phantom{\bullet}\&\phantom{\bullet}\&\bullet\arrow[dl, orange]\\
	\&\&\bullet\\
	\phantom{\bullet}\\
	\bullet\arrow[rr, ForestGreen]\&\&\bullet\arrow[uuur, red]
  \end{tikzcd}}
&&
   \begin{tikzcd}[ampersand replacement = \&, row sep = tiny, column sep = tiny]
	 \&\phantom{\bullet}\&\bullet\\\phantom{\bullet}\\\bullet\arrow[rr, ForestGreen]\&\&\bullet\arrow[uu, blue]
   \end{tikzcd} = \begin{tikzcd}[ampersand replacement = \&, row sep = tiny, column sep = tiny]
	 \bullet\arrow[rr, ForestGreen]\&\phantom{\bullet}\&\bullet\\\phantom{\bullet}\\\bullet\arrow[uu, blue]
   \end{tikzcd}.
\end{tabularx}

Here, we also see that the last relation in the second row is not homogeneous, which reflects again the problems with the grading that occured in \cref{nonskoszulsection}.
\newpage
\appendix
\section{}
\subsection{Checking relations for \texorpdfstring{$\FunctorKLRtosBr$}{Psi}}\label{checkingrelationsklrtosbr}
\begin{proof}[Proof of \cref{functorXiexplicit}]
	For \cref{twicezero} we have to show that $M'\coloneqq i\mind\circ i\mind M= 0$ for every object $M\in\Rep(P)$.
    For every $M$ there exists $n\gg0$ such that $\End_{Rep(P)}(M')\to\End_{\lie{p}(n)}(\FunctorsBrtopn(M'))$ is faithful.
    In virtue of \cref{action} we have $\FunctorsBrtopn(M')=\Theta_k^{-}\Theta_k^{-}\FunctorsBrtopn(M)=0$ by \cite{BDE19}*{Thm.\ 4.5.1}.

	Note that \begin{tikzpicture}[line width = \lw, myscale=0.7]
		\draw (0,0) -- (1,1) (1,0)--(0,1);\node[fill=white, anchor=north] at (0,0) {$b$}; \node[fill=white, anchor=north] at (1,0) {$a$};\node[fill=white, anchor=south] at (0,1) {$a$}; \node[fill=white, anchor=south] at (1,1) {$b$};
	\end{tikzpicture} gets mapped to $0$ if $b =a-1$, and thus \cref{inverse} holds if $\abs{a-b}=1$.
If $\abs{a-b}>1$, then $\lambda_{b,a}((b-a)s+1)\circ\lambda_{a,b}((a-b)s+1)=\lambda_{b,a}\lambda_{a,b}(1-(a-b)^2)=1$ by \cref{scalarone},
thus \cref{inverse} holds.

For \cref{untwist} we have $\cap\circ\lambda_{a+1,a}(1+ s)\circ\mathrm{incl}=2\cap\circ\mathrm{incl}=0$, where the last line uses that $(\id^{\otimes k}\otimes\cap) e(i_1, \dots, i_k,j,l) = 0$ if $l\neq j+1$ by \cite{Cou18a}*{Lemma 6.3.1}.

Using the snake identities in $\sBr$ together with \cite{Cou18a}*{Lemma 6.3.1}, \cref{snake} follows.

Now observe that both sides of \cref{tangleone} give a zero morphism if $a=b,b+1$.
So assume that $a\neq b,b+1$.
Then we have
\begin{align*}
	LHS&=\pr\circ\lambda_{b,a}((b-a) s_k+1)\circ\cup_{k+1}\\
	&=\pr\circ(\lambda_{b,a}(b-a)s_k\cup_{k+1}+\lambda_{b,a}\cup_{k+1})\\
	&=\pr\circ\lambda_{b,a}(b-a)s_k\cup_{k+1}\\
	RHS&=\pr\circ\lambda_{a,b+1}((a-b-1) s_{k+1}+1)\circ\cup_k\\
	&=\pr\circ(\lambda_{a,b+1}(a-b-1)s_{k+1}\cup_k+\lambda_{a,b+1}\cup_k)\\
	&=\pr\circ\lambda_{a,b+1}(a-b-1)s_{k+1}\cup_k.
\end{align*}
In each case, the second summand vanishes as $a\neq b,b+1$ by \cite{Cou18a}*{Lemma 6.3.1}.
Now the equality follows from \cref{scalartwo} and the corresponding relation in the super Brauer category.

Additionally, consider the braid relation \cref{braid}, which trivially holds true if $b=a-1$, $c=b-1$, $c=a-1$ or $a$, $b$ and $c$ are not pairwise distinct.
Otherwise,
\begin{align*}
	LHS=&\lambda_{b,a}((b-a)s_{k-1}+1)\circ\lambda_{c,a}((c-a)s_{k}+1)\circ\lambda_{c,b}((c-b)s_{k-1}+1)\\
	=&\lambda_{b,a}\lambda_{c,a}\lambda_{c,b}({\color{blue}{(b-a)(c-a)(c-b)s_{k-1}s_{k}s_{k-1}}}+{\color{olive}{(c-a)(c-b)s_{k}s_{k-1}}}\\
	&+{\color{green}{(b-a)(c-b)s_{k-1}^2}}+{\color{cyan}{(c-b)s_{k-1}}}+{\color{red}{(b-a)(c-a)s_{k-1}s_{k}}}+{\color{orange}{(c-a)s_{k}}}\\
	&+{\color{cyan}{(b-a)s_{k-1}}}+{\color{violet}{1}}),\\
	RHS=&\lambda_{c,b}((c-b)s_{k}+1)\circ\lambda_{c,a}((c-a)s_{k-1}+1)\circ\lambda_{b,a}((b-a)s_{k}+1)\\
	=&\lambda_{c,b}\lambda_{c,a}\lambda_{b,a}({\color{blue}{(c-b)(c-a)(b-a)s_{k}s_{k-1}s_{k}}}+{\color{red}{(c-a)(b-a)s_{k-1}s_{k}}}\\
	&+{\color{green}{(c-b)(b-a)s_{k}^2}}+{\color{orange}{(b-a)s_{k}}}+{\color{olive}{(c-b)(c-a)s_{k}s_{k-1}}}+{\color{cyan}{(c-a)s_{k-1}}}\\
	&+{\color{orange}{(c-b)s_{k}}}+{\color{violet}{1}}).
\end{align*}
Note that the blue terms agree as the braid relation holds in the super Brauer category.

Similar to checking \cref{twicezero}, we find $n\gg0$ so that we can calculate with finite dimensional $\lie{p}(n)$-representations instead.
By \cite{BDE19}*{Thm.4.5.1} we know that the left-hand side of \cref{idempotentone} gets mapped to an isomorphism.
Furthermore, we have by \cref{snake}:
\begin{equation}
	\begin{tikzpicture}[line width = \lw, myscale=0.5]
		\node at (0,0) (A) {$a$};
		\node at (1,0) (B) {$a+1$};
		\node at (2,0) (C) {$a$};
		\node at (0,1.75) (D) {$a$};
		\node at (1,1.75) (E) {$a-1$};
		\node at (2,1.75) (F) {$a$};
		\node at (0,3.5) (G) {$a$};
		\node at (1,3.5) (H) {$a+1$};
		\node at (2,3.5) (I) {$a$};
		\draw (A)..controls +(0.2,0.8) and +(-0.2,0.8)..(B)(D)..controls +(0.2,-0.8) and +(-0.2,-0.8)..(E)(E)..controls +(0.2,0.8) and +(-0.2,0.8)..(F)(H)..controls +(0.2,-0.8) and +(-0.2,-0.8)..(I) (C)--(F) (D)--(G);	
		\node at (0,5.25) (J) {$a$};
		\node at (1,5.25) (K) {$a-1$};
		\node at (2,5.25) (L) {$a$};
		\node at (0,7) (M) {$a$};
		\node at (1,7) (N) {$a+1$};
		\node at (2,7) (O) {$a$};
		\draw (G)..controls +(0.2,0.8) and +(-0.2,0.8)..(H)(J)..controls +(0.2,-0.8) and +(-0.2,-0.8)..(K)(K)..controls +(0.2,0.8) and +(-0.2,0.8)..(L)(N)..controls +(0.2,-0.8) and +(-0.2,-0.8)..(O) (I)--(L) (J)--(M);
	\end{tikzpicture}=\eps\begin{tikzpicture}[line width = \lw, myscale=0.7]
	\node at (0,0) (A) {$a$};
	\node at (1,0) (B) {$a+1$};
	\node at (2,0) (C) {$a$};
	\node at (0,1.75) (D) {$a$};
	\node at (1,1.75) (E) {$a-1$};
	\node at (2,1.75) (F) {$a$};
	\node at (0,3.5) (G) {$a$};
	\node at (1,3.5) (H) {$a+1$};
	\node at (2,3.5) (I) {$a$};
	\draw (A)..controls +(0.2,0.8) and +(-0.2,0.8)..(B)(D)..controls +(0.2,-0.8) and +(-0.2,-0.8)..(E)(E)..controls +(0.2,0.8) and +(-0.2,0.8)..(F)(H)..controls +(0.2,-0.8) and +(-0.2,-0.8)..(I) (C)--(F) (D)--(G);
\end{tikzpicture}
\end{equation}
This means that $\eps LHS$ in \cref{idempotentone} gets mapped to an isomorphism and an idempotent, thus it has to be the identity.
Similar for \cref{idempotenttwo}.
\end{proof}
\subsection{Proof of Theorem~\ref{spanningsetisbasis}}\label{proofofspanningsetisbasis}
\begin{defi}
	We have a filtration $\{0\}=F_{\leq -1}\subseteq F_{\leq 0}\subseteq F_{\leq 1}\subseteq\dots$ on $\Tud$ given by $F_{\leq b} = \bigcup_{\abs{\lambda}\leq b}\Tud(\lambda)\sqcup\Tud(\lambda)$.
	Denote by $\Rl{}$ the image of $\Psi$ and by $\Rts{}$ the ideal in $\sG$ generated by $\Rl{}$.
	Let $\Rl{\leq b}\coloneqq\Psi(F_{\leq b})$ and $\Rts{\leq b}$ be the ideal in $\sG$ generated by $\Rl{\leq b}$.
    Analogously, define $\Rl{< b}$ and $\Rts{<b}$.
	Then $\Rl{\leq b}$ provides a filtration of $\Rl{}$ and $\Rts{\leq b}$ one of $\Rts{}$.
\end{defi}
The next lemma is directly implied by the definition of $\Psi_{\ta{t}\ta{s}}$.
\begin{lem}\label{thetaiaddablebox}
    Let $\ta{t}$, $\ta{s}\in\Tud(\lambda)$ of length $m$ and $n$ respectively and $\alpha\in\Add_i(\lambda)$.
    Then, $\theta_i(\Psi_{\ta{t}\ta{s}}) = \Psi_{\ta{u}\ta{v}}$, where $\ta{u}$, $\ta{v}\in\Tud(\lambda\oplus\alpha)$ of length $m+1$ and $n+1$ respectively such that $\restr{\ta{u}}{m}=\ta{t}$ and $\restr{\ta{v}}{n} = \ta{{s}}$.
\end{lem}
\begin{lem}\label{thetainotaddablebox}
	Let $\lambda$ be a partition of size $b$ and suppose that $\Add_i(\lambda)=\emptyset$.
	If $\Rl{<b'}=\Rts{<b'}$ for all $b'\leq b$, then $\theta_i(\Psi_{\ta{t}^\lambda\ta{t}^\lambda})\in \Rl{<b}$.
\end{lem}
\begin{rem}\label{thetainotaddableboxforgeneraltands}
    Let $\ta{t}$, $\ta{s}\in\Tud(\lambda)$ and suppose that $\Add_i(\lambda)=\emptyset$ and $\Rl{<b'}=\Rts{<b'}$ for all $b'\leq\abs{\lambda}$.
    Then \cref{thetainotaddablebox} implies $\theta_i(\Psi_{\ta{t}\ta{s}})\in \Rl{<\abs{\lambda}}$.
\end{rem}
\begin{proof}[Proof of \cref{thetainotaddablebox}]
	If $b=0$, the claim follows from \cref{firstzero}.

	Now assume that the claim holds for all $b'<b$. 
	Note that by \cref{thetaiaddablebox} as well as \cref{thetainotaddableboxforgeneraltands} we have proved on the way that $\restr{\theta_i}{\Rl{\leq b-1}}$ is a filtered map of degree $1$.
	Let $(i_1, \dots, i_b)\coloneqq\bm{i}^l_{\ta{t}^\lambda}=\bm{i}^r_{\ta{t}^\lambda}$.
	Also note that $\Psi_{\ta{t}^\lambda\ta{t}^\lambda}=\id_{(i_1, \dots, i_b)}$.
	We consider four different cases.
	\begin{enumerate}
		\item If $i=i_b$, we have $\theta_i(\Psi_{\ta{t}^\lambda\ta{t}^\lambda}) = 0$ by \cref{twicezero}.
		\item If $\abs{i-i_b}>1$, then also $\Add_i(\ta{t}^\lambda_{b-1})=\emptyset$.
			By induction, $\id_{(i_1, \dots, i_{b-1}, i)}\in \Rl{<b-1}$ and hence $\id_{(i_1, \dots, i_{b-1}, i, i_b)}\in \Rl{<b}$ as $\restr{\theta_i}{\Rl{\leq b-1}}$ is filtered of degree $1$.
			Now $\Rl{<b}=\Rts{<b}$ by assumption and 
			\begin{equation*}
				\Psi_{\ta{t}^\lambda\ta{t}^\lambda}=
				\begin{tikzpicture}[line width = \lw, myscale=0.6, yscale=0.8]
					\node at (1,1.5) (A) {$i_1$};
					\node at (3,1.5) (B){$i_{b-1}$};
					\node at (4,1.5) (C){$i$};
					\node at (5,1.5) (D){$i_b$};
					\node at (2,1.5) {$\dots$};
					\draw (1,0)-- (A) -- (1,3) (3,0) --(B) -- (3,3) (4,0) -- (D.south) (D.north) -- (4,3) (5,0) -- (C.south) (C.north) --(5,3);
				\end{tikzpicture}
			\end{equation*}\nopagebreak
            by \cref{inverse}.
			Thus, $\Psi_{\ta{t}^\lambda\ta{t}^\lambda}\in \Rl{<b}$.
		\item Suppose that $i_b=i+1$.
			By definition of $\ta{t}^\lambda$, $i_b$ is the residue of the last box in the last row of $\lambda$.
			As there is no addable box with residue $i$, the last row of $\lambda$ has more than one box.
			But this means by construction that $i_{b-1}=i$.
			
			Now let $\ta{u}$ be the up-down-tableau such that $\ta{u}_k=\ta{t}^\lambda_k$ for $k<b$, $\ta{u}_{b+1}=\ta{u}_{b-1}$ and $\ta{u}_b=\ta{u}_{b-2}$.
			So the first $b-1$ steps agree with $\ta{t}$, then we remove and add the box that was added in step $b-1$.
			Similarly, we define $\ta{v}_k$ to be the up-down-tableau such that $\ta{v}_k=\ta{t}^\lambda_k$ for $k\leq b$ and $\ta{v}_{b+1}=\ta{v}_{b-1}$.
			In other words $\ta{v}$ agrees with $\ta{t}$ on the first $b$ steps and then removes the box that was added in step $b$.
			Observe that the source residue sequence of $\ta{u}$ and the target residue sequence of $\ta{v}$ are exactly $(i_1, \dots, i_b, i)$ by construction.
			By \cref{idempotentone} we have 
			\begin{align*}
				\Psi_{\ta{u}\ta{v}}&=
				\begin{tikzpicture}[line width = \lw, myscale=0.7]
					\node at (0,0) (A) {$i_1$};
					\node at (2,0) (B) {$i_{b-2}$};
					\node at (3,0) (C) {$i$};
					\node at (4,0) (D) {$i+1$};
					\node at (5,0) (E) {$i$};
					\node at (0,1.5) (F) {$i_1$};
					\node at (2,1.5) (G) {$i_{b-2}$};
					\node at (3,1.5) (H) {$i$};
					\node at (4,1.5) (I) {$i+1$};
					\node at (5,1.5) (J) {$i$};
					\node at (1,0.75) {$\dots$};
					\draw (A) -- (F) (B)--(G) (C)..controls +(0.2,0.8) and +(-0.2,0.8)..(D) (E)--(H) (I)..controls +(0.2,-0.8) and +(-0.2,-0.8)..(J);
				\end{tikzpicture}
				=\eps
				\begin{tikzpicture}[line width = \lw, myscale=0.7]
					\node at (0,0) (A) {$i_1$};
					\node at (2,0) (B) {$i_{b-2}$};
					\node at (3,0) (C) {$i$};
					\node at (4,0) (D) {$i+1$};
					\node at (5,0) (E) {$i$};
					\node at (0,1.5) (F) {$i_1$};
					\node at (2,1.5) (G) {$i_{b-2}$};
					\node at (3,1.5) (H) {$i$};
					\node at (4,1.5) (I) {$i+1$};
					\node at (5,1.5) (J) {$i$};
					\node at (1,0.75) {$\dots$};
					\draw (A) -- (F) (B)--(G) (C)--(H) (D)--(I) (E)--(J);
				\end{tikzpicture}\\
				&=\eps\theta_i(\Psi_{\ta{t}^\lambda\ta{t}^\lambda}).
			\end{align*}
			The number of propagating strands in $\Psi_{\ta{u}\ta{v}}$ is exactly one less than the one in $\Psi_{\ta{t}^\lambda\ta{t}^\lambda}$. Thus, $\theta_i(\Psi_{\ta{t}^\lambda\ta{t}^\lambda})\in \Rl{<b}$.
		\item Suppose that $i_b=i-1$.
			This case is similar to the one before, but you also have to use \cref{inverse} to move a value $i$ to the position $b-1$.
	\end{enumerate}
\end{proof}
\begin{lem}\label{thetaipreservesspanningset}
    If $\Rl{<b'}=\Rts{<b'}$ for all $b'\leq b$, then $\theta_i(\Rl{\leq b})\subseteq \Rl{\leq b+1}$ and $\theta_i(\Rl{< b})\subseteq \Rl{< b+1}$.
\end{lem}
\begin{proof}
	This follows directly from \cref{thetaiaddablebox} and \cref{thetainotaddableboxforgeneraltands}.
\end{proof}
\begin{rem}\label{identitiesinspanningset}
	Let $(i_1, \dots, i_b)$ be an object in $\sG$.
	Suppose that $\Rl{<b'}=\Rts{<b'}$ for all $b'\leq b$.
    Then $\id_{(i_1, \dots i_b)}\in \Rl{\leq b}$ by \cref{thetaipreservesspanningset} as $\id_{(i_1)}\in \Rl{\leq 1}$. 
\end{rem}
Our next goal is to show that $\Rl{\leq b}=\Rts{\leq b}$ for all $b$. 
This together with \cref{identitiesinspanningset} implies that the set in \cref{spanningsetisbasis} forms a spanning set of $\sG$.
From now on fix up-down-tableaux $\ta{t}$ and $\ta{s}$ of shape $\lambda$.
We define $b\coloneqq \abs{\lambda}$ so that $\Psi_{\ta{t}\ta{s}}\in \Rl{\leq b}$.
Let $(i_1^l, \dots, i_m^l)\coloneqq\bm{i}^l_{\ta{t}}$ and $(i_1^r, \dots, i_n^r)\coloneqq\bm{i}^r_{\ta{s}}$.
\begin{lem}\label{closedundercups}
	Let $x_i=\begin{tikzpicture}[line width = \lw, myscale=0.7]
		\node at (0,1) (A){$i_1^r$};
		\node at (4,1) (B){$i_n^r$};
		\node at (1,1) (E){$i_k^r$};
		\node at (3,1) (F){$i_{k+1}^r$};
		\node at (1.5,1) (c){$i$};
		\node at (0.5,0.5){$\dots$};
		\node at (3.5,0.5){$\dots$};
		\draw (A)-- (0,0) (B) -- (4,0)(E) -- (1,0)(F) -- (3,0) (c)..controls +(0.2,-0.8) and +(-0.2,-0.8)..(2.5,1);
		\node[fill=white] at (2.25,1) (d){$i-1$};
	\end{tikzpicture}$.
	If $\Rl{\leq b'}=\Rts{\leq b'}$ for all $b'< b$, then $x_i\Psi_{\ta{t}\ta{s}}\in \Rl{\leq b}$.
\end{lem}
\begin{proof}
    If $\alpha\in\Add_i(\ta{s}_k)$, let $\ta{u}\in\Tud(\lambda)$ with $\restr{\ta{u}}{k}=\restr{\ta{s}}{k}$, $\ta{u}_{k+1}=\ta{u}_k\oplus\alpha$, $\ta{u}_{j+2}=\ta{s}_{j}$ for $k\leq j\leq n$.
    Then, $x_i\Psi_{\ta{t}\ta{s}}=\pm\Psi_{\ta{t}\ta{u}}\in \Rl{\leq b}$.

	Suppose now that $\Add_i(\ta{s}_k)=\emptyset$.
    We may also assume that $\abs{\ta{s}_k}=k$ as removing boxes would correspond to cups commuting with the cup of $x_i$.
    By the proof of \cref{thetainotaddablebox} we see that the object $(i_1, \dots, i_k, a, a-1)$ is either $0$ or we find a subsequence $(a,a\pm1,a)$.
    If the subsequence is $(a,a-1,a)$, applying \cref{idempotenttwo} gives the diagram according to removing with $a-1$ the box with residue $a$ and then adding two boxes.

    On the other hand if the subsequence is $(a,a+1,a)$, we get a valid up-down tableau $\ta{v}$ with $\res^r(\ta{v})=(i_1, \dots, i_k, a,a-1)$ where the last two entries remove the boxes corresponding to the subsequence.
    If $\ta{v}$ can be extended to an up-down tableau of shape $\mu$ by $(i_{k+1}, \dots, i_m)$, then $\abs{\mu}<b$ and $x\Psi_{\ta{s}\ta{t}}=a\Psi_{\ta{t}^\mu\ta{v}}b\in \Rl{<b}$ by construction.
    If it cannot be extended, we either find a subsequence $(a-1,a-1)$ which is $0$ by \cref{twicezero} or we try to add a box of residue $a+2$.
    But this strand can be moved to the left using \cref{inverse}, where we then either get $0$ by \cref{twicezero} or \cref{firstzero} or we find a subsequence $(a+2,a+3,a+2)$.
    In the last case we can apply \cref{idempotentone} and then use the same argument as above and end up eventually with $0$.
\end{proof}
\begin{lem}\label{closedundercrossings}
	Let $z=\begin{tikzpicture}[line width = \lw, myscale=0.7]
		\node[anchor=north] at (0,1) (A){$i_1^r$};
		\node[anchor=north] at (2,1) (B){$i_{k-1}^r$};
		\node[anchor=north] at (3,1) (c){$i_{k}^r$};
		\node[anchor=north] at (4,1) (d){$i_{k+1}^r$};
		\node[anchor=north] at (5,1) (e){$i_{k+2}^r$};
		\node[anchor=north] at (7,1) (f){$i_n^r$};
		\node[anchor=north] at (1,1.5){$\dots$};
		\node[anchor=north] at (6,1.5){$\dots$};
		\draw (A)-- (0,2) (B) -- (2,2) (4,2)--(c.north) (3,2)--(d.north) (5,2)--(e) (7,2)--(f);
	\end{tikzpicture}$.
	If $\Rl{\leq b'}=\Rts{\leq b'}$ for all $b'< b$, then $z\Psi_{\ta{t}\ta{s}}\in \Rl{\leq b}$.
\end{lem}
\begin{proof}
     Suppose first that $\abs{i_k-i_{k+1}}>1$. 
	Then we claim that $\ta{s}s_k$ is an up-down-tableau.
    If in steps $k$ and $k+1$ we only add respectively remove boxes, then this is clear as the boxes neither appear in the same row nor column.
    
    In the other two cases let $i$ be the residue of the removal.
    This means that we removed a box $\alpha$ with residue $i+1$.
    So this actually swaps with all residues which are $\leq i-1$ and $\geq i+3$. 
    
    The only case left to consider is, when the added box $\beta$ has residue $i+2$.
    But note that after the removal of $\alpha$, $\alpha$ is addable again.
    As $\alpha$ has residue $i+1$, no box with residue $i+2$ can be addable.
    And vice versa, if we add $\beta$ after adding $\alpha$, the box $\beta$ lies directly to the right of $\alpha$ in the same row.
    Thus, we cannot remove $\alpha$ afterwards.
    Therefore, this case cannot appear.
    
    It remains to show the statement for $i_k=i_{k+1}\pm 1$.
    If steps $k$ and $k+1$ consist out of adding boxes $\alpha$ and $\beta$, these two boxes appear in the same row respectively column.
    This means that $\beta$ cannot be added to $\ta{s}_{k-1}$.
    By \cref{thetainotaddableboxforgeneraltands} and \cref{thetaipreservesspanningset} we know that $z\Psi_{\ta{t}\ta{s}}\in \Rl{< b}$.
    
    Suppose that we remove a box in step $k$ and add a box in step $k+1$.
    Let further be $l$ the step in which the box was added which was removed in step $k$.
    Without loss of generality we may assume that $l=k-1$.
    Using \cref{inverse} we can move the $i_l$ pass every distant entry and every neighbored entry has to be removed prior step $k$, which results in a cup that does not interact with the crossing $z$, meaning that we can swap these two as well.
    We then either have the subsequence $(i_k+1, i_k, i_k+1)$, in which case applying $z$ gives $0$ by \cref{twicezero}.
    Or we have $(i_k+1, i_k, i_k-1)$, in which case step $\ta{s}s_k$ is an up-down-tableau and $z\Psi_{\ta{t}\ta{s}}=\Psi_{\ta{t}\ta{s}s_k}$.
    
    Suppose that we add a box in step $k$ and remove a box in step $k+1$.
    Then $z\Psi_{\ta{t}\ta{s}}=0$ by \cref{inverse} or \cref{untwist} respectively.
    
    Suppose that we remove in both steps a box.

    Let $l$ denote the index of adding the box for step $k$ and $l'$ the one for $k+1$.
    Without loss of generality we may assume that $l=k-1$ and $l'=k-2$.
    Note that $l'$ has to appear before $l$ as $\ta{s}$ is an up-down-tableau.
    
    Then we either have a subsequence $(a+1,a,a-1,a)$ or $(a-1, a, a-1, a-2)$.
    In the first case, applying $z$ gives $0$ by \cref{twicezero}.
    In the second case we have the following equality (locally) using \cref{inverse} and \cref{tangletwo}:
    \begin{equation*}
        \begin{tikzpicture}[line width = \lw, myscale=0.7]
            \node at (0,0) (A) {$a-1$};
            \node at (1,0) (B) {$a$};
            \node at (2,0) (C) {$a-1$};
            \node at (3,0) (D) {$a-2$};
            \draw (A)..controls +(0.2,-1.2) and +(-0.2,-1.2)..(D) (B)..controls +(0.2,-0.8) and +(-0.2,-0.8)..(C) (A)--(0,1.5) (B)--(1,1.5) (C.north)--(3,1.5) (D.north)--(2,1.5);
        \end{tikzpicture}
        =
        \begin{tikzpicture}[line width = \lw, myscale=0.7]
            \node at (0,0) (A) {$a-1$};
            \node at (1,0) (B) {$a-2$};
            \node at (2,0) (C) {$a$};
            \node at (3,0) (D) {$a-1$};
            \draw(A)..controls +(0.2,-0.8) and +(-0.2,-0.8)..(B) (C)..controls +(0.2,-0.8) and +(-0.2,-0.8)..(D)  (A)--(0,1.5) (1,1.5)--(C.north)(B.north)--(2,1.5)(3, 1.5)--(D);
        \end{tikzpicture}
    \end{equation*}
    Now removing steps $k-1$ and $k$ from $\ta{s}$ gives a valid up-down-tableau $\ta{u}$ of shape $\lambda$.
    Thus, looking at the diagrams we see that $z\Psi_{\ta{t}\ta{s}}$ is obtained from $\Psi_{\ta{t}\ta{u}}$ via left multiplication with a cup and a distant crossing. 
    Now the claim follows from \cref{closedundercups} and the first paragraph about distant crossings.
\end{proof}
\begin{lem}\label{closedundercaps}
	Let $y=\begin{tikzpicture}[line width = \lw, myscale=0.7]
		\node at (0,1) (A){$i_1^r$};
		\node at (2,1) (B){$i_{k-1}^r$};
		\node at (3,1) (c){$i_{k}^r$};
		\node at (4,1) (d){$i_{k+1}^r$};
		\node at (5,1) (e){$i_{k+2}^r$};
		\node at (7,1) (f){$i_n^r$};
		\node at (1,1.5){$\dots$};
		\node at (6,1.5){$\dots$};
		\draw (A)-- (0,2) (B) -- (2,2) (c)..controls +(0.2,0.8) and +(-0.2,0.8)..(d) (5,2)--(e) (7,2)--(f);
	\end{tikzpicture}$.
	If $\Rl{\leq b'}=\Rts{\leq b'}$ for all $b'< b$, then $y\Psi_{\ta{t}\ta{s}}\in \Rl{\leq b}$.
\end{lem}
\begin{proof}
	If in the $k$-th step of $\ta{s}$ a box is removed and in the $k+1$-th one is added, using \cref{snake} we get $y\Psi_{\ta{t}\ta{s}}=\pm\Psi_{\ta{t}\ta{u}}$, where $\ta{u}$ is obtained from $\ta{s}$ via deleting steps $k$ and $k+1$.
	
	If the $k$-th step adds a box and the $k+1$-th removes one, then $y\Psi_{\ta{t}\ta{s}}=0$ by \cref{untwist}.
	
	If in the $k$-th and $k+1$-th steps boxes are removed from $\ta{s}$, let $\ta{s}'$ be the up-down tableaux, that is obtained from $\ta{s}$ by removing all the \enquote{cups} of $\ta{s}$, i.e.\ it is the same as $\ta{s}$ but whenever we would remove a box in $\ta{s}$ or add a box that later would be removed we skip this step. 
	Now $\Psi_{\ta{t}\ta{s}} = c\cdot\Psi_{\ta{t}\ta{s}'}$, where $c$ is a diagram consisting of cups (which might intersect).
	As the $k$-th and $k+1$-th step both remove boxes, we see that $y\cdot\Psi_{\ta{t}\ta{s}}=c'\cdot\Psi_{\ta{t}\ta{s}'}$ by \cref{snake}, where $c'$ also consists only of cups.
	The statement then follows from \cref{closedundercups,closedundercrossings}.

	The remaining case to consider is when two boxes are added.
	Restrict first to the case that $\ta{s}=\ta{t}=\ta{t}^\lambda$.
	Then if removing the $k$-th and $k+1$-th step of $\ta{t}^\lambda$ stays an up-down-tableaux, we define $\ta{u}$ to be obtained from $\ta{s}$ via removing steps $k$ and $k+1$.
	We furthermore define $\ta{v}$ to be obtained from $\ta{t}$ via removing the added box of step $k$ in step $k+1$. 
	Then we have $y\Psi_{\ta{t}\ta{s}}=\Psi_{\ta{v}\ta{s}}\in \Rl{<b}$.
	If $\ta{t}^\lambda$ is not an up-down-tableau after removing steps $k$ and $k+1$, then there exists some point $l$, where there is no addable box of residue $i_l^r$.
	But in this case (using \cref{thetainotaddablebox}) we also end up with $y\Psi_{\ta{t}\ta{s}}\in \Rl{<b}$.
	By assumption $\Rl{<b}$ is a two-sided ideal, so we can extend this result to arbitrary $\ta{t}$ and $\ta{s}$ instead of just $\ta{t}^\lambda$.
\end{proof}
\begin{cor}\label{twosidedideal}
	If $\Rl{\leq b'}=\Rts{\leq b'}$ for all $b'< b$, then $\Rl{\leq b}=\Rts{\leq b}$.
\end{cor}
\begin{proof}
	This is \cref{closedundercups,closedundercrossings,closedundercaps} and its mirror versions.
\end{proof}
\begin{cor}\label{spanningset}
	The set of all $\Psi_{\ta{t}\ta{s}}$ forms a spanning set for $\sG$.
\end{cor}
\begin{proof}
	By \cref{twosidedideal} we know that all $\Rl{\leq b}$ form two-sided ideals and by \cref{identitiesinspanningset} we see that all identities lie in some $\Rl{\leq b}$ for some $b$.
	These two facts together imply that the $\Psi_{\ta{t}\ta{s}}$ span $\sG$.
\end{proof}
\begin{proof}[Proof of \cref{spanningsetisbasis}]
	We know by \cref{spanningset} that it generates. 
	Furthermore, we know by \cref{functorXiexplicit} that the functor $\FunctorKLRtosBr\colon\sG\to\Rep(P)$ is in fact full.
	In particular,
	\begin{align*}
		&\dim\Hom_{\sBr}(n,m)\leq\bigoplus_{\bm{i}\in\bbZ^n}\bigoplus_{\bm{j}\in\bbZ^m}\dim\Hom_{\sG}(\bm{i}, \bm{j})\\&\leq\#\{(\ta{t},\ta{s})\mid \ta{t}\text{, }\ta{s} \text{ up-down-tableaux of the same shape of size $n$ respectively $m$}\}.
	\end{align*}
	But we know that the right-hand side is exactly the dimension of $\Hom_{\sBr}(n,m)$.
\end{proof}
\subsection{Proof of associativity}
In order to show that $\K$ is an associative algebra, we will look at a slightly more general situation. 
We consider multiple circle diagrams stacked over one another, i.e.\ we take a sequence $\nu_1, \dots \nu_k$ such that $\underline{\nu_i}\overline{\nu_{i+1}}$ is orientable for all $i$. 
Drawing all these beneath each other, as for the multiplication procedure, we get a big orientable diagram.
We will prove that any two surgery procedures commute with each other.

Before we prove this, we will first analyze, when a surgery procedure gives $0$.
\begin{lem}\label{surgproczero}
	A surgery procedure gives $0$ if and only if we are in one of the following two situations (up to rotational symmetry):
	
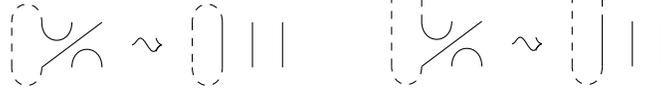
\begin{equation*}	
		\begin{tikzpicture}[scale=0.4]
			\LINE{-2}{-1.5}
			\CUP[dashed]{-1}
			\LINE[dashed]{0}{1.5}
			\CAP[dashed]{1}
			\CUP{1}
			\SETCOORD{0}{-1.5}
			\CAP{1}
			\draw[decorate, decoration={snake}, ->] (1,-0.75)--(2,-0.75);
			\SETSTART{4}{-1.5}
			\LINE{0}{1.5}
			\CAP[dashed]{-1}
			\LINE[dashed]{0}{-1.5}
			\CUP[dashed]{1}
			\SETCOORD{1}{0}
			\LINE{0}{1.5}
			\SETCOORD{1}{0}
			\LINE{0}{-1.5}
		\end{tikzpicture}\qquad\qquad
		\begin{tikzpicture}[scale=0.4]
			\LINE{-2}{-1.5}
			\CUP[dashed]{-1}
			\LINE[dashed]{0}{2.25}
			\SETCOORD{1}{0}
			\LINE[dashed]{0}{-0.75}
			\CUP{1}
			\SETCOORD{0}{-1.5}
			\CAP{1}
			\draw[decorate, decoration={snake}, ->] (1,-0.75)--(2,-0.75);
			\SETSTART{4}{0.75}
			\LINE[dashed]{0}{-0.75}
			\LINE{0}{-1.5}
			\CUP[dashed]{-1}
			\LINE[dashed]{0}{2.25}
			\SETCOORD{2 }{-0.75}
			\LINE{0}{-1.5}
			\SETCOORD{1}{0}
			\LINE{0}{1.5}
		\end{tikzpicture}
	\end{equation*}
	In the second picture the two rays end on the same side of $0$, but they may also end on the bottom.
\end{lem}
\begin{proof}
	One easily checks that these are all the possible cases.
\end{proof}
\begin{thm}\label{surgproccommute}
	Given two potential surgeries $D$ and $D'$, we have $D\circ D'=D'\circ D$.
\end{thm}
\begin{proof}
	Assume we are given two potential surgeries.
	If both surgeries produce in the first step orientable diagrams, the surgery procedures commute.
	So we may assume that one of them produces a non-orientable diagram.
    \Cref{surgproczero} describes how this has to look locally.
	But now note that if the dashed cup does not interact with the other surgery procedure, then the overall result will be zero, independent of the order. 
	So we may assume that the dashed cup does interact with the other surgery. 
	We will make a case distinction on how this cup connects to the other surgery.
	By rotating the bottom surgery, we can reduce to three different main cases as follows:
	\begin{equation*}	
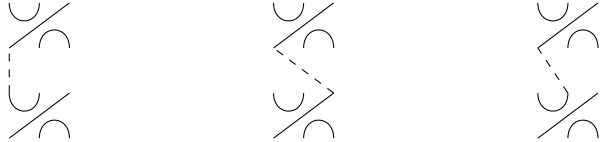
		\begin{array}{m{0.25\textwidth}m{0.25\textwidth}m{0.25\textwidth}}
			\begin{tikzpicture}[scale=0.4]
				\CAP{-1}
				\SETCOORD{-1}{0}
				\LINE{2}{1.5}
				\SETCOORD{-1}{0}
				\CUP{-1}
				\LINE[dashed]{0}{1.5}
				\LINE{2}{1.5}
				\SETCOORD{-2}{0}
				\CUP{1}
				\SETCOORD{0}{-1.5}
				\CAP{1}
			\end{tikzpicture}&
			\begin{tikzpicture}[scale=0.4]
				\CAP{-1}
				\SETCOORD{0}{1.5}
				\CUP{-1}
				\SETCOORD{0}{-1.5}
				\LINE{2}{1.5}
				\LINE[dashed]{-2}{1.5}
				\LINE{2}{1.5}
				\SETCOORD{-2}{0}
				\CUP{1}
				\SETCOORD{0}{-1.5}
				\CAP{1}
			\end{tikzpicture}&
			\begin{tikzpicture}[scale=0.4]
				\CAP{-1}
				\SETCOORD{-1}{0}
				\LINE{2}{1.5}
				\SETCOORD{-2}{0}
				\CUP{1}
				\LINE[dashed]{-1}{1.5}
				\LINE{2}{1.5}
				\SETCOORD{-2}{0}
				\CUP{1}
				\SETCOORD{0}{-1.5}
				\CAP{1}
			\end{tikzpicture}
		\end{array}
	\end{equation*}
	We make a further case distinction on how the other end of the dashed cup connects to the other surgery.
	As these need to be connected in the total picture, we are left with three more cases for each main case (see \cref{assocfirstdistinctionfigure} for this).
	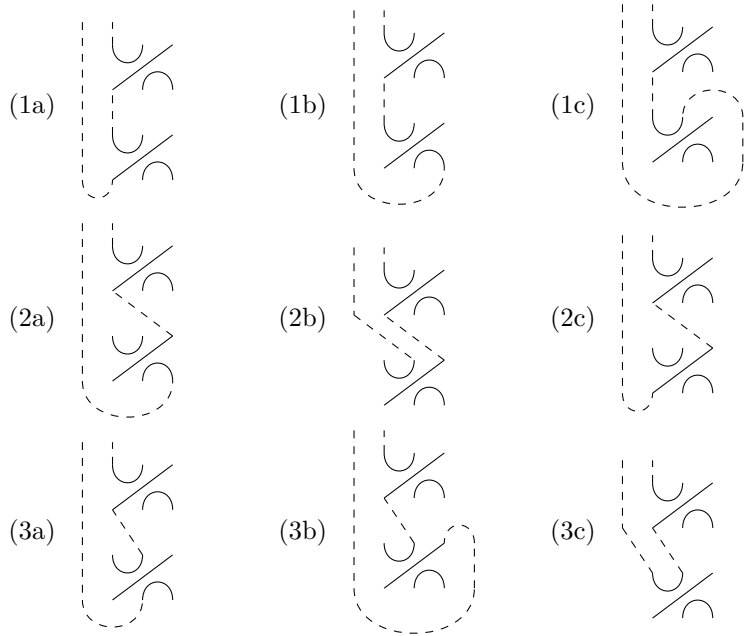
\begin{figure}\centering
		\setcounter{enumi}{0}
		\setcounter{enumii}{0}
		\renewcommand\theenumii{\alph{enumii}}
		\begin{tabularx}{0.8\textwidth}{lXlXll}
			\refstepcounter{enumi}\refstepcounter{enumii}(\theenumi\theenumii)&
			\begin{tikzpicture}[scale=0.4]
				\LINE[dashed]{0}{-5.25}
				\CUP[dashed]{1}
				\SETCOORD{2}{0}
				\CAP{-1}
				\SETCOORD{-1}{0}
				\LINE{2}{1.5}
				\SETCOORD{-1}{0}
				\CUP{-1}
				\LINE[dashed]{0}{1.5}
				\LINE{2}{1.5}
				\SETCOORD{0}{-1.5}
				\CAP{-1}
				\SETCOORD{0}{1.5}
				\CUP{-1}
				\LINE[dashed]{0}{0.75}
			\end{tikzpicture}
			&\refstepcounter{enumii}(\theenumi\theenumii)&
			\begin{tikzpicture}[scale=0.4]
				\LINE[dashed]{0}{-5.25}
				\CUP[dashed]{3}
				\CAP{-1}
				\SETCOORD{-1}{0}
				\LINE{2}{1.5}
				\SETCOORD{-1}{0}
				\CUP{-1}
				\LINE[dashed]{0}{1.5}
				\LINE{2}{1.5}
				\SETCOORD{0}{-1.5}
				\CAP{-1}
				\SETCOORD{0}{1.5}
				\CUP{-1}
				\LINE[dashed]{0}{0.75}
			\end{tikzpicture}
			&\refstepcounter{enumii}(\theenumi\theenumii)&
			\begin{tikzpicture}[scale=0.4]	
				\LINE[dashed]{0}{-5.25}
				\CUP[dashed]{4}
				\LINE[dashed]{0}{1.5}
				\CAP[dashed]{-2}
				\SETCOORD{1}{-1.5}
				\CAP{-1}
				\SETCOORD{-1}{0}
				\LINE{2}{1.5}
				\SETCOORD{-1}{0}
				\CUP{-1}
				\LINE[dashed]{0}{1.5}
				\LINE{2}{1.5}
				\SETCOORD{0}{-1.5}
				\CAP{-1}
				\SETCOORD{0}{1.5}
				\CUP{-1}
				\LINE[dashed]{0}{0.75}
			\end{tikzpicture}\\
			\refstepcounter{enumi}\setcounter{enumii}{0}\refstepcounter{enumii}(\theenumi\theenumii)&
			\begin{tikzpicture}[scale=0.4]
				\LINE[dashed]{0}{-5.25}
				\CUP[dashed]{3}
				\CAP{-1}
				\SETCOORD{0}{1.5}
				\CUP{-1}
				\SETCOORD{0}{-1.5}
				\LINE{2}{1.5}
				\LINE[dashed]{-2}{1.5}
				\LINE{2}{1.5}
				\SETCOORD{0}{-1.5}
				\CAP{-1}
				\SETCOORD{0}{1.5}
				\CUP{-1}
				\LINE[dashed]{0}{0.75}
			\end{tikzpicture}
			&\refstepcounter{enumii}(\theenumi\theenumii)&
			\begin{tikzpicture}[scale=0.4]
				\LINE[dashed]{0}{-2.25}
				\LINE[dashed]{2}{-1.5}
				\SETCOORD{1}{-1.5}
				\CAP{-1}
				\SETCOORD{0}{1.5}
				\CUP{-1}
				\SETCOORD{0}{-1.5}
				\LINE{2}{1.5}
				\LINE[dashed]{-2}{1.5}
				\LINE{2}{1.5}
				\SETCOORD{0}{-1.5}
				\CAP{-1}
				\SETCOORD{0}{1.5}
				\CUP{-1}
				\LINE[dashed]{0}{0.75}
			\end{tikzpicture}
			&\refstepcounter{enumii}(\theenumi\theenumii)&
			\begin{tikzpicture}[scale=0.4]
				\LINE[dashed]{0}{-5.25}
				\CUP[dashed]{1}
				\SETCOORD{2}{0}
				\CAP{-1}
				\SETCOORD{0}{1.5}
				\CUP{-1}
				\SETCOORD{0}{-1.5}
				\LINE{2}{1.5}
				\LINE[dashed]{-2}{1.5}
				\LINE{2}{1.5}
				\SETCOORD{0}{-1.5}
				\CAP{-1}
				\SETCOORD{0}{1.5}
				\CUP{-1}
				\LINE[dashed]{0}{0.75}
			\end{tikzpicture}\\
			\refstepcounter{enumi}\setcounter{enumii}{0}\refstepcounter{enumii}(\theenumi\theenumii)&
			\begin{tikzpicture}[scale=0.4]
				\LINE[dashed]{0}{-5.25}
				\CUP[dashed]{2}
				\SETCOORD{1}{0}
				\CAP{-1}
				\SETCOORD{-1}{0}
				\LINE{2}{1.5}
				\SETCOORD{-2}{0}
				\CUP{1}
				\LINE[dashed]{-1}{1.5}
				\LINE{2}{1.5}
				\SETCOORD{0}{-1.5}
				\CAP{-1}
				\SETCOORD{0}{1.5}
				\CUP{-1}
				\LINE[dashed]{0}{0.75}
			\end{tikzpicture}
			&\refstepcounter{enumii}(\theenumi\theenumii)&
			\begin{tikzpicture}[scale=0.4]
				\LINE[dashed]{0}{-5.25}
				\CUP[dashed]{4}
				\LINE[dashed]{0}{1.5}
				\CAP[dashed]{-1}
				\SETCOORD{0}{-1.5}
				\CAP{-1}
				\SETCOORD{-1}{0}
				\LINE{2}{1.5}
				\SETCOORD{-2}{0}
				\CUP{1}
				\LINE[dashed]{-1}{1.5}
				\LINE{2}{1.5}
				\SETCOORD{0}{-1.5}
				\CAP{-1}
				\SETCOORD{0}{1.5}
				\CUP{-1}
				\LINE[dashed]{0}{0.75}
			\end{tikzpicture}
			&\refstepcounter{enumii}(\theenumi\theenumii)&
			\begin{tikzpicture}[scale=0.4]
				\LINE[dashed]{0}{-2.25}
				\LINE[dashed]{1}{-1.5}
				\SETCOORD{2}{-1.5}
				\CAP{-1}
				\SETCOORD{-1}{0}
				\LINE{2}{1.5}
				\SETCOORD{-2}{0}
				\CUP{1}
				\LINE[dashed]{-1}{1.5}
				\LINE{2}{1.5}
				\SETCOORD{0}{-1.5}
				\CAP{-1}
				\SETCOORD{0}{1.5}
				\CUP{-1}
				\LINE[dashed]{0}{0.75}
			\end{tikzpicture}
		\end{tabularx}
		\caption{First case distinction in the proof of associativity}\label{assocfirstdistinctionfigure}
	\end{figure}
	Observe that in each of these cases the rays at the top may also end on the bottom or may be joined to form a cap. 
	If they are not joined to a cap, then the two endpoints lie on the same side of $0$.
	\begin{enumerate}[leftmargin=*, widest=(1--3a)]
		\item[(1--3a)] In these three cases note that after performing both surgeries the result is not orientable as we have either a circle or a non-propagating line with two endpoints at the same side of $0$.
		\item[(1b)]\label{1b} By assumption the dashed lines have to be connected, so it has to look like the following:
		
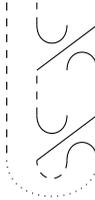
\begin{equation*}
			\begin{tikzpicture}[scale=0.4]
				\LINE[dashed]{0}{-5.25}
				\CUP[dotted]{3}
				\CAP{-1}
				\CUP[dashed]{-1}
				\LINE{2}{1.5}
				\SETCOORD{-1}{0}
				\CUP{-1}
				\LINE[dashed]{0}{1.5}
				\LINE{2}{1.5}
				\SETCOORD{0}{-1.5}
				\CAP{-1}
				\SETCOORD{0}{1.5}
				\CUP{-1}
				\LINE[dashed]{0}{0.75}
			\end{tikzpicture}
		\end{equation*}
		But in this case, the bottom surgery procedure only \enquote{straightens} the cup-cap at the bottom and this commutes with the other surgery procedure.
		\item[(1--2c)] These diagrams are not orientable as the components inside the dashed lines will either give circle or non-propagating lines.
		As these non-prop\-a\-ga\-ting lines would have both endpoints between the dashed ones, they would lie on the same side of $0$ by assumption on the dashed lines.
		\item[(2b)] Observe that by our assumption the two dashed lines have to be connected.
			For this we have two possibilities.
			
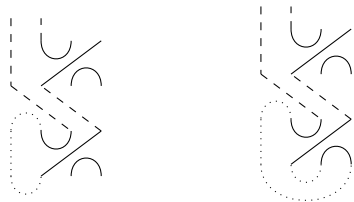
\begin{equation*}
				\begin{tikzpicture}[scale=0.4]
					\LINE[dashed]{0}{-2.25}
					\LINE[dashed]{2}{-1.5}
					\SETCOORD{1}{-1.5}
					\CAP{-1}
					\SETCOORD{0}{1.5}
					\CUP{-1}
					\CAP[dotted]{-1}
					\LINE[dotted]{0}{-1.5}
					\CUP[dotted]{1}
					\LINE{2}{1.5}
					\LINE[dashed]{-2}{1.5}
					\LINE{2}{1.5}
					\SETCOORD{0}{-1.5}
					\CAP{-1}
					\SETCOORD{0}{1.5}
					\CUP{-1}
					\LINE[dashed]{0}{0.75}
				\end{tikzpicture}\qquad\qquad\qquad
				\begin{tikzpicture}[scale=0.4]
					\LINE[dashed]{0}{-2.25}
					\LINE[dashed]{2}{-1.5}
					\CUP{-1}
					\CAP[dotted]{-1}
					\LINE[dotted]{0}{-1.5}
					\CUP[dotted]{3}
					\CAP{-1}
					\CUP[dotted]{-1}
					\LINE{2}{1.5}
					\LINE[dashed]{-2}{1.5}
					\LINE{2}{1.5}
					\SETCOORD{0}{-1.5}
					\CAP{-1}
					\SETCOORD{0}{1.5}
					\CUP{-1}
					\LINE[dashed]{0}{0.75}
				\end{tikzpicture}
			\end{equation*}
			In the left diagram, the second surgery procedure produces $0$ as well by \cref{surgproczero}, so they commute.
			For the second picture note that the second surgery procedure \enquote{straightens} the cup-cap at the bottom, which commutes with the other surgery procedure.
		\item[(3b)] Look at the left endpoint of the bottom cap. 
					It cannot be a ray or form a circle as then the diagram would not be orientable, so we are in the following situation.
			
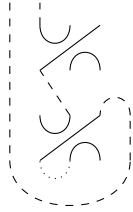
\begin{equation*}
				\begin{tikzpicture}[scale=0.4]
					\LINE[dashed]{0}{-5.25}
					\CUP[dashed]{4}
					\LINE[dashed]{0}{1.5}
					\CAP[dashed]{-1}
					\SETCOORD{0}{-1.5}
					\CAP{-1}
					\CUP[dotted]{-1}
					\LINE{2}{1.5}
					\SETCOORD{-2}{0}
					\CUP{1}
					\LINE[dashed]{-1}{1.5}
					\LINE{2}{1.5}
					\SETCOORD{0}{-1.5}
					\CAP{-1}
					\SETCOORD{0}{1.5}
					\CUP{-1}
					\LINE[dashed]{0}{0.75}
				\end{tikzpicture}
			\end{equation*}
			Similarly, to before, the bottom surgery only \enquote{straightens} the cup-cap at the bottom, and thus commutes with the other surgery.
		\item[(3c)] This is the hardest case of all. 
			In \cref{assocseconddistinctionfigure} we make another case distinction depending on how the bottom right vertex is connected to the rest.
			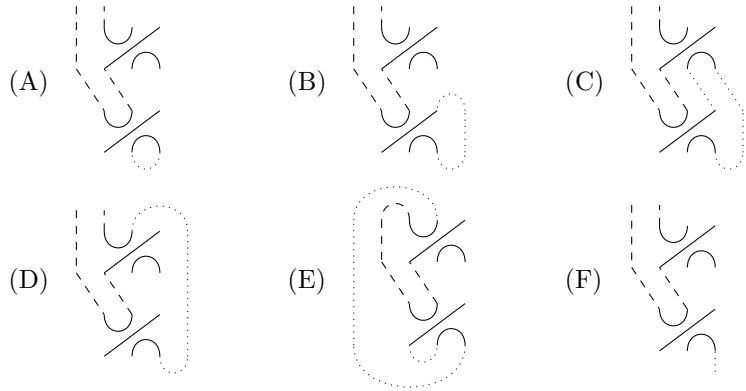
\begin{figure}
				\centering
				\setcounter{enumi}{0}
				\renewcommand\theenumi{\Alph{enumi}}
				\begin{tabularx}{0.8\textwidth}{lXlXll}
					\refstepcounter{enumi}(\theenumi)&
					\begin{tikzpicture}[scale=0.37]
						\LINE[dashed]{0}{-2.25}
						\LINE[dashed]{1}{-1.5}
						\SETCOORD{1}{-1.5}
						\CAP{1}
						\CUP[dotted]{-1}
						\SETCOORD{-1}{0}
						\LINE{2}{1.5}
						\SETCOORD{-2}{0}
						\CUP{1}
						\LINE[dashed]{-1}{1.5}
						\LINE{2}{1.5}
						\SETCOORD{0}{-1.5}
						\CAP{-1}
						\SETCOORD{0}{1.5}
						\CUP{-1}
						\LINE[dashed]{0}{0.75}
					\end{tikzpicture}
					&\refstepcounter{enumi}(\theenumi)&
					\begin{tikzpicture}[scale=0.37]
						\LINE[dashed]{0}{-2.25}
						\LINE[dashed]{1}{-1.5}
						\SETCOORD{1}{-1.5}
						\CAP{1}
						\CUP[dotted]{1}
						\LINE[dotted]{0}{1.5}
						\CAP[dotted]{-1}
						\SETCOORD{-2}{-1.5}
						\LINE{2}{1.5}
						\SETCOORD{-2}{0}
						\CUP{1}
						\LINE[dashed]{-1}{1.5}
						\LINE{2}{1.5}
						\SETCOORD{0}{-1.5}
						\CAP{-1}
						\SETCOORD{0}{1.5}
						\CUP{-1}
						\LINE[dashed]{0}{0.75}
					\end{tikzpicture}
					&\refstepcounter{enumi}(\theenumi)&
					\begin{tikzpicture}[scale=0.37]
						\LINE[dashed]{0}{-2.25}
						\LINE[dashed]{1}{-1.5}
						\SETCOORD{1}{-1.5}
						\CAP{1}
						\CUP[dotted]{1}
						\LINE[dotted]{0}{1.5}
						\LINE[dotted]{-1}{1.5}
						\SETCOORD{-2}{-3}
						\LINE{2}{1.5}
						\LINE[dotted]{-1}{1.5}
						\SETCOORD{-1}{-1.5}
						\CUP{1}
						\LINE[dashed]{-1}{1.5}
						\LINE{2}{1.5}
						\SETCOORD{0}{-1.5}
						\CAP{-1}
						\SETCOORD{0}{1.5}
						\CUP{-1}
						\LINE[dashed]{0}{0.75}
					\end{tikzpicture}\\
					\refstepcounter{enumi}(\theenumi)&
					\begin{tikzpicture}[scale=0.37]
						\LINE[dashed]{0}{-2.25}
						\LINE[dashed]{1}{-1.5}
						\SETCOORD{1}{-1.5}
						\CAP{1}
						\CUP[dotted]{1}
						\LINE[dotted]{0}{4.5}
						\CAP[dotted]{-2}
						\SETCOORD{-1}{-4.5}
						\LINE{2}{1.5}
						\SETCOORD{-2}{0}
						\CUP{1}
						\LINE[dashed]{-1}{1.5}
						\LINE{2}{1.5}
						\SETCOORD{0}{-1.5}
						\CAP{-1}
						\SETCOORD{0}{1.5}
						\CUP{-1}
						\LINE[dashed]{0}{0.75}
					\end{tikzpicture}
					&\refstepcounter{enumi}(\theenumi)&
					\begin{tikzpicture}[scale=0.37]
						\CAP[dashed]{-1}
						\LINE[dashed]{0}{-1.5}
						\LINE[dashed]{1}{-1.5}
						\SETCOORD{0}{-1.5}
						\CUP[dotted]{1}
						\CAP{1}
						\CUP[dotted]{-4}
						\LINE[dotted]{0}{4.5}
						\CAP[dotted]{3}
						\SETCOORD{-1}{-4.5}
						\LINE{2}{1.5}
						\SETCOORD{-2}{0}
						\CUP{1}
						\LINE[dashed]{-1}{1.5}
						\LINE{2}{1.5}
						\SETCOORD{0}{-1.5}
						\CAP{-1}
						\SETCOORD{0}{1.5}
						\CUP{-1}
					\end{tikzpicture}
					&\refstepcounter{enumi}(\theenumi)&
					\begin{tikzpicture}[scale=0.37]
						\LINE[dashed]{0}{-2.25}
						\LINE[dashed]{1}{-1.5}
						\SETCOORD{2}{-2.25}
						\LINE[dotted]{0}{0.75}
						\CAP{-1}
						\SETCOORD{-1}{0}
						\LINE{2}{1.5}
						\SETCOORD{-2}{0}
						\CUP{1}
						\LINE[dashed]{-1}{1.5}
						\LINE{2}{1.5}
						\SETCOORD{0}{-1.5}
						\CAP{-1}
						\SETCOORD{0}{1.5}
						\CUP{-1}
						\LINE[dashed]{0}{0.75}
					\end{tikzpicture}
				\end{tabularx}
				\caption{Subdivision of the case (3c) in the proof of associativity}\label{assocseconddistinctionfigure}
			\end{figure}
			\begin{enumerate}[label=(\Alph*), leftmargin=*]
				\item This diagram is not orientable.
				\item The bottom surgery procedure gives $0$ as well by \cref{surgproczero}.
				\item Again both surgery procedures give $0$, so they commute.
				\item We distinguish two cases
				
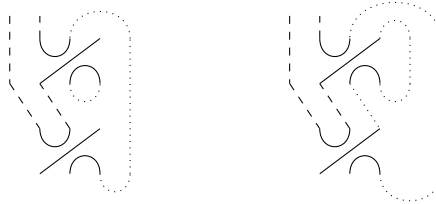
\begin{equation*}
					\begin{tikzpicture}[scale=0.4]
						\LINE[dashed]{0}{-2.25}
						\LINE[dashed]{1}{-1.5}
						\SETCOORD{1}{-1.5}
						\CAP{1}
						\CUP[dotted]{1}
						\LINE[dotted]{0}{4.5}
						\CAP[dotted]{-2}
						\SETCOORD{-1}{-4.5}
						\LINE{2}{1.5}
						\SETCOORD{-2}{0}
						\CUP{1}
						\LINE[dashed]{-1}{1.5}
						\LINE{2}{1.5}
						\SETCOORD{0}{-1.5}
						\CAP{-1}
						\CUP[dotted]{1}
						\SETCOORD{-1}{1.5}
						\CUP{-1}
						\LINE[dashed]{0}{0.75}
					\end{tikzpicture}
					\qquad\qquad\qquad
					\begin{tikzpicture}[scale=0.4]
						\LINE[dashed]{0}{-2.25}
						\LINE[dashed]{1}{-1.5}
						\SETCOORD{1}{-1.5}
						\CAP{1}
						\CUP[dotted]{2}
						\LINE[dotted]{0}{4.5}
						\CAP[dotted]{-3}
						\SETCOORD{-1}{-4.5}
						\LINE{2}{1.5}
						\SETCOORD{-2}{0}
						\CUP{1}
						\LINE[dashed]{-1}{1.5}
						\LINE{2}{1.5}
						\CAP[dotted]{1}
						\LINE[dotted]{0}{-1.5}
						\CUP[dotted]{-1}
						\CAP{-1}
						\LINE[dotted]{1}{-1.5}
						\SETCOORD{-1}{3}
						\CUP{-1}
						\LINE[dashed]{0}{0.75}
					\end{tikzpicture}
				\end{equation*}
				The left picture is not orientable, whereas the right one gives $0$ after applying both.
				\item The bottom surgery \enquote{straightens} only the cup-cap, so it commutes with the other one.
				\item We have two different possibilities.
				\begin{align*}
					&
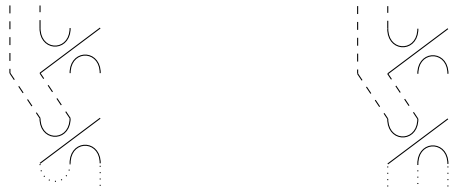
\begin{tikzpicture}[scale=0.4]
						\LINE[dashed]{0}{-2.25}
						\LINE[dashed]{1}{-1.5}
						\SETCOORD{2}{-2.25}
						\LINE[dotted]{0}{0.75}
						\CAP{-1}
						\CUP[dotted]{-1}
						\LINE{2}{1.5}
						\SETCOORD{-2}{0}
						\CUP{1}
						\LINE[dashed]{-1}{1.5}
						\LINE{2}{1.5}
						\SETCOORD{0}{-1.5}
						\CAP{-1}
						\SETCOORD{0}{1.5}
						\CUP{-1}
						\LINE[dashed]{0}{0.75}
					\end{tikzpicture}&
					\begin{tikzpicture}[scale=0.4]
						\LINE[dashed]{0}{-2.25}
						\LINE[dashed]{1}{-1.5}
						\SETCOORD{2}{-2.25}
						\LINE[dotted]{0}{0.75}
						\CAP{-1}
						\LINE[dotted]{0}{-0.75}
						\SETCOORD{-1}{0}
						\LINE[dotted]{0}{0.75}
						\LINE{2}{1.5}
						\SETCOORD{-2}{0}
						\CUP{1}
						\LINE[dashed]{-1}{1.5}
						\LINE{2}{1.5}
						\SETCOORD{0}{-1.5}
						\CAP{-1}
						\SETCOORD{0}{1.5}
						\CUP{-1}
						\LINE[dashed]{0}{0.75}
					\end{tikzpicture}
				\end{align*}
				In the left diagram, the bottom surgery only \enquote{straightens} the cup-cap, and thus commutes with the other one.
				For the right diagram, we distinguish how the right endpoint of the upper cap is connected. 
				All possibilities for this are listed in \cref{assocthirddistinctionfigure}.
				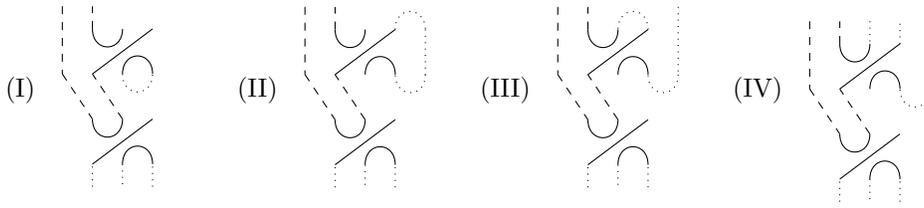
\begin{figure}\centering
					\setcounter{enumi}{0}
					\renewcommand\theenumi{\Roman{enumi}}
					\begin{tabularx}{\textwidth}{lXlXlXll}
						\refstepcounter{enumi}(\theenumi)&
						\begin{tikzpicture}[scale=0.4]
						\LINE[dashed]{0}{-2.25}
						\LINE[dashed]{1}{-1.5}
						\SETCOORD{2}{-2.25}
						\LINE[dotted]{0}{0.75}
						\CAP{-1}
						\LINE[dotted]{0}{-0.75}
						\SETCOORD{-1}{0}
						\LINE[dotted]{0}{0.75}
						\LINE{2}{1.5}
						\SETCOORD{-2}{0}
						\CUP{1}
						\LINE[dashed]{-1}{1.5}
						\LINE{2}{1.5}
						\SETCOORD{0}{-1.5}
						\CAP{-1}
						\CUP[dotted]{1}
						\SETCOORD{-1}{1.5}
						\CUP{-1}
						\LINE[dashed]{0}{0.75}
					\end{tikzpicture}
					&\refstepcounter{enumi}(\theenumi)&
					\begin{tikzpicture}[scale=0.4]
						\LINE[dashed]{0}{-2.25}
						\LINE[dashed]{1}{-1.5}
						\SETCOORD{2}{-2.25}
						\LINE[dotted]{0}{0.75}
						\CAP{-1}
						\LINE[dotted]{0}{-0.75}
						\SETCOORD{-1}{0}
						\LINE[dotted]{0}{0.75}
						\LINE{2}{1.5}
						\SETCOORD{-2}{0}
						\CUP{1}
						\LINE[dashed]{-1}{1.5}
						\LINE{2}{1.5}
						\CAP[dotted]{1}
						\LINE[dotted]{0}{-1.5}
						\CUP[dotted]{-1}
						\CAP{-1}
						\SETCOORD{0}{1.5}
						\CUP{-1}
						\LINE[dashed]{0}{0.75}
					\end{tikzpicture}
					&\refstepcounter{enumi}(\theenumi)&
					\begin{tikzpicture}[scale=0.4]
						\LINE[dashed]{0}{-2.25}
						\LINE[dashed]{1}{-1.5}
						\SETCOORD{2}{-2.25}
						\LINE[dotted]{0}{0.75}
						\CAP{-1}
						\LINE[dotted]{0}{-0.75}
						\SETCOORD{-1}{0}
						\LINE[dotted]{0}{0.75}
						\LINE{2}{1.5}
						\SETCOORD{-2}{0}
						\CUP{1}
						\LINE[dashed]{-1}{1.5}
						\LINE{2}{1.5}
						\CAP[dotted]{-1}
						\SETCOORD{2}{0.75}
						\LINE[dotted]{0}{-2.25}
						\CUP[dotted]{-1}
						\CAP{-1}
						\SETCOORD{0}{1.5}
						\CUP{-1}
						\LINE[dashed]{0}{0.75}
					\end{tikzpicture}
					&\refstepcounter{enumi}(\theenumi)&
					\begin{tikzpicture}[scale=0.4]
						\LINE[dashed]{0}{-2.25}
						\LINE[dashed]{1}{-1.5}
						\SETCOORD{2}{-2.25}
						\LINE[dotted]{0}{0.75}
						\CAP{-1}
						\LINE[dotted]{0}{-0.75}
						\SETCOORD{-1}{0}
						\LINE[dotted]{0}{0.75}
						\LINE{2}{1.5}
						\SETCOORD{-2}{0}
						\CUP{1}
						\LINE[dashed]{-1}{1.5}
						\LINE{2}{1.5}
						\LINE[dotted]{0}{0.75}
						\SETCOORD{-1}{0}
						\LINE[dotted]{0}{-0.75}
						\SETCOORD{2}{0.75}
						\LINE[dotted]{0}{-2.25}
						\CUP[dotted]{-1}
						\CAP{-1}
						\SETCOORD{0}{1.5}
						\CUP{-1}
						\LINE[dashed]{0}{0.75}
						\draw[white] (coord)--+(0,1);
					\end{tikzpicture}
				\end{tabularx}
				\caption{Subdivision of case (f) in the proof of associativity}\label{assocthirddistinctionfigure}
			\end{figure}
			\begin{enumerate}[label=(\Roman*)]
				\item The diagram is not orientable.
				\item The upper surgery produces two circles, where one does not interact with the other one, so the overall product is $0$ and hence they commute.
				\item The diagram is not orientable by assumption on the dashed lines.
				\item In this case we claim that not all the following three pictures are orientable.
				\begin{equation*}
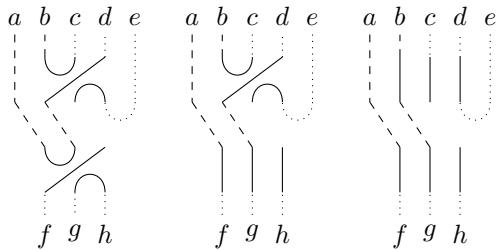
					\begin{array}{ccc}
						\begin{tikzpicture}[scale=0.4]
							\LINE[dashed]{0}{-2.25}
							\LINE[dashed]{1}{-1.5}
							\SETCOORD{2}{-2.25}
							\LINE[dotted]{0}{0.75}
							\CAP{-1}
							\LINE[dotted]{0}{-0.75}
							\SETCOORD{-1}{0}
							\LINE[dotted]{0}{0.75}
							\LINE{2}{1.5}
							\SETCOORD{-2}{0}
							\CUP{1}
							\LINE[dashed]{-1}{1.5}
							\LINE{2}{1.5}
							\LINE[dotted]{0}{0.75}
							\SETCOORD{-1}{0}
							\LINE[dotted]{0}{-0.75}
							\SETCOORD{2}{0.75}
							\LINE[dotted]{0}{-2.25}
							\CUP[dotted]{-1}
							\CAP{-1}
							\SETCOORD{0}{1.5}
							\CUP{-1}
							\LINE[dashed]{0}{0.75}
							\draw[white] (coord)--+(0,1);
							\node[anchor=south] at (0,0) {$a$};
							\node[anchor=south] at (1,0) {$b$};
							\node[anchor=south] at (2,0) {$c$};
							\node[anchor=south] at (3,0) {$d$};
							\node[anchor=south] at (4,0) {$e$};
							\node[anchor=north] at (1,-6) {$f$};
							\node[anchor=north] at (2,-6) {$g$};
							\node[anchor=north] at (3,-6) {$h$};
						\end{tikzpicture}&
						\begin{tikzpicture}[scale=0.4]
							\LINE[dashed]{0}{-2.25}
							\LINE[dashed]{1}{-1.5}
							\SETCOORD{2}{-2.25}
							\LINE[dotted]{0}{0.75}
							\LINE{0}{1.5}
							\SETCOORD{-1}{-2.25}
							\LINE[dotted]{0}{0.75}
							\LINE{0}{1.5}
							\SETCOORD{-1}{-2.25}
							\LINE[dotted]{0}{0.75}
							\LINE{0}{1.5}
							\SETCOORD{1}{0}
							\LINE[dashed]{-1}{1.5}
							\LINE{2}{1.5}
							\LINE[dotted]{0}{0.75}
							\SETCOORD{-1}{0}
							\LINE[dotted]{0}{-0.75}
							\SETCOORD{2}{0.75}
							\LINE[dotted]{0}{-2.25}
							\CUP[dotted]{-1}
							\CAP{-1}
							\SETCOORD{0}{1.5}
							\CUP{-1}
							\LINE[dashed]{0}{0.75}
							\draw[white] (coord)--+(0,1);
							\node[anchor=south] at (0,0) {$a$};
							\node[anchor=south] at (1,0) {$b$};
							\node[anchor=south] at (2,0) {$c$};
							\node[anchor=south] at (3,0) {$d$};
							\node[anchor=south] at (4,0) {$e$};
							\node[anchor=north] at (1,-6) {$f$};
							\node[anchor=north] at (2,-6) {$g$};
							\node[anchor=north] at (3,-6) {$h$};
						\end{tikzpicture}&
						\begin{tikzpicture}[scale=0.4]
							\LINE[dashed]{0}{-2.25}
							\LINE[dashed]{1}{-1.5}
							\SETCOORD{2}{-2.25}
							\LINE[dotted]{0}{0.75}
							\LINE{0}{1.5}
							\SETCOORD{-1}{-2.25}
							\LINE[dotted]{0}{0.75}
							\LINE{0}{1.5}
							\SETCOORD{-1}{-2.25}
							\LINE[dotted]{0}{0.75}
							\LINE{0}{1.5}
							\SETCOORD{1}{0}
							\LINE[dashed]{-1}{1.5}
							\SETCOORD{2}{1.5}
							\LINE[dotted]{0}{0.75}
							\SETCOORD{-1}{0}
							\LINE[dotted]{0}{-0.75}
							\SETCOORD{2}{0.75}
							\LINE[dotted]{0}{-2.25}
							\CUP[dotted]{-1}
							\LINE{0}{1.5}
							\SETCOORD{-1}{-1.5}
							\LINE{0}{1.5}
							\SETCOORD{-1}{-1.5}
							\LINE{0}{1.5}
							\LINE[dashed]{0}{0.75}
							\draw[white] (coord)--+(0,1);
							\node[anchor=south] at (0,0) {$a$};
							\node[anchor=south] at (1,0) {$b$};
							\node[anchor=south] at (2,0) {$c$};
							\node[anchor=south] at (3,0) {$d$};
							\node[anchor=south] at (4,0) {$e$};
							\node[anchor=north] at (1,-6) {$f$};
							\node[anchor=north] at (2,-6) {$g$};
							\node[anchor=north] at (3,-6) {$h$};
						\end{tikzpicture}
					\end{array}
				\end{equation*}
				If we can show this claim, this means that applying the second surgery procedure before the first will also give $0$, thus they commute.
				
				First note that not all of $\{b,c,d,e\}$ can end on the same side, as then one of the pictures would not be orientable (the non-propagating lines would not all wrap around the same spot).
				Similarly, not all of $\{a,f,g,h\}$ can end on the same side.				
				Therefore, $e$ and $h$ end on the same side.

				Now look at the middle picture.
				The two loose ends in the middle have to be now connected as otherwise $e$ and $h$ would contribute to two non-propagating lines next to one another, which is not orientable.

				So $\{e,g,h\}$ (resp.\ $\{d,e,h\}$) cannot end on the same side for the same reasons as before.
				But this means that $\{a,b,c,f\}$ end on the same side, which produces two non-propagating lines in the middle picture that are not nested.
				Thus, there always exists a non-orientable picture and the overall surgery procedure result is $0$.\qedhere
			\end{enumerate}
		\end{enumerate}
	\end{enumerate}
\end{proof}
\subsection{The functor \texorpdfstring{$\Phi$}{Phi} is well-defined}
\begin{lem}\label{phirespectsrest}
	The functor $\Phi$ respects the relations \cref{firstzero,twicezero,inverse,untwist,snake,idempotentone,idempotenttwo}.
\end{lem}
\begin{proof}
	\begin{itemize}[leftmargin=*]
		\item By definition of $\hat{\theta}_i$, we see that $\hat{\theta}_iP(\overline{\iota}) = G_ie_\iota$ has a basis given by all orientable generalized circle diagrams $a t^i\overline{\iota}$.
		But by assumption $\overline{\iota}$ has no caps, so this is only orientable if $i=0$, thus \eqref{firstzero} holds.
		\item We easily observe that $\hat{\theta}_i\circ\hat{\theta}_i=0$ as $\hat{G}_{t^it^i}=0$ because it contains a circle and so no generalized circle diagram of this form is orientable, hence \eqref{twicezero} also holds.
		\item For \eqref{inverse} observe that if $i$ and $j$ are distant $\Psi_{i,j}$ is by definition an isomorphism and its inverse is given by $\Psi_{j,i}$.
		\item For \cref{snake,idempotentone,idempotenttwo} observe that $\hat{\theta}_a\eta_a$, $\eta_a\hat{\theta}_{a+1}$, $\hat{\theta}_a\epsilon_a$ and $\epsilon_a\hat{\theta}_{a-1}$ all look as follows (up to vertical mirror image).
		
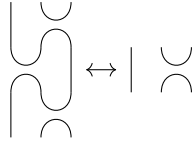
\begin{equation*}
			\begin{tikzpicture}[scale=0.4]
				\CUP{-1}
				\SETCOORD{-1}{0}
				\LINE{0}{-1.5}
				\CUP{1}
				\CAP{1}
				\LINE{0}{-1.5}
				\CUP{-1}
				\CAP{-1}
				\LINE{0}{-1.5}
				\SETCOORD{1}{0}
				\CAP{1}
				\draw[<->] (0.5,-2.25)--(1.5,-2.25);
				\SETCOORD{2}{1.5}
				\LINE{0}{1.5}
				\SETCOORD{1}{0}
				\CUP{1}
				\SETCOORD{0}{-1.5}
				\CAP{-1}
			\end{tikzpicture}
		\end{equation*}
	Thus, these morphisms do not change orientability, and thus are isomorphisms.
	From this description it is then easy to check that \cref{snake,idempotentone,idempotenttwo} hold.
	\item Lastly we want to show \eqref{untwist}.
	Note that the left-hand side is locally given by
	
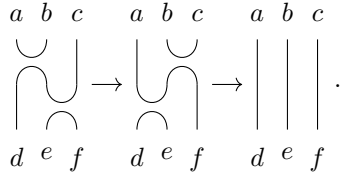
\begin{equation*}
		\begin{tikzpicture}[scale=0.4]
			\node[anchor=south] at (0,0.3) {$a$};
			\node[anchor=south] at (1,0.3) {$b$};
			\node[anchor=south] at (2,0.3) {$c$};
			\node[anchor=north] at (0,-3.3) {$d$};
			\node[anchor=north] at (1,-3.3) {$e$};
			\node[anchor=north] at (2,-3.3) {$f$};
			
			\node[anchor=south] at (4,0.3) {$a$};
			\node[anchor=south] at (5,0.3) {$b$};
			\node[anchor=south] at (6,0.3) {$c$};
			\node[anchor=north] at (4,-3.3) {$d$};
			\node[anchor=north] at (5,-3.3) {$e$};
			\node[anchor=north] at (6,-3.3) {$f$};
			
			\node[anchor=south] at (8,0.3) {$a$};
			\node[anchor=south] at (9,0.3) {$b$};
			\node[anchor=south] at (10,0.3) {$c$};
			\node[anchor=north] at (8,-3.3) {$d$};
			\node[anchor=north] at (9,-3.3) {$e$};
			\node[anchor=north] at (10,-3.3) {$f$};
			\CUP{1}
			\SETCOORD{1}{0}
			\LINE{0}{-1.5}
			\CUP{-1}
			\CAP{-1}
			\LINE{0}{-1.5}
			\SETCOORD{1}{0}
			\CAP{1}
			\draw[->] (2.5,-1.5)--(3.5,-1.5);
			\SETCOORD{2}{3}
			\LINE{0}{-1.5}
			\CUP{1}
			\CAP{1}
			\LINE{0}{-1.5}
			\SETCOORD{-1}{0}
			\CAP{-1}
			\SETCOORD{1}{3}
			\CUP{1}
			\draw[->] (6.5,-1.5)--(7.5,-1.5);
			\SETCOORD{2}{0}
			\LINE{0}{-3}
			\SETCOORD{1}{0}
			\LINE{0}{3}
			\SETCOORD{1}{0}
			\LINE{0}{-3}
		\end{tikzpicture}.
	\end{equation*}
	The only possibility such that this composition is not $0$ is if we can find a diagram such that all of these local changes preserve its orientability.
	
	\noindent For this note that none of these six endpoints can be connected to another one, as there exists a picture that creates a circle (also observe that e.g.\ $a$ cannot be connected to $e$ as there is nothing left for $d$).
	So we may assume that there are only rays connected to every endpoint.
	If the first two pictures are orientable, then neither all of $a$, $b$ and $c$ (resp. $d$, $e$ and $f$) end on the top nor all three on the bottom as then one of these pictures has a non-propagating line, where both ends lie on the same side of $0$.
	Orientability of the first and third picture implies that not all of $a$, $b$ and $d$ (resp. $c$, $e$ and $f$) can end on the same half.
	Similarly, considering the second and third picture not all of $a$, $d$ and $e$ (resp. $b$, $c$ and $f$) can end on the same half.
	This means that if all three pictures are orientable, no three consecutive (considering the letters as lying on a circle) rays can end on the same half. 
	But we have six rays, and thus necessarily three consecutive ones have to end on the same half.
	This is a contradiction, and thus \eqref{untwist} holds.\qedhere
	\end{itemize}
\end{proof}
\begin{lem}\label{phirespectstangle}
	The functor $\Phi$ respects also \eqref{tangleone}.
\end{lem}
\begin{proof}
	First assume that $a\neq b-1, b+2$.
	In this case the left-hand side of \eqref{tangleone} in terms of geometric bimodules looks like
	
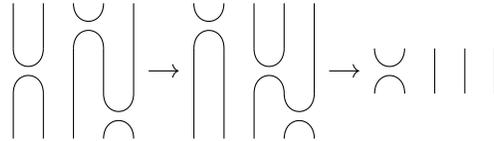
\begin{equation*}
		\begin{tikzpicture}[scale=0.4]
			\LINE{0}{-1.5}
			\CUP{1}
			\LINE{0}{1.5}
			\SETCOORD{1}{0}
			\CUP{1}
			\SETCOORD{1}{0}
			\LINE{0}{-3}
			\CUP{-1}
			\LINE{0}{1.5}
			\CAP{-1}
			\LINE{0}{-3}
			\SETCOORD{-2}{0}
			\LINE{0}{1.5}
			\CAP{1}
			\LINE{0}{-1.5}
			\SETCOORD{2}{0}
			\CAP{1}
			\draw[->] (4.5,-2.25)--(5.5,-2.25);
			\SETCOORD{2}{4.5}
			\CUP{1}
			\SETCOORD{1}{0}
			\LINE{0}{-1.5}
			\CUP{1}
			\LINE{0}{1.5}
			\SETCOORD{1}{0}
			\LINE{0}{-3}
			\CUP{-1}
			\CAP{-1}
			\LINE{0}{-1.5}
			\SETCOORD{-2}{0}
			\LINE{0}{3}
			\CAP{1}
			\LINE{0}{-3}
			\SETCOORD{2}{0}
			\CAP{1}
			\draw[->] (10.5,-2.25)--(11.5,-2.25);
			\SETCOORD{2}{3}
			\CUP{1}
			\SETCOORD{1}{0}
			\LINE{0}{-1.5}
			\SETCOORD{1}{0}
			\LINE{0}{1.5}
			\SETCOORD{1}{0}
			\LINE{0}{-1.5}
			\SETCOORD{-3}{0}
			\CAP{-1}
		\end{tikzpicture}
	\end{equation*}
	whereas the right-hand side is given by 
	
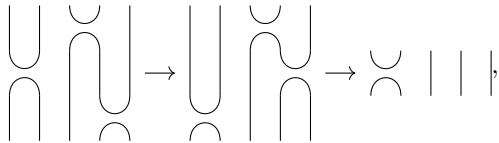
\begin{equation*}
		\begin{tikzpicture}[scale=0.4]
			\LINE{0}{-1.5}
			\CUP{1}
			\LINE{0}{1.5}
			\SETCOORD{1}{0}
			\CUP{1}
			\SETCOORD{1}{0}
			\LINE{0}{-3}
			\CUP{-1}
			\LINE{0}{1.5}
			\CAP{-1}
			\LINE{0}{-3}
			\SETCOORD{-2}{0}
			\LINE{0}{1.5}
			\CAP{1}
			\LINE{0}{-1.5}
			\SETCOORD{2}{0}
			\CAP{1}
			\draw[->] (4.5,-2.25)--(5.5,-2.25);
			\SETCOORD{2}{4.5}
			\LINE{0}{-3}
			\CUP{1}
			\LINE{0}{3}
			\SETCOORD{1}{0}
			\CUP{1}
			\SETCOORD{1}{0}
			\LINE{0}{-1.5}
			\CUP{-1}
			\CAP{-1}
			\LINE{0}{-3}
			\SETCOORD{-2}{0}
			\CAP{1}
			\SETCOORD{2}{0}
			\LINE{0}{1.5}
			\CAP{1}
			\LINE{0}{-1.5}
			\draw[->] (10.5,-2.25)--(11.5,-2.25);
			\SETCOORD{2}{3}
			\CUP{1}
			\SETCOORD{1}{0}
			\LINE{0}{-1.5}
			\SETCOORD{1}{0}
			\LINE{0}{1.5}
			\SETCOORD{1}{0}
			\LINE{0}{-1.5}
			\SETCOORD{-3}{0}
			\CAP{-1}
		\end{tikzpicture},
	\end{equation*}
	and it is easy to see that the first step is an isomorphism and does not influence orientability for the second.

	Now assume that $a= b-1$ (the case $a=b+2$ is handled similarly).
	Now the situation looks like
	
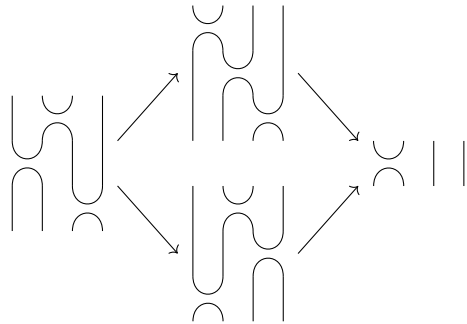
\begin{equation*}
		\begin{tikzpicture}[scale=0.4]
			\LINE{0}{-1.5}
			\CUP{1}
			\CAP{1}
			\LINE{0}{-1.5}
			\CUP{1}
			\LINE{0}{3}
			\SETCOORD{-1}{0}
			\CUP{-1}
			\SETCOORD{-1}{-4.5}
			\LINE{0}{1.5}
			\CAP{1}
			\LINE{0}{-1.5}
			\SETCOORD{1}{0}
			\CAP{1}
			\draw[->] (3.5,-1.5)--(5.5,0.75);
			\draw[->] (3.5,-3)--(5.5,-5.25);
			\SETCOORD{3}{7.5}
			\CUP{1}
			\SETCOORD{1}{0}
			\LINE{0}{-1.5}
			\CUP{-1}
			\CAP{-1}
			\LINE{0}{-3}
			\SETCOORD{1}{0}
			\LINE{0}{1.5}
			\CAP{1}
			\CUP{1}
			\LINE{0}{3}
			\SETCOORD{-1}{-4.5}
			\CAP{1}
			\SETCOORD{-3}{-1.5}
			\LINE{0}{-3}
			\CUP{1}
			\LINE{0}{1.5}
			\CAP{1}
			\CUP{1}
			\LINE{0}{1.5}
			\SETCOORD{-1}{0}
			\CUP{-1}
			\SETCOORD{-1}{-4.5}
			\CAP{1}
			\SETCOORD{1}{0}
			\LINE{0}{1.5}
			\CAP{1}
			\LINE{0}{-1.5}
			\draw[->] (9.5,0.75)--(11.5,-1.5);
			\draw[->] (9.5,-5.25)--(11.5,-3);
			\SETCOORD{3}{6}
			\CUP{1}
			\SETCOORD{1}{0}
			\LINE{0}{-1.5}
			\SETCOORD{1}{0}
			\LINE{0}{1.5}
			\SETCOORD{-2}{-1.5}
			\CAP{-1}
		\end{tikzpicture}
	\end{equation*}
	Now note that the bottom left map is an isomorphism.
	This means that this diagram cannot commute only if the top left map sends something to zero, which is nonzero under the bottom composition.
	By \cref{surgproczero} and \cref{figuremapsarerotationsofsurgproc} this means that we are in one of the following two cases
	\begin{equation*}
		\setlength{\arraycolsep}{50pt}
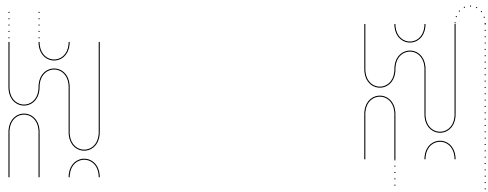
		\begin{array}{lr}
			\begin{tikzpicture}[scale=0.4]
				\LINE[dotted]{0}{-1}
				\SETCOORD{-1}{1}
				\LINE[dotted]{0}{-1}
				\LINE{0}{-1.5}
				\CUP{1}
				\CAP{1}
				\LINE{0}{-1.5}
				\CUP{1}
				\LINE{0}{3}
				\SETCOORD{-1}{0}
				\CUP{-1}
				\SETCOORD{-1}{-4.5}
				\LINE{0}{1.5}
				\CAP{1}
				\LINE{0}{-1.5}
				\SETCOORD{1}{0}
				\CAP{1}
			\end{tikzpicture}&
			\begin{tikzpicture}[scale=0.4]
				\LINE[dotted]{0}{5.5}
				\CAP[dotted]{-1}
				\SETCOORD{-3}{0}
				\LINE{0}{-1.5}
				\CUP{1}
				\CAP{1}
				\LINE{0}{-1.5}
				\CUP{1}
				\LINE{0}{3}
				\SETCOORD{-1}{0}
				\CUP{-1}
				\SETCOORD{-1}{-4.5}
				\LINE{0}{1.5}
				\CAP{1}
				\LINE{0}{-1.5}
				\SETCOORD{1}{0}
				\CAP{1}
				\SETCOORD{-2}{0}
				\LINE[dotted]{0}{-1}
			\end{tikzpicture}
		\end{array}
	\end{equation*}
	where the dotted lines are either joined or end on the same side of $0$.
	Now observe that the first case produces a not orientable diagram in the end.
	Furthermore, for the second case at least one of the dotted lines has to connect to the bottom right cap, as the picture is otherwise not orientable.
	If the right dotted line connects to this cap, then we create a circle in the end, so it is left to consider the following case.
	
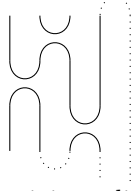
\begin{equation*}
		\begin{tikzpicture}[scale=0.4]
			\LINE[dotted]{0}{5.5}
			\CAP[dotted]{-1}
			\SETCOORD{-3}{0}
			\LINE{0}{-1.5}
			\CUP{1}
			\CAP{1}
			\LINE{0}{-1.5}
			\CUP{1}
			\LINE{0}{3}
			\SETCOORD{-1}{0}
			\CUP{-1}
			\SETCOORD{-1}{-4.5}
			\LINE{0}{1.5}
			\CAP{1}
			\LINE{0}{-1.5}
			\SETCOORD{1}{0}
			\CAP{1}
			\SETCOORD{-2}{0}
			\CUP[dotted]{1}
			\SETCOORD{1}{0}
			\LINE[dotted]{0}{-1}
		\end{tikzpicture}
	\end{equation*}
	But as both dotted endpoints either join or lie on the same side of $0$, this means that the resulting diagram in the end will also not be orientable.
\end{proof}
\begin{lem}\label{phirespectsbraid}
	The functor $\Phi$ respects also \eqref{braid}.
\end{lem}
\begin{proof}
	First suppose that $\abs{a-b}>1$ and $\abs{a-c}>1$.
	This means in terms of generalized crossingless matchings that the cup cap pair corresponding to $a$ does not interact with $b$ and $c$.
	So swapping $a$ with $b$ (resp. $c$) is just given by changing the order of the cup cap pairs, which does not change anything regarding the orientability.
	Therefore, the left-hand side of \eqref{braid} is given by first swapping $b$ and $c$ and then moving $a$ to the top, whereas the right-hand side first moves $a$ to the top and then swapping $b$ and $c$.
	Because $a$ is distant to $b$ and $c$ both give the same result.

	A similar argumentation proves the cases, where $b$ (resp. $c$) is distant to $a$ and $c$ (resp. $a$ and $b$).

	So the only remaining case is that $c=b+1=a+2$.
	In this case the situation looks as follows.
	
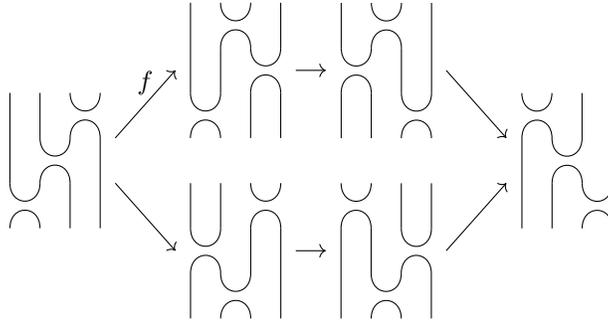
\begin{equation*}
		\begin{tikzpicture}[scale=0.4]
			\CUP{-1}
			\SETCOORD{-1}{0}
			\LINE{0}{-1.5}
			\CUP{1}
			\CAP{1}
			\LINE{0}{-3}
			\SETCOORD{-3}{0}
			\CAP{1}
			\SETCOORD{1}{0}
			\LINE{0}{1.5}
			\CAP{-1}
			\CUP{-1}
			\LINE{0}{3}
			\draw[->] (0.5, -1.5) -- (2.5,0.75) node[midway, anchor = south] {$f$};
			\draw[->] (0.5, -3) -- (2.5,-5.25);
			\SETCOORD{6}{-1.5}
			\CAP{1}
			\SETCOORD{1}{0}
			\LINE{0}{1.5}
			\CAP{1}
			\LINE{0}{-1.5}
			\SETCOORD{0}{4.5}
			\LINE{0}{-1.5}
			\CUP{-1}
			\CAP{-1}
			\LINE{0}{-1.5}
			\CUP{-1}
			\LINE{0}{3}
			\SETCOORD{1}{0}
			\CUP{1}
			\draw[->] (6.5, 0.75) -- (7.5,0.75);
			\SETCOORD{0}{-6}
			\CUP{1}
			\SETCOORD{-2}{0}
			\LINE{0}{-1.5}
			\CUP{-1}
			\LINE{0}{1.5}
			\SETCOORD{0}{-4.5}
			\LINE{0}{1.5}
			\CAP{1}
			\CUP{1}
			\LINE{0}{1.5}
			\CAP{1}
			\LINE{0}{-3}
			\SETCOORD{-2}{0}
			\CAP{1}
			\draw[->] (6.5, -5.25) -- (7.5,-5.25);
			\SETCOORD{3}{0}
			\LINE{0}{3}
			\CAP{1}
			\LINE{0}{-1.5}
			\CUP{1}
			\CAP{1}
			\LINE{0}{-1.5}
			\SETCOORD{-2}{0}
			\CAP{1}
			\SETCOORD{-2}{4.5}
			\CUP{1}
			\SETCOORD{1}{0}
			\LINE{0}{-1.5}
			\CUP{1}
			\LINE{0}{1.5}
			\draw[->] (11.5, -5.25) -- (13.5,-3);
			\SETCOORD{0}{1.5}
			\CAP{-1}
			\SETCOORD{-1}{0}
			\LINE{0}{1.5}
			\CAP{-1}
			\LINE{0}{-1.5}
			\SETCOORD{0}{4.5}
			\LINE{0}{-1.5}
			\CUP{1}
			\CAP{1}
			\LINE{0}{-1.5}
			\CUP{1}
			\LINE{0}{3}
			\SETCOORD{-2}{0}
			\CUP{1}
			\draw[->] (11.5, 0.75) -- (13.5,-1.5);
			\SETCOORD{4}{-3}
			\CUP{1}
			\SETCOORD{1}{0}
			\LINE{0}{-1.5}
			\CUP{-1}
			\CAP{-1}
			\LINE{0}{-3}
			\SETCOORD{3}{0}
			\CAP{-1}
			\SETCOORD{-1}{0}
			\LINE{0}{1.5}
			\CAP{1}
			\CUP{1}
			\LINE{0}{3}
		\end{tikzpicture}
	\end{equation*} 
	These two compositions do not agree only if there is an orientable generalized circle diagram that is mapped to $0$ under one composition and something nonzero via the other one.
	Now the two middle arrows are isomorphisms, thus this situation can only occur if one of the first maps sends such a generalized circle diagram to $0$.
	We may assume that $f$ maps such a generalized circle diagram to $0$ as the other case is obtained by rotational symmetry.

	Using \cref{figuremapsarerotationsofsurgproc} and \cref{surgproczero} we are in one of the following two cases
	\begin{equation*}
		\setlength{\arraycolsep}{50pt}
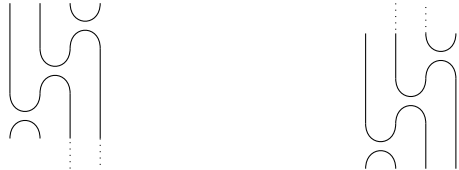
		\begin{array}{lr}
			\begin{tikzpicture}[scale=0.4]
				\CUP{-1}
				\SETCOORD{-1}{0}
				\LINE{0}{-1.5}
				\CUP{1}
				\CAP{1}
				\LINE{0}{-3}
				\LINE[dotted]{0}{-1}
				\SETCOORD{-3}{1}
				\CAP{1}
				\SETCOORD{1}{-1}
				\LINE[dotted]{0}{1}
				\LINE{0}{1.5}
				\CAP{-1}
				\CUP{-1}
				\LINE{0}{3}
			\end{tikzpicture}&
			\begin{tikzpicture}[scale=0.4]
				\CUP{-1}
				\LINE[dotted]{0}{1}
				\SETCOORD{-1}{0}
				\LINE[dotted]{0}{-1}
				\LINE{0}{-1.5}
				\CUP{1}
				\CAP{1}
				\LINE{0}{-3}
				\SETCOORD{-3}{0}
				\CAP{1}
				\SETCOORD{1}{0}
				\LINE{0}{1.5}
				\CAP{-1}
				\CUP{-1}
				\LINE{0}{3}
			\end{tikzpicture}
		\end{array}
	\end{equation*}
	and either the two dotted lines are connected or these are both rays ending on the same side of $0$.

	In the first case note that the diagram in the end is not orientable.
	Thus, either composition produces $0$.

	So we are left to look at the second case.
	We make a case distinction on how the top right endpoint of the cup connects to the rest of the diagram.
	We have essentially four different cases, which are presented in \cref{braidrelationcasedistinction}.
	\begin{figure}
		\setcounter{enumi}{0}
		\renewcommand\theenumi{\roman{enumi}}
		\begin{tabularx}{\textwidth}{lXlXlXll}
			\refstepcounter{enumi}(\theenumi)&
			\begin{tikzpicture}[scale=0.4]
				\CUP[dashed]{1}
				\LINE[dashed]{0}{4.5}
				\CAP[dashed]{-1}
				\CUP{-1}
				\LINE[dotted]{0}{1}
				\SETCOORD{-1}{0}
				\LINE[dotted]{0}{-1}
				\LINE{0}{-1.5}
				\CUP{1}
				\CAP{1}
				\LINE{0}{-3}
				\SETCOORD{-3}{0}
				\CAP{1}
				\SETCOORD{1}{0}
				\LINE{0}{1.5}
				\CAP{-1}
				\CUP{-1}
				\LINE{0}{3}
			\end{tikzpicture}
			&\refstepcounter{enumi}(\theenumi)&
			\begin{tikzpicture}[scale=0.4]
				\CUP[dashed]{1}
				\SETCOORD{-2}{0}
				\CUP[dashed]{3}
				\LINE[dashed]{0}{4.5}
				\CAP[dashed]{-1}
				\CUP{-1}
				\LINE[dotted]{0}{1}
				\SETCOORD{-1}{0}
				\LINE[dotted]{0}{-1}
				\LINE{0}{-1.5}
				\CUP{1}
				\CAP{1}
				\LINE{0}{-3}
				\SETCOORD{-3}{0}
				\CAP{1}
				\SETCOORD{1}{0}
				\LINE{0}{1.5}
				\CAP{-1}
				\CUP{-1}
				\LINE{0}{3}
			\end{tikzpicture}
			&\refstepcounter{enumi}(\theenumi)&
			\begin{tikzpicture}[scale=0.4]
				\CAP[dashed]{-1}
				\LINE[dashed]{0}{-4.5}
				\CUP[dashed]{5}
				\LINE[dashed]{0}{4.5}
				\CAP[dashed]{-1}
				\CUP{-1}
				\LINE[dotted]{0}{1}
				\SETCOORD{-1}{0}
				\LINE[dotted]{0}{-1}
				\LINE{0}{-1.5}
				\CUP{1}
				\CAP{1}
				\LINE{0}{-3}
				\SETCOORD{-3}{0}
				\CAP{1}
				\SETCOORD{1}{0}
				\LINE{0}{1.5}
				\CAP{-1}
				\CUP{-1}
				\LINE{0}{3}
			\end{tikzpicture}
			&\refstepcounter{enumi}(\theenumi)&
			\begin{tikzpicture}[scale=0.4]
				\LINE[dashed]{0}{-1}
				\CUP{-1}
				\LINE[dotted]{0}{1}
				\SETCOORD{-1}{0}
				\LINE[dotted]{0}{-1}
				\LINE{0}{-1.5}
				\CUP{1}
				\CAP{1}
				\LINE{0}{-3}
				\SETCOORD{-3}{0}
				\CAP{1}
				\SETCOORD{1}{0}
				\LINE{0}{1.5}
				\CAP{-1}
				\CUP{-1}
				\LINE{0}{3}
			\end{tikzpicture}
		\end{tabularx}
		\caption{Case distinction for the proof of the braid relation}\label{braidrelationcasedistinction}
	\end{figure}
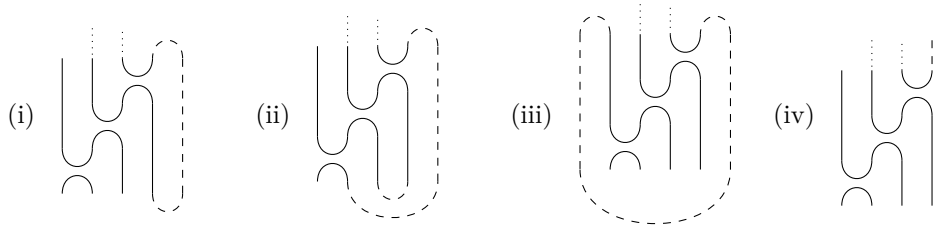
	\begin{enumerate}[leftmargin=*]
		\item This diagram is not orientable by assumption on the dotted lines.
		\item The small dashed cup forms a circle in the resulting diagram, thus both compositions have to be $0$.
		\item This diagram is not orientable as the two dotted lines will necessarily be connected.
		\item In this case we make a case distinction to where the top left endpoint connects.
		We distinguish another three cases.
		\begin{center}
			\setcounter{enumi}{0}
			\renewcommand\theenumi{\alph{enumi}}
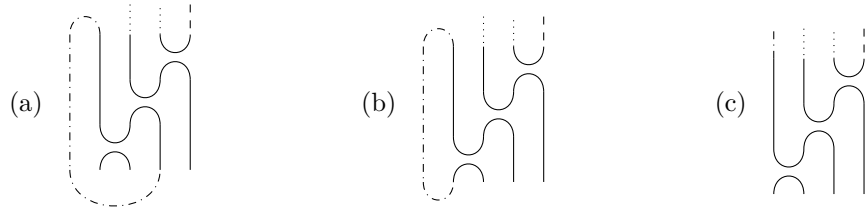
			\begin{tabularx}{\linewidth}{lXlXll}	
				\refstepcounter{enumi}(\theenumi)&
				\begin{tikzpicture}[scale=0.4]
					\LINE[dashed]{0}{-1}
					\CUP{-1}
					\LINE[dotted]{0}{1}
					\SETCOORD{-1}{0}
					\LINE[dotted]{0}{-1}
					\LINE{0}{-1.5}
					\CUP{1}
					\CAP{1}
					\LINE{0}{-3}
					\SETCOORD{-3}{0}
					\CAP{1}
					\SETCOORD{1}{0}
					\LINE{0}{1.5}
					\CAP{-1}
					\CUP{-1}
					\LINE{0}{3}
					\CAP[dashdotted]{-1}
					\LINE[dashdotted]{0}{-4.5}
					\CUP[dashdotted]{3}
				\end{tikzpicture}	
				&\refstepcounter{enumi}(\theenumi)&
				\begin{tikzpicture}[scale=0.4]
					\LINE[dashed]{0}{-1}
					\CUP{-1}
					\LINE[dotted]{0}{1}
					\SETCOORD{-1}{0}
					\LINE[dotted]{0}{-1}
					\LINE{0}{-1.5}
					\CUP{1}
					\CAP{1}
					\LINE{0}{-3}
					\SETCOORD{-3}{0}
					\CAP{1}
					\SETCOORD{1}{0}
					\LINE{0}{1.5}
					\CAP{-1}
					\CUP{-1}
					\LINE{0}{3}
					\CAP[dashdotted]{-1}
					\LINE[dashdotted]{0}{-4.5}
					\CUP[dashdotted]{1}
				\end{tikzpicture}
				&\refstepcounter{enumi}(\theenumi)&
				\begin{tikzpicture}[scale=0.4]
					\LINE[dashed]{0}{-1}
					\CUP{-1}
					\LINE[dotted]{0}{1}
					\SETCOORD{-1}{0}
					\LINE[dotted]{0}{-1}
					\LINE{0}{-1.5}
					\CUP{1}
					\CAP{1}
					\LINE{0}{-3}
					\SETCOORD{-3}{0}
					\CAP{1}
					\SETCOORD{1}{0}
					\LINE{0}{1.5}
					\CAP{-1}
					\CUP{-1}
					\LINE{0}{3}
					\LINE[dashdotted]{0}{1}
				\end{tikzpicture}
			\end{tabularx}
		\end{center}
		\begin{enumerate}[label=(\alph*), leftmargin=*]
			\item This diagram is not orientable.
			\item In this case the end result would look like
			
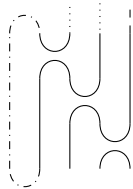
\begin{equation*}
				\begin{tikzpicture}[scale=0.4]
					\CUP{1}
					\LINE[dotted]{0}{1}
					\SETCOORD{1}{0}
					\LINE[dotted]{0}{-1}
					\LINE{0}{-1.5}
					\CUP{-1}
					\CAP{-1}
					\LINE{0}{-3}
					\CUP[dashdotted]{-1}
					\LINE[dashdotted]{0}{4.5}
					\CAP[dashdotted]{1}
					\SETCOORD{3}{-4.5}
					\CAP{-1}
					\SETCOORD{-1}{0}
					\LINE{0}{1.5}
					\CAP{1}
					\CUP{1}
					\LINE{0}{3}
					\LINE[dashed]{0}{1}
				\end{tikzpicture}
			\end{equation*}
			which is not orientable by the assumption on the dotted lines.	
			\item Note that any configuration of cups below this diagram produce a non-orientable diagram along the bottom composition.
			So in these cases, the compositions would both give $0$.
			And now we distinguish whether the dotted lines are connected or not.
			First assume that they are connected.
			
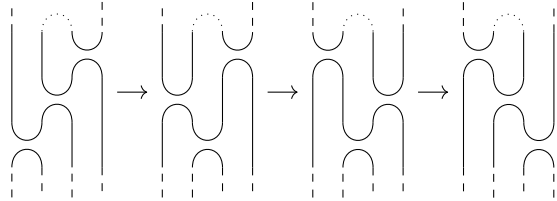
\begin{equation*}
				\begin{tikzpicture}[scale=0.4]
					\LINE[dashed]{0}{-1}
					\CUP{-1}
					\CAP[dotted]{-1}
					\LINE{0}{-1.5}
					\CUP{1}
					\CAP{1}
					\LINE{0}{-3}\LINE[dashed]{0}{-1}
					\SETCOORD{-3}{0}
					\LINE[dashed]{0}{1}
					\CAP{1}
					\LINE[dashed]{0}{-1}
					\SETCOORD{1}{0}
					\LINE[dashed]{0}{1}
					\LINE{0}{1.5}
					\CAP{-1}
					\CUP{-1}
					\LINE{0}{3}
					\LINE[dashed]{0}{1}
					\draw[->] (0.5, -3) -- (1.5,-3);
					\SETCOORD{8}{0}
					\LINE[dashed]{0}{-1}
					\CUP{-1}
					\CAP[dotted]{-1}
					\LINE{0}{-1.5}
					\CUP{-1}
					\LINE{0}{1.5}
					\LINE[dashed]{0}{1}
					\SETCOORD{0}{-6.5}
					\LINE[dashed]{0}{1}
					\LINE{0}{1.5}
					\CAP{1}
					\CUP{1}
					\LINE{0}{1.5}
					\CAP{1}
					\LINE{0}{-3}
					\LINE[dashed]{0}{-1}
					\SETCOORD{-2}{0}
					\LINE[dashed]{0}{1}
					\CAP{1}
					\LINE[dashed]{0}{-1}
					\draw[->] (5.5, -3) -- (6.5,-3);
					\SETCOORD{3}{0}
					\LINE[dashed]{0}{1}
					\LINE{0}{3}
					\CAP{1}
					\LINE{0}{-1.5}
					\CUP{1}
					\CAP{1}
					\LINE{0}{-1.5}
					\LINE[dashed]{0}{-1}
					\SETCOORD{-2}{0}
					\LINE[dashed]{0}{1}
					\CAP{1}
					\LINE[dashed]{0}{-1}
					\SETCOORD{-2}{6.5}
					\LINE[dashed]{0}{-1}
					\CUP{1}
					\CAP[dotted]{1}
					\LINE{0}{-1.5}
					\CUP{1}
					\LINE{0}{1.5}
					\LINE[dashed]{0}{1}
					\draw[->] (10.5, -3) -- (11.5,-3);
					\SETCOORD{2}{0}
					\LINE[dashed]{0}{-1}
					\CUP{1}
					\CAP[dotted]{1}
					\LINE{0}{-1.5}
					\CUP{-1}
					\CAP{-1}
					\LINE{0}{-3}
					\LINE[dashed]{0}{-1}
					\SETCOORD{3}{0}
					\LINE[dashed]{0}{1}
					\CAP{-1}
					\LINE[dashed]{0}{-1}
					\SETCOORD{-1}{0}
					\LINE[dashed]{0}{1}
					\LINE{0}{1.5}
					\CAP{1}
					\CUP{1}
					\LINE{0}{3}
					\LINE[dashed]{0}{1}
				\end{tikzpicture}
			\end{equation*} 
			And now we have to show that not all of these picture can be orientable.
			Observe that for every two consecutive lines there is one diagram, where these are joined by a line.
			Assume there exists a configuration such that all four diagrams are orientable.
			There has to be a half where at least three strands end.
			But this means that two of these strands have to lie on the same side of $0$.
			So there is a diagram that has a non-propagating line, where both endpoints lie on the same side of $0$, in contradiction to orientability.

			The last case is that we only have rays at the top and bottom.
			
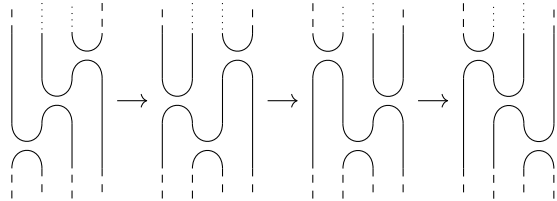
\begin{equation*}
				\begin{tikzpicture}[scale=0.4]
					\LINE[dashed]{0}{-1}
					\CUP{-1}
					\LINE[dotted]{0}{1}
					\SETCOORD{-1}{0}
					\LINE[dotted]{0}{-1}
					\LINE{0}{-1.5}
					\CUP{1}
					\CAP{1}
					\LINE{0}{-3}\LINE[dashed]{0}{-1}
					\SETCOORD{-3}{0}
					\LINE[dashed]{0}{1}
					\CAP{1}
					\LINE[dashed]{0}{-1}
					\SETCOORD{1}{0}
					\LINE[dashed]{0}{1}
					\LINE{0}{1.5}
					\CAP{-1}
					\CUP{-1}
					\LINE{0}{3}
					\LINE[dashed]{0}{1}
					\draw[->] (0.5, -3.25) -- (1.5,-3.25);
					\SETCOORD{7}{0}
					\LINE[dotted]{0}{-1}
					\CUP{1}
					\LINE[dashed]{0}{1}
					\SETCOORD{-2}{0}
					\LINE[dotted]{0}{-1}
					\LINE{0}{-1.5}
					\CUP{-1}
					\LINE{0}{1.5}
					\LINE[dashed]{0}{1}
					\SETCOORD{0}{-6.5}
					\LINE[dashed]{0}{1}
					\LINE{0}{1.5}
					\CAP{1}
					\CUP{1}
					\LINE{0}{1.5}
					\CAP{1}
					\LINE{0}{-3}
					\LINE[dashed]{0}{-1}
					\SETCOORD{-2}{0}
					\LINE[dashed]{0}{1}
					\CAP{1}
					\LINE[dashed]{0}{-1}
					\draw[->] (5.5, -3.25) -- (6.5,-3.25);
					\SETCOORD{3}{0}
					\LINE[dashed]{0}{1}
					\LINE{0}{3}
					\CAP{1}
					\LINE{0}{-1.5}
					\CUP{1}
					\CAP{1}
					\LINE{0}{-1.5}
					\LINE[dashed]{0}{-1}
					\SETCOORD{-2}{0}
					\LINE[dashed]{0}{1}
					\CAP{1}
					\LINE[dashed]{0}{-1}
					\SETCOORD{-2}{6.5}
					\LINE[dashed]{0}{-1}
					\CUP{1}
					\LINE[dotted]{0}{1}
					\SETCOORD{1}{0}
					\LINE[dotted]{0}{-1}
					\LINE{0}{-1.5}
					\CUP{1}
					\LINE{0}{1.5}
					\LINE[dashed]{0}{1}
					\draw[->] (10.5, -3.25) -- (11.5,-3.25);
					\SETCOORD{2}{0}
					\LINE[dashed]{0}{-1}
					\CUP{1}
					\LINE[dotted]{0}{1}
					\SETCOORD{1}{0}
					\LINE[dotted]{0}{-1}
					\LINE{0}{-1.5}
					\CUP{-1}
					\CAP{-1}
					\LINE{0}{-3}
					\LINE[dashed]{0}{-1}
					\SETCOORD{3}{0}
					\LINE[dashed]{0}{1}
					\CAP{-1}
					\LINE[dashed]{0}{-1}
					\SETCOORD{-1}{0}
					\LINE[dashed]{0}{1}
					\LINE{0}{1.5}
					\CAP{1}
					\CUP{1}
					\LINE{0}{3}
					\LINE[dashed]{0}{1}
				\end{tikzpicture}
			\end{equation*} 
			Again we have to show that there is no possible configuration that all these pictures are orientable.
			Similar to before, there cannot be three of the bottom strands ending on the same half.
			Also, the four top rays cannot end on the same half using the same reasoning.
			But this means that the left two bottom rays end for example at the bottom and the right two bottom ones at the top.
			Also observe that there is a diagram such that each of these pairs are connected.
			Furthermore, there is a diagram such that the left two (resp.\ right two) top strands are connected.
			This means that the bottom pairs cannot be joined by two more rays and so there cannot be a configuration such that all four diagrams are orientable.
			So both compositions give $0$.\qedhere
		\end{enumerate}
	\end{enumerate}
\end{proof}
\bibliography{biblio}
\end{document}